\documentclass{article}
\usepackage[utf8]{inputenc}
\usepackage[T1]{fontenc} 
\usepackage{amsmath}
\usepackage{amsfonts}
\usepackage{amssymb}
\usepackage{amsthm}
\usepackage{color}
\usepackage{esint}
\usepackage{enumerate}
\usepackage{dsfont}
\usepackage{stmaryrd}
\usepackage{hyperref}
\usepackage{graphicx}
\usepackage{graphics}
\usepackage{epstopdf}
\usepackage{epsfig}
\usepackage{setspace}
\usepackage{comment} 
\usepackage{tikz}
\usetikzlibrary{positioning,arrows}
\usetikzlibrary{calc,arrows}
\usetikzlibrary{er,positioning,bayesnet}

\graphicspath{{figures_review/}}
\usepackage{float}
\usepackage{appendix}
\newtheorem{theorem}{Theorem}
\newtheorem{lemma}[theorem]{Lemma}
\newtheorem{remark}[theorem]{Remark}

\newtheorem{corollary}[theorem]{Corollary}

\newcommand{\RR}{\mathbb{R}}

\newcommand{\NN}{\mathbb{N}}

\newcommand{\dis}{\displaystyle}
\newcommand{\dive}{\operatorname{div}}

\newcommand{\bA}{\boldsymbol{A}}
\newcommand{\bX}{\boldsymbol{X}}
\newcommand{\bb}{\boldsymbol{b}}
\newcommand{\bn}{\boldsymbol{n}}
\newcommand{\bc}{\boldsymbol{c}}
\newcommand{\bx}{\boldsymbol{x}}
\newcommand{\bv}{\boldsymbol{v}}
\newcommand{\bz}{\boldsymbol{z}}
\newcommand{\bs}{\boldsymbol{s}}
\newcommand{\bw}{\boldsymbol{w}}
\newcommand{\bq}{\boldsymbol{q}}
\newcommand{\bQ}{\boldsymbol{Q}}
\newcommand{\bp}{\boldsymbol{p}}
\newcommand{\bg}{\boldsymbol{g}}
\newcommand{\nab}{\boldsymbol{\nabla}}
\newcommand{\Aa}{\mathbb{A}}
\newcommand{\DD}{\mathbb{D}}
\newcommand{\MM}{\mathbb{M}}
\newcommand{\mT}{\mathcal{T}}

\newcommand{\intO}{\int_{\Omega}}
\newcommand{\rott}{\operatorname{\bf curl}}
\newcommand{\vertiii}[1]{{\left\vert\kern-0.25ex\left\vert\kern-0.25ex\left\vert #1 \right\vert\kern-0.25ex\right\vert\kern-0.25ex\right\vert}}

\begin{document}

\date{\today}
\author{Ludovic Chamoin$^{1,2}$ and Frédéric Legoll$^{3,4}$
\\
{\footnotesize $^1$ Universit\'e Paris-Saclay, ENS Paris-Saclay, CNRS, LMT,} \\ {\footnotesize 4 avenue des Sciences, 91190 Gif-sur-Yvette, France}
\\  
{\footnotesize $^2$ Institut Universitaire de France (IUF), 1 rue Descartes, 75231 Paris Cedex 5, France}
\\
{\footnotesize $^3$ Ecole Nationale des Ponts et Chauss\'ees, Laboratoire Navier,} \\ {\footnotesize 6-8 avenue Blaise Pascal, 77455 Marne-La-Vall\'ee Cedex 2, France}
\\  
{\footnotesize $^4$ Inria Paris, MATHERIALS project-team, 2 rue Simone Iff,} \\ {\footnotesize CS 42112, 75589 Paris Cedex 12, France}
\\
{\footnotesize \tt ludovic.chamoin@ens-paris-saclay.fr, frederic.legoll@enpc.fr}\\
}
\title{A pedagogical review on \textit{a posteriori} error estimation in Finite Element computations}
%\runningheads{L. Chamoin, F. Legoll}{Review on error estimation for linear elliptic problems}

\maketitle

%\doublespacing

\begin{abstract}
This article is a review on basic concepts and tools devoted to \textit{a posteriori} error estimation for problems solved with the Finite Element Method. For the sake of simplicity and clarity, we mostly focus on linear elliptic diffusion problems, approximated by a conforming numerical discretization. The review mainly aims at presenting in a unified manner a large set of powerful verification methods, around the concept of equilibrium. Methods based on that concept provide error bounds that are fully computable and mathematically certified. We discuss recovery methods, residual methods, and duality-based methods for the estimation of the whole solution error (i.e. the error in energy norm), as well as goal-oriented error estimation (to assess the error on specific quantities of interest). We briefly survey the possible extensions to non-conforming numerical methods, as well as more complex (e.g. nonlinear or time-dependent) problems. We also provide some illustrating numerical examples on a linear elasticity problem in 3D.
\end{abstract}

\medskip


%\fl{changer le layout latex pour qqch de plus neutre}

\tableofcontents

\section{Introduction}\label{section:introduction}

A large number of physical phenomena are described by partial differential equations. Since it is usually not possible to obtain closed form expressions for the exact solution of these equations, numerical methods with mathematically-based algorithms and discretization techniques are often used as simulation tools. Such methods typically only deliver an approximate solution (somehow the best function within a predefined finite-dimensional space) that is different from the exact solution. Consequently, in order to certify the accuracy of numerical simulations, two important questions are
\begin{itemize}
\item[(i)] How large is the overall discretization error between the exact and the approximate solution?
\item[(ii)] Where in the spatial domain (and time domain for time-dependent problems) is this error localized?
\end{itemize}
Answering in an appropriate manner to these two questions may be crucial in many engineering activities since a decision is often taken on the basis of numerical simulation results. Taking the argument one step further, a long-term goal in scientific computing is to design algorithms such that (i) a given precision is reached at the end of the simulation, and (ii) the computational work to achieve this accuracy is as small as possible.

In this article, we consider these questions for a particular numerical method, a conforming Finite Element Method (FEM), and mainly for a particular class of problems, namely linear elliptic diffusion problems. We already note that the approaches we discuss here can be extended to other mathematical problems and to other types of discretization. We review here the various \textit{a posteriori} error estimation methods (and the associated adaptive strategies) which have been developed over the last years, enlightening specific properties and links between the methods. We particularly highlight that methods providing for fully computable and guaranteed error bounds can be unified around the tools of dual analysis and the concept of equilibrium.

\medskip

This review article is organized as follows. The reference problem along with some elements of \textit{a priori} error analysis is given in Section~\ref{section:refpb}. The \textit{a posteriori} error estimation methods are next presented in three groups: flux recovery methods in Section~\ref{section:recovery}, residual methods in Section~\ref{section:residuals}, and duality-based constitutive relation error methods in Section~\ref{section:CRE}. Numerical illustrations of \textit{a posteriori} error estimation on a three-dimensional elasticity problem are shown in Section~\ref{section:illustration}. A unified perspective on these different methods is provided in Section~\ref{section:connexions}. Mesh adaptation is next discussed in Section~\ref{section:adaptivity}. Goal-oriented error estimation, where the aim is to certify the error on a given quantity of interest rather than on the whole solution itself, is addressed in Section~\ref{section:goaloriented} (with, in particular, a numerical illustration in Section~\ref{sec:goal_num_illus}). Extensions to other Finite Element schemes (i.e. beyond classical conforming FEM) and to other mathematical problems are briefly discussed in Sections~\ref{section:extensions1} and~\ref{section:extensions2}, respectively. Concluding remarks are collected in Section~\ref{section:conclusion}. We eventually provide in Appendix~\ref{section:appendixA} a proof of the fact that the Constitutive Relation Error estimator presented in Section~\ref{section:CRE}, and which is always an upper bound of the numerical error, is also, up to a multiplicative constant independent of the mesh size, a lower bound of the numerical error (see Theorem~\ref{theo:borneinf} and Corollary~\ref{coro:erc_autre} there).

\medskip

This review is written at an elementary level. It describes basic features of the approaches, presents them in a unified manner, and purposedly skips some technicalities (for which we prefer to refer to the relevant bibliography). We hope it will be useful for researchers and engineers looking for an up-to-date overview of the field, both in terms of theoretical and implementation aspects.

% ---------------------------------------------------------------------
% ---------------------------------------------------------------------

\section{Reference problem and notations}\label{section:refpb}

\subsection{Reference problem}\label{section:pbreffort}

Throughout this article, except in Section~\ref{section:extensions2}, we consider the following problem:
\begin{equation}\label{eq:refpbstrong}
-\dive (\Aa \nab u) = f \ \ \text{in $\Omega$}, \qquad u=0 \ \ \text{on $\Gamma_D$}, \qquad (\Aa \nab u) \cdot \bn = g \ \ \text{on $\Gamma_N$},
\end{equation}
where $\Omega$ is an open bounded subset of $\RR^d$ with Lipschitz boundary $\partial \Omega$, and $\Gamma_D$ and $\Gamma_N$ are parts of $\partial \Omega$ such that $\overline{\Gamma_D \cup \Gamma_N} = \partial \Omega$, $\Gamma_D \cap \Gamma_N = \emptyset$ and $|\Gamma_D| \neq 0$. We assume that $f \in L^2(\Omega)$ and that $\Aa \in [L^\infty(\Omega)]^{d\times d}$ is a symmetric matrix (we refer to Section~\ref{sec:adv-diff} for the study of some non-symmetric operators), which is uniformly bounded and positive in the sense that there exists $a_{\rm max} \geq a_{\rm min} > 0$ such that
\begin{equation}
\label{eq:elliptic}
\forall \boldsymbol{\xi} \in \RR^d, \quad a_{\rm min} |\boldsymbol{\xi}|^2 \le \Aa(x) \boldsymbol{\xi} \cdot \boldsymbol{\xi} \le a_{\rm max} |\boldsymbol{\xi}|^2 \quad \text{a.e. in $\Omega$}.
\end{equation}
The quantity $\bq=\Aa\nab u$ is the flux associated with $u$.

\medskip

The norm and semi-norm of a function $v$ in $H^k(\Omega)$ are denoted 
$$
\| v \|_k = \left( \intO \sum_{|\alpha| \le k} |D^\alpha v|^2 \right)^{1/2}, \qquad |v|_k = \left( \intO \sum_{|\alpha| = k} |D^\alpha v|^2 \right)^{1/2},
$$
where $D^\alpha v$ is the $\alpha$-th order derivative of $v$ (with $\alpha \in \NN^d$). Note that the norm of $v \in L^2(\Omega)$ is denoted $\| v \|_0$. When the norm is restricted over a subdomain $K \subset \Omega$, the subdomain is explicit in the notation. Hence, for instance, $\| v \|_{0,K}$ denotes the $L^2(K)$ norm of $v$.

\medskip

Considering the Hilbert space $V = \{v \in H^1(\Omega), \ \ v=0 \ \text{on $\Gamma_D$} \}$ endowed with the $H^1$ norm $\| \cdot \|_1$, we recall that the weak formulation of~\eqref{eq:refpbstrong} is:
\begin{equation}\label{eq:refpb}
\text{Find $u \in V$ such that, for any $v \in V$,} \qquad B(u,v)=F(v),
\end{equation}
where
$$
B(u,v) =\intO \Aa \nab u \cdot \nab v, \qquad F(v)=\intO f \, v + \int_{\Gamma_N} g \, v.
$$
The well-posedness of~\eqref{eq:refpb} of course directly follows from the Lax-Milgram theorem. The bilinear form $B$ is symmetric, continuous and coercive on $V$. It hence defines an inner product and induces the energy norm $\vertiii{v}=\sqrt{B(v,v)}$ on $V$, which is equivalent to $\| v \|_1$ on $V$:
\begin{equation}
\label{eq:equivalence}
\forall v \in V, \qquad \sqrt{\frac{a_{\rm min}}{1+C_{\Omega}^2}} \ \| v \|_1 \le \vertiii{v} \le \sqrt{a_{\rm max}} \ \| v \|_1,
\end{equation}
where $C_{\Omega}$ is the Poincar\'e constant of $\Omega$, which satisfies $\| v \|_0 \le C_{\Omega} \ |v|_1$ for any $v \in V$. We see that
$$
\forall v,w \in V, \quad | B(v,w) | \leq \vertiii{v} \ \vertiii{w}.
$$
In the sequel, we also use the notation $\dis \vertiii{\widetilde{\bq}}_q = \sqrt{\intO \Aa^{-1} \widetilde{\bq} \cdot \widetilde{\bq}}$ for any vector-valued field $\widetilde{\bq} \in (L^2(\Omega))^d$. 

\subsection{Finite Element (FE) approximation}

Before we proceed, we mention that we will assume throughout our text that the reader is reasonably familiar with finite element methods. We refer to the classical textbooks~\cite{BAT96,BER04,BRE08,CAN06,ERN04,QUA94} (see also introductory expositions in~\cite{BON14}, and (in French) in~\cite{JOL90,LUC97,RAP99,RAV83}).

\smallskip

Let $\mT_h$ be a partition of $\Omega$. We denote by $h_K$ the diameter of each element $K \in \mT_h$ (i.e. the largest distance between any two points in $K$) and by $\rho_K$ the diameter of the largest circle (or sphere) contained in $K$. We set $\dis h=\max_{K \in \mT_h}h_K$. The mesh is assumed to be regular (non-degenerate), in the sense that there exists $\gamma_0 >0$, independent of $h$, such that, for any $K \in \mT_h$, we have $1 \leq h_K/\rho_K \leq \gamma_0$. Let $U(K)$ denote the neighborhood of the element~$K$:
\begin{equation}
\label{eq:neighbor}
U(K)=\text{int}\left\{\cup \overline{J}, \quad J \in \mT_h, \quad \overline{J} \cap \overline{K} \neq \emptyset \right\}.
\end{equation}
We introduce the space $V_h^p$ of continuous and locally supported functions which are polynomials of degree up to $p$ on each element $K$. Note that $V_h^p \subset V$. We are thus considering here a conforming Finite Element Method (FEM). Non-conforming FEM will be considered in Section~\ref{section:extensions1}. The FE approximation of~\eqref{eq:refpb} is
\begin{equation}\label{eq:FEpb}
\text{Find $u_h \in V_h^p$ such that, for any $v \in V_h^p$,} \qquad B(u_h,v)=F(v),
\end{equation}
which is a well-posed problem, again in view of the Lax-Milgram theorem. Let
$$
R(v) = F(v)-B(u_h,v)
$$
be the so-called residual and 
\begin{equation*} %\label{eq:def_R_star}
\| R \|_\star
=
\sup_{v \in V, \, v\neq 0} \frac{|R(v)|}{\vertiii{v}}
\end{equation*}
be the dual norm of the residual. The discretization error of the approach is $e=u-u_h \in V$.
It satisfies the following three properties:
\begin{equation}\label{eq:errorprop}
\begin{aligned}
\forall v \in V, \quad B(e,v) &= R(v) \quad \text{(residual equation)}, 
\\
\forall v \in V_h^p, \quad B(e,v) &= 0 \quad \text{(Galerkin orthogonality)},
\\
\vertiii{e} &= \| R \|_\star.
\end{aligned}
\end{equation}

\begin{remark}
In this article, we do not consider errors other than those arising from discretization. In particular, we do not consider any geometry error (the partition $\mT_h$ is assumed to exactly coincide with $\Omega$) and any quadrature error (integrals over any element $K$ are assumed to be exactly computed). In the same spirit, we note that the problem~\eqref{eq:FEpb} may lead to a large linear system, which is then solved using iterative solvers. We ignore here the error coming from the use of such iterative solvers and refer to~\cite{HAB21,JIR10,MAL20} for its assessment.
\end{remark}

A direct consequence of the Galerkin orthogonality is the best approximation property:
\begin{equation}\label{eq:bestapprox}
\forall v_h \in V_h^p, \quad \vertiii{e} = \sqrt{\vertiii{u-v_h}^2 - \vertiii{v_h-u_h}^2} \le \vertiii{u-v_h}.
\end{equation}
Using~\eqref{eq:equivalence}, this implies the C\'ea's lemma:
$$
\forall v_h \in V_h^p, \quad \| e \|_1 \leq \sqrt{\frac{(1+C_{\Omega}^2) \, a_{\rm max}}{a_{\rm min}}} \ \| u-v_h \|_1.
$$

\subsection{Notion of error estimate}

The discretization error $e$ has two components in a given subdomain of $\Omega$: one that is locally generated and one that is transported from elsewhere (the so-called pollution error~\cite{BAB87}). Its evaluation $\vertiii{e}$ in the energy norm provides a global measure of the overall quality of the FE solution. The quantity $\vertiii{e}$ may be evaluated using error estimation methods, which can be classified in two groups:
\begin{itemize}
\item[(i)] \textit{a priori} error estimation, which can be performed before the approximate solution $u_h$ is known;
\item[(ii)] \textit{a posteriori} error estimation, which is obtained after $u_h$ is computed and therefore uses information from $u_h$ to estimate $\vertiii{e}$. The goal of \textit{a posteriori} error estimates is not only to offer a criterion that indicates whether a prescribed accuracy is met (i.e. whether $\vertiii{e}$ is smaller than some threshold), but also to give local error indicators which can be used to drive an adaptive mesh refinement strategy.
\end{itemize}
Focusing on the second group of methods, the global error estimate $\eta$ usually satisfies $\dis \eta^2=\sum_{K\in \mT_h}\eta_K^2$, where $\eta_K$ is a local error estimate associated with each element $K$. A crucial property which is demanded to \textit{a posteriori} error estimators is the equivalence between $\eta$ and the energy norm of the exact error, i.e. that there exist positive constants $C_1$ and $C_2$ independent of $h$ such that
$$
C_1 \vertiii{e} \leq \eta \leq C_2 \vertiii{e}.
$$
The quality of the estimator is assessed by the global effectivity index $i_{\rm eff}=\eta/\vertiii{e}$. The estimator is guaranteed (resp. relevant) if $i_{\rm eff} \geq 1$ (resp. $i_{\rm eff}\approx 1$). Local effectivity indices may also be defined.

\medskip

One can then formulate properties describing an acceptable error estimate:
\begin{itemize}
\item \textbf{reliability}: ensure that $\vertiii{e} \leq \eta$ holds, i.e. $i_{\rm eff} \ge 1$, and that $\eta$ is fully computable in order for $\eta$ to serve as a stopping criterion;
\item \textbf{accuracy}: ensure that the predicted error is close to the actual (unknown) error, i.e. $i_{\rm eff} \approx 1$; 
\item \textbf{local effectivity}: ensure that $\eta_K \leq C \, \vertiii{e}_{U(K)}$ for each element $K$, where we recall that $U(K)$ denotes the neighborhood of the element $K$ (see~\eqref{eq:neighbor}). This property is important for adaptive mesh refinement;
\item \textbf{asymptotic exactness}: ensure that $\dis \lim_{h \to 0} i_{\rm eff} = 1$;
\item \textbf{robustness}: guarantee the previous properties independently of the regularity of the solution and of the mesh, and independently of variations of the problem parameters;
\item \textbf{practicality}: provide an estimate $\eta$ and contributions $\eta_K$ which can be evaluated locally, with a small computational cost. Indeed, if this evaluation turns out to require a global computation with large resources, it may be just cheaper to solve the reference problem on a uniformly refined mesh (even though an estimate of the error is not available).
\end{itemize}

\subsection{\textit{A priori} error estimation}\label{section:apriori}

\textit{A priori} estimation of errors in numerical methods has long been an enterprise of numerical analysts~\cite{AZZ72,STR73,CIA78,BER00}. It allows to bound the error before any numerical solution is computed, and therefore gives rough information on the asymptotic convergence and stability of the numerical method that is used. It is not designed to provide a computable error estimate for a given mesh. The general approach for \textit{a priori} error estimation uses approximation theory or truncation error analysis~\cite{CHE66}. It can be briefly presented as follows.

For any integers $l \in [1,p+1]$ and $m \le l$ (recall that $p$ is the maximal polynomial degree in $V_h^p$), and for any $K \in \mT_h$ and any $v \in H^l(U(K))$, the Cl\'ement interpolant $\Pi_hv \in V_h^p$ of $v$ (see~\cite{CLE75} and the textbook~\cite[Lemma 1.127 and Remark 1.129]{ERN04}) satisfies
\begin{equation}\label{clementineq}
\| v- \Pi_h v \|_{m,K} \le C h_K^{l-m} \ \| v \|_{l,U(K)},
\end{equation}
where $C$ is independent of $h_K$, $K \in \mT_h$ and $v \in H^l(U(K))$, but depends on $p$ and $\gamma_0$ (the regularity parameter of the mesh). We refer to~\cite{SCO90} and the textbooks~\cite{VER96}, \cite[Sec. IX.3]{BER04} and~\cite[Sec. 4.8]{BRE08} for related estimations. We deduce from~\eqref{clementineq} that, when $p \geq 1$, we have, for any $K \in \mT_h$ and any $v \in H^1(U(K))$,
\begin{equation}\label{eq:Clem2a}
\| v-\Pi_hv\|_{0,K} \le C \, h_K \, \| v\|_{1,U(K)},
\end{equation}
where $C$ is independent of $v$ and $h_K$. This property will be useful in the sequel, as well as the following edge estimate (see~\cite[Lemma 1.127 and Remark 1.129]{ERN04}). Let $\Gamma_{\rm int}$ denote the union of the internal edges (or, if $d \geq 3$, interfaces) of the mesh. For any edge $\Gamma \subset \Gamma_{\rm int}\cup\Gamma_N$, let $U(\Gamma)$ denote the set of elements in $\mT_h$ sharing at least one vertex with $\Gamma$. Then, for any $v \in H^1(U(\Gamma))$, we have
\begin{equation}\label{eq:Clem2b}
\| v-\Pi_hv\|_{0,\Gamma} \le C \, l_\Gamma^{1/2} \, \| v \|_{1,U(\Gamma)},
\end{equation}
where $l_\Gamma$ is the diameter of $\Gamma$ (when $d=2$, $l_\Gamma$ is simply the edge length) and where $C$ is independent of $v$, $l_\Gamma$ and $\Gamma$. 

\medskip

The \textit{a priori} error estimation follows from~\eqref{clementineq}. We assume that there exists an integer $\alpha \in [1,p]$ such that $\| u \|_{\alpha+1}<\infty$. When $\Omega$ is a (convex) domain without re-entrant corners, $\Aa$ is sufficiently regular, $f \in L^2(\Omega)$ and when the Neumann boundary data $g$ is sufficiently smooth on $\Gamma_N$ and compatible with the homogeneous Dirichlet boundary condition on $\Gamma_D$, then the solution to~\eqref{eq:refpbstrong} belongs to $H^2(\Omega)$ and we thus can take $\alpha=1$ (see e.g.~\cite[Theorem 3.12]{ERN04}).
We then infer from the best approximation property~\eqref{eq:bestapprox} (taking $v_h$ equal to the Cl\'ement interpolant $\Pi_hu \in V_h^p$ of $u$) and from the bound~\eqref{clementineq} (with $m=1$ and $l=\alpha+1$) that
\begin{equation}
\label{eq:taux_alpha}
\vertiii{u-u_h} 
=
\vertiii{e} \le \vertiii{u-\Pi_hu} \le C \, h^\alpha \, \|u\|_{\alpha+1},
\end{equation}
where the constant $C$ is independent of $h$ and $u$ (throughout the article, the constant $C$ may change from one line to the next; when valid, the independence of that constant with respect to the mesh size and other quantities will always be underlined). When $u$ is sufficiently smooth (that is when $\| u \|_{p+1}<\infty$), we have the optimal rate of convergence $\vertiii{e} \le C \, h^p \| u \|_{p+1}$.

\medskip

The main drawback of this \textit{a priori} bound is that the right-hand side of~\eqref{eq:taux_alpha} is not computable in practice (since it is defined in terms of the unknown exact solution $u$), and difficult to accurately estimate. In addition, when estimating this right-hand side is possible, it is usually observed that the above bound highly overestimates the exact error.

\begin{remark}
\label{rem:AN}
A direct consequence of~\eqref{eq:taux_alpha} is that the $L^2$-norm of the error satisfies $\| e \|_0 \leq C \, h^\alpha \|u\|_{\alpha+1}$, for some $C$ independent of $h$ and $u$. However, this bound is not sharp. Sharp bounds can be obtained using the Aubin-Nitsche lemma, which is based on the following adjoint problem:
\begin{equation}\label{eq:AN}
\text{Find $\varphi \in V$ such that, for any $v \in V$,} \quad B^\ast(\varphi,v)=\intO (u-u_h)v,
\end{equation}
where $B^\ast(w,v)=B(v,w)$ for any $v$ and $w$ in $V$. In our case, $B$ is symmetric, hence $B^\ast=B$. 

We follow~\cite[Sec. 5.4]{BRE08} or~\cite[Sec. X.1]{BER04}. Assume that $\Aa$ and $\Omega$ are such that the solution $\varphi$ to~\eqref{eq:AN} belongs to $H^2(\Omega)$ and satisfies
\begin{equation}\label{eq:inter}
\| \varphi \|_2 \leq C \, \| u-u_h \|_0.
\end{equation}
A sufficient condition is that $\Aa$ is smooth and $\Omega$ is convex. Taking $v=u-u_h$ in~\eqref{eq:AN}, we have, for any $\varphi_h \in V^p_h$,
$$
\| u-u_h \|^2_0
=
B(u-u_h,\varphi)
=
B(u-u_h,\varphi-\varphi_h)
\leq
C \| u-u_h \|_1 \ \| \varphi-\varphi_h \|_1.
$$
Using~\eqref{clementineq} to bound the last factor and next~\eqref{eq:inter}, we deduce that
$$
\| u-u_h \|^2_0
\leq
C h \| u-u_h \|_1 \ \| \varphi \|_2
\leq
C h \| u-u_h \|_1 \ \| u-u_h \|_0
$$
and hence
\begin{equation}\label{eq:AN_utile}
\|u-u_h\|_0 \le C h \| u-u_h\|_1.
\end{equation}
In view of~\eqref{eq:taux_alpha}, we eventually obtain
$$
\|u-u_h\|_0 \le C h^{\alpha+1} \|u\|_{\alpha+1},
$$
where the constant $C$ is independent of $h$ and $u$. The convergence rate is now better than what we had at the beginning of the remark, simply using~\eqref{eq:taux_alpha}.
\end{remark}

\begin{remark}
It is interesting at this point to draw a parallel between the properties of the Cl\'ement interpolant, or of finite element approximations that we mentioned above, and a classical result on the polynomial interpolation of some function $f \in C^{p+1}$. From the values of $f$ at $p+1$ sampling points $x_i$ ($i = 0,1,\dots,p$), we define the Lagrange interpolant as the polynomial $P$ of degree $p$ such that $P(x_i)= f(x_i)$ for $i=0,1,\dots,p$. The interpolation error reads
$$
R(x) = f(x)-P(x) = \frac{f^{(p+1)}(\nu_x)}{(p+1)!}\prod_{i=0}^p (x-x_i) \quad \text{with $\nu_x \in [x_0,x_p]$},
$$
and it vanishes at sample points. For a uniform distribution of points, that is when $x_{i+1}-x_i = h$ for any $i=0,1,\dots,p-1$, the previous expression yields
$$
|R(x)| \le \frac{|f^{(p+1)}(\nu_x)|}{4(p+1)} \, h^{p+1}.
$$
We thus again observe a convergence rate of the order $h^{p+1}$ when using polynomials of degree $p$ to approximate sufficiently smooth functions $f$.
\end{remark}

To illustrate the \textit{a priori} error estimate, we consider the simple problem shown on Figure~\ref{fig:laplace}, that is the Laplace equation $-\Delta u = 0$ in a square domain (of unit size) with prescribed boundary conditions: a Neumann condition $\nab u \cdot \bn = 1$ on the top side of the square, and homogeneous Dirichlet conditions $u=0$ on the other sides. The expression of the exact solution of this problem (which is smooth) is available, so that the discretization error in the energy norm $\vertiii{e}$ when computing a finite element approximation can be evaluated exactly. This error is shown in Figure~\ref{fig:laplace} for several values of the uniform mesh size $h$, and when considering various finite element types. The asymptotic evolutions confirm the predictions given by \textit{a priori} error estimation.

\begin{figure}[H]
\begin{center}
\includegraphics[width=55mm]{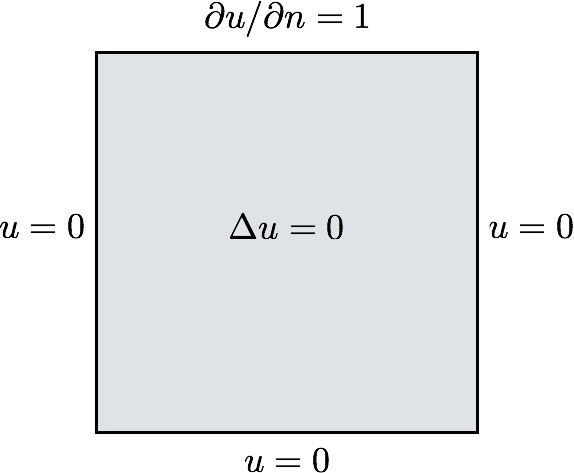} \hspace{2.em}
\includegraphics[width=55mm]{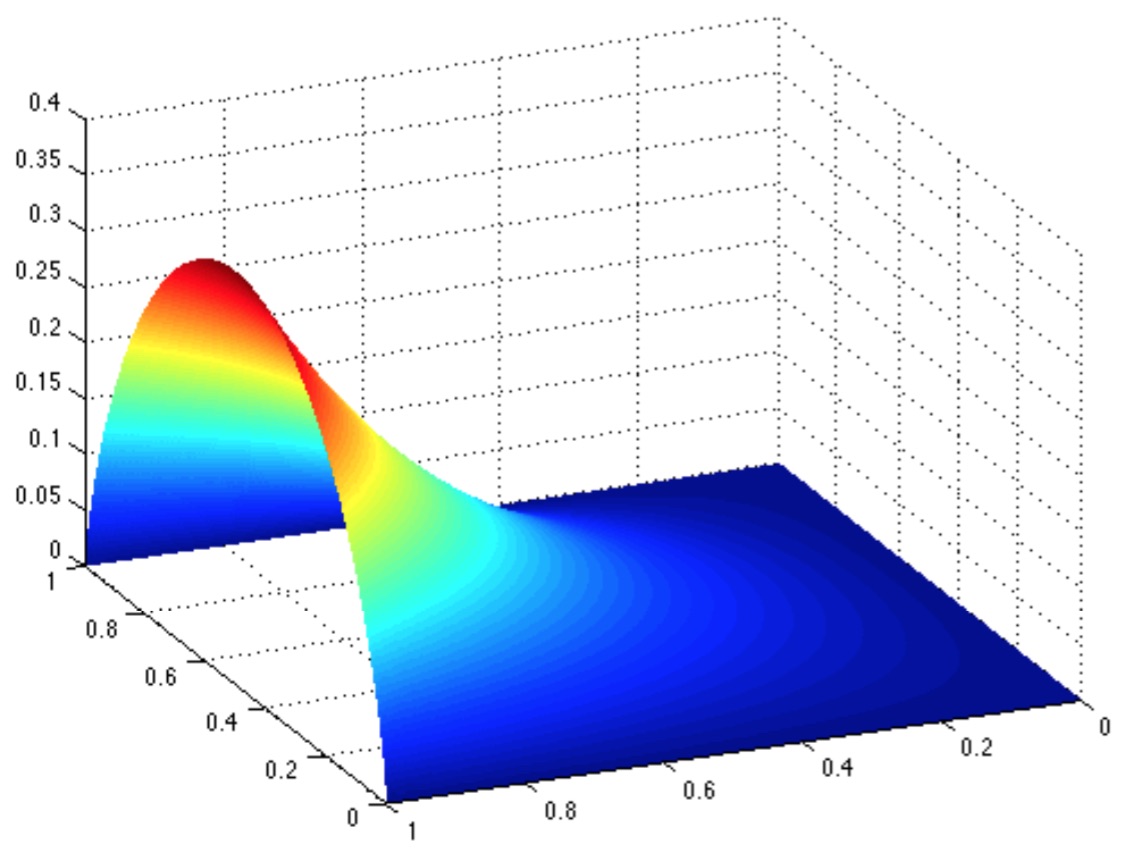} \hspace{2.em}
\\
\includegraphics[width=70mm]{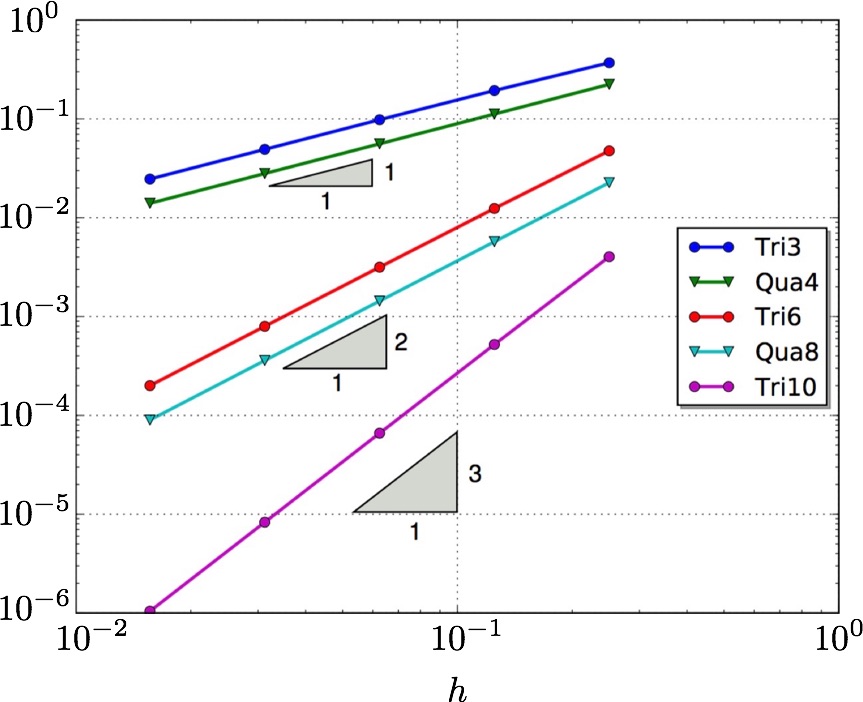}
\end{center}
\caption{Considered problem (left) on a unit square with analytical solution $u$ (center), and evolution of $\vertiii{e}$ as a function of $h$ (right) for various discretization spaces. The error when using P1 (resp. Q1) finite elements (Tri3, resp. Qua4) decays at the rate $O(h)$. When using P2 (resp. quadratic serendipity) finite elements (Tri6, resp. Qua8), it decays at the rate $O(h^2)$. When using P3 finite elements (Tri10), it decays at the rate $O(h^3)$. \label{fig:laplace}}
\end{figure}

\medskip

In contrast to \textit{a priori} error estimates, \textit{a posteriori} error estimates use information from the numerical solution to estimate the error. Furthermore, these estimates allow to adapt the approximation spaces (and the related meshes) to efficiently approximate the solution to the problem. In the following sections, we review three families of \textit{a posteriori} error estimation methods: flux recovery methods (in Section~\ref{section:recovery}), residual methods (in Section~\ref{section:residuals}) and duality-based methods (in Section~\ref{section:CRE}). Such methods allow to get additional information on the discretization error, compared to the asymptotic convergence rate predicted by \textit{a priori} error estimates. We then present in Section~\ref{section:connexions} a unified perspective that encompasses all these methods. We refer to~\cite{STR92} for a numerical comparison of some of these methods.

% ---------------------------------------------------------------------
% ---------------------------------------------------------------------

\section{Recovery methods}\label{section:recovery}

This category of error estimators is based on using approximations of the exact solution which are more accurate than the numerical solution.

\subsection{Richardson extrapolation}\label{section:Richardson}

A well-known technique in this category is the Richardson extrapolation, in which approximate solutions are obtained on sequences of refined nested meshes (or with shape functions of increasing order) and compared one to each other to obtain an interpolated indication of the error. Consider a refined mesh $\mT_{h^\ast}$ obtained by subdivision of $\mT_h$ (we thus have $V_h \subset V_{h^\ast}$), and let $u_{h^\ast} \in V_{h^\ast}$ be the approximation of $u$ computed on this refined mesh. Assuming that we are in the asymptotic range predicted by the \textit{a priori} error estimate~\eqref{eq:taux_alpha}, we write $\vertiii{u-u_{h^\ast}} \approx (h^\ast/h)^\alpha\vertiii{e}$. Using the Galerkin orthogonality property~\eqref{eq:errorprop}, we obtain that
$$
\vertiii{e} = \left( \vertiii{u-u_{h^\ast}}^2 + \vertiii{u_{h^\ast}-u_h}^2 \right)^{1/2} \approx [1-(h^\ast/h)^{2\alpha}]^{-1/2} \ \vertiii{u_{h^\ast}-u_h}.
$$
Having at our disposal $u_h$ and $u_{h^\ast}$, we are thus in position to compute the value of the right-hand side, which provides an approximation of $\vertiii{e}$.

\subsection{Basics on flux recovery methods}

Flux recovery methods, also referred to as flux-projection techniques, are alternative methods which are very popular in the engineering FE community. They were initially proposed in~\cite{ZIE87,ZIE92a,ZIE92b}, and are known under the name of Zienkiewicz--Zhu (ZZ) estimators. The idea is to post-process the approximate flux $\bq_h=\Aa \nab u_h$ in order to get a smoother (continuous) representation $\bq^\ast_h$ of the exact flux $\bq=\Aa\nab u$, which is {\em expected} to be more accurate than $\bq_h$. One next uses the estimate (obtained by substituting the unknown flux $\bq$ by the recovered flux $\bq^\ast_h$)
\begin{equation}\label{eq:ZZestimate}
\vertiii{e} = \vertiii{\bq-\bq_h}_q \approx \vertiii{\bq^\ast_h-\bq_h}_q,
\end{equation}
where the notation $\dis \vertiii{\cdot}_q$ is defined in Section~\ref{section:pbreffort}. The quantity $\vertiii{\bq-\bq^\ast_h}_q$ is hence thought to be much smaller than the error $\vertiii{\bq-\bq_h}_q$. 

The quality of the estimate $\eta = \vertiii{\bq^\ast_h-\bq_h}_q$ depends on that of the smoothing, i.e. on how well $\bq^\ast_h$ approximates $\bq$ (see~\cite{AIN89}). Assuming that there exists $c \in (0,1)$ such that
$$
\vertiii{\bq-\bq^\ast_h}_q \le c \vertiii{\bq-\bq_h}_q,
$$
we directly get, using the triangle inequality, that
$$
\frac{\eta}{1+c} \le \vertiii{e}\le \frac{\eta}{1-c}.
$$
We thus see that, the smaller $c$ is, the better the \textit{a posteriori} estimate $\eta$ is.

Several smoothing techniques may be used to build the field $\bq^\ast_h$, which is usually chosen in $\left( V_h^p \right)^d$:
\begin{itemize}
\item Nodal averaging: the value of $\bq^\ast_h$ at each nodal point is defined as an average of the values of $\bq_h$ at the neighboring Gauss points (which is hence inexpensive to compute);
\item Global $L^2$ projection of $\bq_h$ onto the FE space $\left( V_h^p \right)^d$;
\item Local projection on element patches, which is the so-called Superconvergent Patch Recovery (SPR or ZZ2) technique developed in~\cite{ZHA95,ZIE92a,ZIE92b}. In this more advanced approach, superconvergence properties of the FEM (which occur in very specific cases, see e.g.~\cite{BAB96,WHE87}) are taken into account with a local fitting procedure using polynomials of degree $p$. This technique, which is computationally efficient since it only requires to solve small systems, consists of two steps.

First, over each patch $\Omega_i$ of elements sharing the vertex node $i$, an intermediate recovered flux $\dis \bg_i(\bx)=\sum_n \boldsymbol{\alpha}_n^i \, \phi_n(\bx)$ (where $\phi_n$ are predefined polynomial functions of degree $p$) is constructed using values of $\bq_h$ at sampling points and performing a least squares fitting between $\bg_i$ and $\bq_h$ to determine the coefficients $\left\{ \boldsymbol{\alpha}_n^i \right\}_n$. In other words, a polynomial surface (of the same degree as the FE interpolation) is fitted to the FE flux values at sampling points. These points are (when they exist) superconvergence points $\bx_s$ where $|\bq(\bx_s)-\bq_h(\bx_s)|\le C(u)h^{p+1}$. They are for instance element centroids for linear FE ($p=1$), or Gauss-Legendre quadrature points.

Second, the recovered flux $\bq^\ast_h$, which is chosen as a function in $\left( V_h^p \right)^d$, is defined at each node as the average of the values at this node of the different functions $\bg_i$ (note that a node may belong to various patches $\Omega_i$, and thus several functions $\bg_i$ may be non-zero at that node). We next interpolate these nodal values using FE shape functions in order to get a continuous flux representation $\bq^\ast_h \in \left( V_h^p \right)^d$. Under superconvergence properties, one can show that the estimate $\eta=\vertiii{\bq^\ast_h-\bq_h}_q$ obtained in this way is asymptotically exact, i.e. $\dis \lim_{h \to 0} \frac{\eta}{\vertiii{e}} = 1$.
\end{itemize}

\begin{remark}
Instead of fitting a polynomial of degree $p$ to flux values at some sampling points as in the ZZ2 technique, the fitting of a polynomial of degree $p+1$ to solution values at nodal points was investigated in~\cite{ZHA05}. The approximation $\bq^\ast_h$ is next obtained after taking derivatives.
\end{remark}

\subsection{Recent advances in flux recovery methods}\label{section:advancedZZ}

A weakness of the general flux recovery approach (e.g. basic SPR technique) is that it does not use the fact that $u_h$ solves~\eqref{eq:FEpb}. For instance, it does not use information from Neumann boundaries where imposed tractions are known \textit{a priori}, which may lead to a  lack of accuracy of the recovered flux field along such boundaries. Also, the strategy is hazardous in case of singularities (cracks, multi-materials, \dots) or anisotropic meshes~\cite{CAI17,CAI09}, and may be not locally conservative, in the following sense. The estimated global error $\eta$ can be written as a sum of local contributions, but these local contributions may highly under- or overestimate the true local error. In contrast, the approaches presented in Sections~\ref{section:residuals} and~\ref{section:CRE} below provide estimations of the local errors that have been observed, in practice, to be accurate even in difficult cases. 

\medskip

As a consequence of their weaknesses, several improvements to smoothing techniques have been introduced later on in order to insert more physical insight in the procedure, and improve the accuracy and robustness of the flux recovery. In particular, the idea of locally enforcing the equilibrium equation was introduced in~\cite{BLA94,BOR97,LEE97,ROD07,WIB97,WIB94}. Equilibrium constraints are imposed on each patch by means of penalization (Recovery by Equilibrium in Patches (REP) method) or Lagrange multiplier fields (SPR-C method). They can also be directly inserted in the search space, which then consists of locally equilibrated fields, associated with the solution to some minimization problems over patches~\cite{KVA98,UBE04}. They allow to enhance the quality of the flux field that is recovered over each patch and provide for a locally (but not globally) equilibrated flux field $\bq^\ast_h$, i.e. a field that locally satisfies the equation $-\dive \bq^\ast_h = f$.

The SPR-C method is associated with a continuous fitting approach, and Lagrange multipliers are used to enforce internal and boundary equilibrium equations over each patch. To obtain a globally continuous flux field, a partition of unity is used to properly weight the flux polynomials constructed at the patch level.

\medskip

An upper bound on the error can next be obtained, as we now show, using the enhanced recovered flux field $\bq^\ast_h$ and an additional correction term to account for the fact that this flux field does not {\em globally} satisfy the equilibrium equation (see also~\cite{DIE07}, where an alternative bound to the one given below is proposed). We first recall that
\begin{equation}
\label{eq:def_Hdiv}
H(\dive,\Omega) = \left\{ \bp \in [L^2(\Omega)]^d, \quad \dive \bp \in L^2(\Omega) \right\}
\end{equation}
is an Hilbert space when endowed with the scalar product $\dis \langle \bp, \bq \rangle_{H(\dive,\Omega)} = \intO \bp \cdot \bq +(\dive \bp) (\dive \bq)$ and that, for any $\bp \in H(\dive,\Omega)$, the trace of $\bp\cdot \bn$ on $\partial \Omega$ is well-defined. For any $v \in V$ and for any flux $\bp \in H(\dive,\Omega)$, we then write
\begin{multline} \label{eq:lundi}
B(e,v) = \intO \Aa \nab (u-u_h)\cdot \nab v = \intO (\bq-\bq_h)\cdot \nab v \\ = \intO(\bp-\bq_h)\cdot \nab v + \intO(f + \dive \bp) \, v+\int_{\Gamma_N} (g-\bp\cdot \bn) \, v.
\end{multline}
Taking $\bp=\bq^\ast_h$ and $v = e$, and assuming that the lack of equilibrium $g-\bq^\ast_h\cdot \bn$ along the Neumann boundary is negligible, we get
\begin{multline}\label{eq:boundZZ}
\vertiii{e}^2 =  \intO(\bq^\ast_h-\bq_h)\cdot \nab e + \intO (f+\dive \bq^\ast_h) \, e 
\\ \quad \Longrightarrow \quad
\vertiii{e} \le \vertiii{\bq^\ast_h-\bq_h}_q + C h \| f + \dive \bq^\ast_h \|_0
\end{multline}
where we used the Cauchy-Schwarz inequality and the Aubin-Nitsche lemma (see Remark~\ref{rem:AN}) to bound $\|e\|_0$ by $C h \vertiii{e}$.

We observe that the upper bound involves two parts: the first one is the classical ZZ estimate~\eqref{eq:ZZestimate}, and the second one takes into account the lack of equilibrium of the recovered flux field $\bq^\ast_h$. A numerical procedure based on the Richardson extrapolation (see Section~\ref{section:Richardson}) may be employed to evaluate the constant $C$. Indeed, considering $h^\ast \ll h$ and noticing that
$$
\vertiii{e}^2 \approx \frac{\vertiii{u_{h^\ast}-u_h}^2}{1-(h^\ast/h)^{2\alpha}} \qquad \text{and} \qquad \|e\|_0^2 \approx \frac{\|u_{h^\ast}-u_h\|_0^2}{1-(h^\ast/h)^{2\alpha+2}},
$$
we get $\dis C \approx \sqrt{\frac{\|u_{h^\ast}-u_h\|^2_0 \, (1-(h^\ast/h)^{2\alpha})}{h^2 \, \vertiii{u_{h^\ast}-u_h}^2 \, (1-(h^\ast/h)^{2\alpha+2})}}$.

\begin{remark}
When a solution field $u^\ast_h$ is also recovered from $u_h$, in addition to the locally equilibrated recovered flux $\bq^\ast_h$, a lower bound on the error in the energy norm can be obtained. Starting again from~\eqref{eq:lundi} with $\bp=\bq^\ast_h$ and $v = u^\ast_h-u_h$, we get, using the Cauchy-Schwarz inequality, that
\begin{equation}\label{eq:LBZZ}
\vertiii{e} \ge \frac{\intO(\bq^\ast_h-\bq_h)\cdot \nab(u^\ast_h-u_h) + \intO (f + \dive \bq^\ast_h) \, (u^\ast_h-u_h)}{\vertiii{u^\ast_h-u_h}}.
\end{equation}
\end{remark}

%------------------------------------------------------------------------
%------------------------------------------------------------------------

\section{Residual methods}\label{section:residuals}

The idea of the residual methods is to use the residual equation~\eqref{eq:errorprop} to indirectly bound the error $\vertiii{e}$ or the dual norm $\| R \|_\star$ of the residual (in the so-called explicit approaches, described in Section~\ref{sec:resu_exp}), or to approximate its solution $e$ (in the so-called implicit approaches, described in Section~\ref{sec:resu_imp}). The residual actually contains three types of terms: (i) interior residuals that determine how well the approximate solution $u_h$ satisfies the mathematical model in each element, (ii) boundary residuals associated to non-verification of the boundary conditions, and (iii) inter-element residuals related to discontinuities in normal fluxes along the edges of the mesh elements. Explicit residual methods differ from implicit residual methods in the way these three terms are handled.

\subsection{Explicit methods}\label{sec:resu_exp}

The explicit residual method introduced in~\cite{BAB78,KEL83} employs information available from the FE solution $u_h$ along with residuals to directly compute the error estimate. There is no need to solve any additional boundary value problem, in contrast to the implicit methods presented in Section~\ref{sec:implicit} below. In the explicit methods, which share similarities with \textit{a priori} error estimation~\cite{STE98}, the three types of residual terms (exhibited after decomposing the weak residual functional) are post-processed separately and then lumped together.

We assume in the sequel that $\Aa$ is sufficiently smooth in each element of the mesh (typically, $\Aa \in [H^1(K)]^{d\times d}$ for any $K \in \mT_h$). Recalling that
\begin{equation}
\label{eq:c_est_le_sup}
\vertiii{e} = \sup_{v \in V, \ v\neq 0}\frac{|R(v)|}{\vertiii{v}} \quad \text{with} \quad R(v) = F(v)-B(u_h,v) = \langle R_{|u_h}, v \rangle_{V',V},
\end{equation}
where $R_{|u_h}$ is defined as an element of the dual space of $V$, the task is to find computable (local) estimates for the $H^{-1}$-norm of $R_{|u_h}$. We start from the residual equation~\eqref{eq:errorprop} and write
\begin{align}
  \forall v \in V, \quad B(e,v)
  &= F(v)-B(u_h,v)
  \nonumber
  \\
  &= F(v-\Pi_hv)-B(u_h,v-\Pi_hv) \qquad \text{[since $u_h$ solves~\eqref{eq:FEpb}]}
  \nonumber
  \\
  &= \intO f(v-\Pi_hv) + \int_{\Gamma_N} g(v-\Pi_hv) - \intO \Aa\nab u_h \cdot \nab (v-\Pi_hv)
  \nonumber
  \\
  &= \sum_{K\in \mT_h} \int_K r_K \, (v-\Pi_hv) -\sum_{\Gamma \subset \Gamma_{\rm int}\cup\Gamma_N} \int_\Gamma t_\Gamma \, (v-\Pi_hv),
\label{eq:ressplit}
\end{align}
where $\Pi_hv \in V_h^p$ is the Cl\'ement interpolant of $v$, and where $\Gamma_{\rm int}$ denotes the union of the internal edges. The quantities $r_K$ and $t_\Gamma$, which are defined by
\begin{equation}\label{eq:def_rK_tGamma}
\hspace{-1mm} \begin{array}{rcl}  
  r_K \! &=& \! f +\dive \left[ (\Aa\nab u_h)_{|K} \right],
  \\ \noalign{\vskip 3pt}
  t_\Gamma \! \! &=& 
\! \! \left\{
\begin{array}{l}
\Aa\nab u_h \cdot \bn-g \ \ \text{if $\Gamma \subset \Gamma_N$}, \\
\Aa\nab u_{h|K_1}\cdot \bn_1 + \Aa\nab u_{h|K_2}\cdot \bn_2 \ \ \text{if $\Gamma = (\partial K_1 \cap \partial K_2) \subset \Gamma_{\rm int}$},
\end{array}
\right.
\end{array}
\end{equation}
are called the internal and boundary residuals, respectively ($\bn_1$ and $\bn_2$ are the outward normal vectors to $\partial K_1$ and $\partial K_2$, respectively). They represent the two error sources and are computed from both the approximate solution $u_h$ and input data. Using the Cauchy-Schwarz inequality as well as~\eqref{eq:Clem2a} and~\eqref{eq:Clem2b}, we deduce from~\eqref{eq:ressplit} that, for any $v\in V$,
\begin{align*}
  & |B(e,v)|
  \\
  &\le
  \sum_{K \in \mT_h} \| r_K\|_{0,K} \, \| v-\Pi_hv\|_{0,K} + \sum_{\Gamma \subset \Gamma_{\rm int}\cup\Gamma_N} \| t_\Gamma \|_{0,\Gamma} \, \| v-\Pi_hv\|_{0,\Gamma}
  \\
  &\le
  C\left[\sum_{K \in \mT_h} h_K^2\| r_K\|^2_{0,K}+\sum_{\Gamma \subset \Gamma_{\rm int}\cup\Gamma_N} l_\Gamma \| t_\Gamma \|^2_{0,\Gamma} \right]^{1/2} \left[\sum_{K \in \mT_h} \| v\|^2_{1,U(K)} + \sum_{\Gamma \subset \Gamma_{\rm int}\cup\Gamma_N} \|v\|^2_{1,U(\Gamma)} \right]^{1/2}.
\end{align*}
We notice that $\dis \sum_{K \in \mT_h} \|v\|^2_{1,U(K)} + \sum_{\Gamma \subset \Gamma_{\rm int}\cup\Gamma_N} \|v\|^2_{1,U(\Gamma)} \le \widehat{C} \, \vertiii{v}^2$ where $\widehat{C}$ is independent of $h$ and $v \in V$. Taking $v=e$ in the above estimates, we deduce that
\begin{equation}\label{eq:boundexpres}
\vertiii{e} 
=
\frac{B(e,e)}{\vertiii{e}} 
\le 
C' \left[ \sum_{K \in \mT_h} h_K^2 \|r_K\|^2_{0,K} + \sum_{\Gamma \subset \Gamma_{\rm int} \cup \Gamma_N} l_\Gamma \| t_\Gamma \|^2_{0,\Gamma} \right]^{1/2}.
\end{equation}
Note that the constant $C'$, which only depends on the elements shapes and on $a_{\rm min}$ and $a_{\rm max}$ of~\eqref{eq:elliptic}, is usually unknown. It may be estimated but bounds are dictated by worst case scenarios and usually give very pessimistic estimators~\cite{JOH92}.

\medskip

The bound~\eqref{eq:boundexpres} is the sum of both element-wise and edge-wise contributions. Element-wise indicators of the local contribution to the bound for $\vertiii{e}$ may be defined under the form
\begin{equation}\label{eq:defthetaK}
\eta_K^2 = h_K^2\|r_K\|^2_{0,K} + \sum_{\Gamma \subset \partial K}\beta_\Gamma \, l_\Gamma \, \| t_\Gamma \|^2_{0,\Gamma} \ \ \text{with} \ \ \beta_\Gamma = \left\{
\begin{array}{l}
0 \quad \text{if $\Gamma \subset \Gamma_D$} \\
1/2\quad \text{if $\Gamma \subset \Gamma_{\rm int}$} \\
1\quad \text{if $\Gamma \subset \Gamma_N$}
\end{array}
\right. .
\end{equation}
Note that, for edges $\Gamma \subset \Gamma_{\rm int}$, other splittings of the boundary residual $l_\Gamma \, \| t_\Gamma \|^2_{0,\Gamma}$ can be chosen. The local contributions~\eqref{eq:defthetaK}, which are very cheap to compute, can serve to drive the mesh adaptivity, although they are not an estimate of the error in $K$ (including local and transported components). Likewise, the upper bound~\eqref{eq:boundexpres} cannot be used as a global error estimate since it is not computable in practice due to constants that are, except in very specific cases~\cite{GER12,KEL83}, unknown.

\begin{remark}
Since the constants $C$ in~\eqref{eq:Clem2a} and~\eqref{eq:Clem2b} are usually different, the weighting of internal and boundary residual contributions used in the definition~\eqref{eq:defthetaK} of $\eta_K$ is not justified. However, the correct relative weighting to attach to each type of contribution is far from obvious. In~\cite{CAR00}, explicit estimates for Cl\'ement's interpolation operator (and therefore for the constants $C$ in~\eqref{eq:Clem2a} and~\eqref{eq:Clem2b}) are given in specific cases. The bound~\eqref{eq:boundexpres} is written as
\begin{equation*}
\vertiii{e} \le c_1 \left( \sum_{K \in \mT_h} h_K^2 \|r_K\|^2_{0,K} \right)^{1/2} + c_2 \left( \sum_{\Gamma \subset \Gamma_{\rm int} \cup \Gamma_N} l_\Gamma \| t_\Gamma \|^2_{0,\Gamma} \right)^{1/2} 
\end{equation*}
with explicit expressions of $c_1$ and $c_2$. Nevertheless, the crude estimation usually leads to very pessimistic bounds even for regular meshes, particularly due to the fact that possible cancellations between residual types are lost when one deals with each type separately.
\end{remark}

\begin{remark}
The Aubin-Nitsche lemma (duality argument) can again be used to get \textit{a posteriori} error bounds in the $L^2$-norm. We recall (see~\eqref{eq:AN_utile}) that $\|u-u_h\|_0 \leq C h \| u-u_h \|_1$. We are then in position to use the bound~\eqref{eq:boundexpres} on $\| u-u_h \|_1$ to obtain an \textit{a posteriori} error bound on $\|u-u_h\|_0$. 
\end{remark}

\subsection{Implicit methods}\label{sec:resu_imp}
\label{sec:implicit}

Implicit residual methods, that we now describe, are more difficult and intrusive to implement in scientific computation softwares compared to explicit residual methods. They yet have the potential to provide more detailed and more robust information on the error and its sources. These methods avoid the difficult evaluation of some constants such as the one appearing in~\eqref{eq:boundexpres} by seeking an approximation of the solution $e$ of the residual equation. By treating together the three contributions in the residual, implicit residual methods retain more of the structure of the residual equation than explicit methods do, and thus provide tighter error bounds.

\subsubsection{Hierarchical approach}

The residual equation~\eqref{eq:errorprop} is a global equation which cannot be solved exactly (it is as complex as the reference problem). It requires a refined mesh in order to provide for a nontrivial FE approximate solution, a procedure which is usually not feasible. Defining the hierarchical complement space $V^{\rm comp}$ such that $V=V_h^p \oplus V^{\rm comp}$ (where the decomposition is orthogonal in the sense of the symmetric bilinear form $B$), a first possibility (called \textit{multilevel error estimation}) is to enlarge the Galerkin subspace $V_h^p$ by $V^{\rm comp}_h \subset V^{\rm comp}$ leading to a new subspace $\widetilde{V_h} = V^p_h \oplus V^{\rm comp}_h$ supposed to accurately approximate the solution $e$. By Galerkin orthogonality (see~\eqref{eq:errorprop}), we know that $e \in V^{\rm comp}$. For a simple implementation, $V^{\rm comp}_h$ may be chosen as the function space spanned by so-called \textit{bubble functions} with disjoint supports~\cite{BAN93}.

We then search an approximation $\widetilde{e}_h \in V^{\rm comp}_h$ of $e$ such that
$$
\forall v \in V^{\rm comp}_h, \qquad B(\widetilde{e}_h,v)=R(v) 
$$
and define $\eta=\vertiii{\widetilde{e}_h}$ as an error estimator. This estimator is reliable and effective if the saturation assumption is valid, namely if $e-\widetilde{e}_h$ is much smaller than $e$. More precisely, letting $\widetilde{u}_h$ be the approximation of $u$ in $\widetilde{V_h}$, it is easy to see that $\widetilde{e}_h = \widetilde{u}_h-u_h$, hence $e-\widetilde{e}_h = u-\widetilde{u}_h$. The saturation assumption hence holds whenever $u-\widetilde{u}_h$ is much smaller than $u-u_h$. This can be quantified in a manner similar to that of Section~\ref{section:recovery}, as follows: assuming that there exists $c \in (0,1)$ such that
$$
\vertiii{u-\widetilde{u}_h} = \vertiii{e-\widetilde{e}_h} \le c \, \vertiii{u-u_h} = c \, \vertiii{e},
$$
we directly get, using the triangle inequality and the Galerkin orthogonality, that
$$
\frac{\eta}{1+c} \le \vertiii{e}\le \frac{\eta}{\sqrt{1-c^2}}.
$$

As an alternative, the residual equation may be decomposed into a series of \textit{decoupled local} boundary value problems, and local approximations of $e$ are then searched. Such methods can be classified in different categories, depending on:
\begin{enumerate}
\item the small domain in which the local problem is posed: over each element in the mesh $\mT_h$ (element residual methods) or over patches of elements (subdomain residual methods); 
\item the boundary conditions imposed on the local problems. When Dirichlet boundary conditions are used, one obtains continuous approximations of $e$ and lower bounds on $\vertiii{e}$. In contrast, using Neumann boundary conditions allows to derive equilibrated flux fields and upper bounds on $\vertiii{e}$;
\item the numerical method used to approximate the solution of the local problems. In practice, a standard FE method on a finer mesh (or using higher-degree polynomial functions) is usually employed producing asymptotic estimates which have bounding properties only with respect to a reference numerical solution (more precisely, the quantities that can be computed in practice provide rigorous bounds on the error $u_h^\star - u_h$ rather than on $u - u_h$, where $u_h^\star$ is the approximation of $u$ obtained by using the finer mesh, resp. the higher-degree approximation space). Another option is to use a dual approach yielding a direct approximation of the local flux field and allowing for the computation of guaranteed upper bounds on the error. This approach will be further discussed in Section~\ref{section:CRE} below.
\end{enumerate}
We first consider subdomain residual methods in Section~\ref{section:locpbpatches} before turning to element residual methods in Section~\ref{section:elementresidualmeth}.

\subsubsection{Subdomain residual methods}\label{section:locpbpatches}

We describe here three subdomain implicit residual methods, the latter two being very close one to each other. 

A first subdomain implicit residual method was proposed in~\cite{BAB78}, in which auxiliary local problems with homogeneous Dirichlet boundary conditions are solved on each patch $\Omega_i$ of elements connected to the vertex $i$. Denoting by $\{ \varphi_j \}$ the basis functions of the approximation space $V_h^p$, and by $\{ \phi_i \}$ the first-order Lagrange basis functions associated to the mesh vertices (the support of $\phi_i$ is $\Omega_i$), we have, due to the partition of unity property $\dis \sum_i \phi_i=1$, that
\begin{equation}\label{eq:ressplitpatch}
\forall v \in V, \qquad B(e,v) =R \left( v\sum_i \phi_i \right) = \sum_i R(v \, \phi_i). 
\end{equation}
The idea is to replace the above single global residual problem by a set of local, independent problems. Noticing that $v \, \phi_i \in V_0(\Omega_i)$, with $V_0(\Omega_i) = \left\{w \in H^1(\Omega_i), \ \ \text{$w=0$ on $\partial \Omega_i$} \right\}$, and inspired by~\eqref{eq:ressplit}, we introduce the following local problem on each patch $\Omega_i$:
\begin{multline}\label{eq:reslocpatch}
\text{Find $e_i \in V_0(\Omega_i)$ such that, for any $v \in V_0(\Omega_i)$,} \\ B_{\Omega_i}(e_i,v) = \sum_{K\subset \Omega_i} \int_K r_K \, v -\sum_{\Gamma \subset \Omega_i}\int_\Gamma t_\Gamma \, v,
\end{multline}
where $r_K$ and $t_\Gamma$ are defined by~\eqref{eq:def_rK_tGamma} and where $\dis B_{\Omega_i}(u,v)=\int_{\Omega_i} \Aa\nab u \cdot \nab v$ is the restriction of $B$ on $\Omega_i$. Defining the estimate $\dis \eta^2=\sum_i B_{\Omega_i}(e_i,e_i)=\sum_i \vertiii{e_i}^2_{\Omega_i}$, it was shown in~\cite{BAB78} that there exist $C_1$ and $C_2$ independent of $h$ such that
$$
C_1 \, \eta \le \vertiii{e} \le C_2 \, \eta.
$$
In view of~\eqref{eq:c_est_le_sup}, the quantity $\sum_i e_i \in V$ obviously satisfies the lower bound
\begin{equation}\label{eq:LBres1}
\frac{|R(\sum_i e_i)|}{\vertiii{\sum_i e_i}}\le \vertiii{e}. 
\end{equation}
However, this approach does not provide a guaranteed upper bound on $e$ and $\eta$ is often a poor approximation of the error.

\medskip

A variation in the subdomain-residual method has been first proposed in~\cite{CAR00b,MAC00,MOR03,PRU04}, using auxiliary local problems with Neumann boundary conditions. It explicitly employs the partition of unity property verified by piecewise affine FE shape functions $\{ \phi_i \}$ in order to introduce well-posed (self-equilibrated) Neumann problems localized over patches $\Omega_i$ (which, we recall, are the support of the shape functions $\phi_i$). Consider the space $\dis W(\Omega_i) = \left\{ v \in L^1_{\rm loc}(\Omega_i), \ \ \int_{\Omega_i} \phi_i \, \Aa\nab v\cdot \nab v < \infty, \ \ \text{$v=0$ on $\partial \Omega_i \cap \Gamma_D$} \right\}$,
endowed with the inner product $\dis B_{\phi_i,\Omega_i}(u,v) = \int_{\Omega_i} \phi_i \, \Aa\nab u \cdot \nab v$ and the associated semi-norm $\vertiii{v}_{\phi_i,\Omega_i}$. Starting from~\eqref{eq:ressplitpatch}, the idea is to introduce the following local problems:
\begin{equation}\label{eq:locpbparchphii}
\text{Find $\xi_i \in W(\Omega_i)$ such that, for any $v \in W(\Omega_i)$,}\quad B_{\phi_i,\Omega_i}(\xi_i,v)=R(v \, \phi_i).
\end{equation}
For patches $\Omega_i$ which are not connected to $\Gamma_D$, the local problems~\eqref{eq:locpbparchphii} are of Neumann type. The compatibility condition is satisfied since, for any constant test function $v=v_0$, we have $\dis R(v_0 \, \phi_i) = v_0 \, R(\phi_i)=0$, the last equality being a consequence of~\eqref{eq:errorprop} and $\phi_i \in V_h^p$. The solution to~\eqref{eq:locpbparchphii} being defined up to an additive constant, it is in practice searched in the space $\dis \left\{ v \in W(\Omega_i), \ \int_{\Omega_i} v \, \phi_i=0 \right\}$. The well-posedness of~\eqref{eq:locpbparchphii} is established in~\cite{MOR03} using a weighted Poincar\'e-Wirtinger inequality.

Using the above $\xi_i$, we now build an error estimator. For any $v \in V$, we see that $v_{|\Omega_i} \in W(\Omega_i)$ for any $i$. We thus infer from~\eqref{eq:ressplitpatch} and~\eqref{eq:locpbparchphii} that
\begin{multline*}
B(e,v) 
= 
\sum_i R(v \, \phi_i)
=
\sum_i B_{\phi_i,\Omega_i}(\xi_i,v) 
\le 
\sum_i \sqrt{B_{\phi_i,\Omega_i}(\xi_i,\xi_i)} \sqrt{B_{\phi_i,\Omega_i}(v,v)} 
\\
\le 
\sqrt{\sum_i B_{\phi_i,\Omega_i}(\xi_i,\xi_i)} \sqrt{\sum_i B_{\phi_i,\Omega_i}(v,v)}.
\end{multline*}
Choosing $v=e$, we obtain
$$
\vertiii{e}^2 = B(e,e) \leq \eta \ \sqrt{\sum_i B_{\phi_i,\Omega_i}(e,e)}
$$
where $\dis \eta = \sqrt{\sum_i B_{\phi_i,\Omega_i}(\xi_i,\xi_i)} = \sqrt{\sum_i \vertiii{\xi_i}^2_{\phi_i,\Omega_i}}$. Noticing that $\dis \sum_i B_{\phi_i,\Omega_i}(e,e) = \sum_i B_{\phi_i,\Omega}(e,e) = B(e,e)$, we then deduce that
\begin{equation}
\label{eq:retour}
\vertiii{e} \leq \eta = \sqrt{\sum_i \vertiii{\xi_i}^2_{\phi_i,\Omega_i}}.
\end{equation}
Furthermore, a lower bound can be easily obtained as follows~\cite{PRU04}: using~\eqref{eq:c_est_le_sup} and the choice $\widetilde{v} = \sum_i \xi_i \, \phi_i \in V$, we write
\begin{equation}\label{eq:LBres2}
\vertiii{e} 
= 
\sup_{v \in V, \, v\neq 0} \frac{|R(v)|}{\vertiii{v}}
\geq
\frac{|R(\widetilde{v})|}{\vertiii{\widetilde{v}}}
=
\frac{\eta^2}{\vertiii{\widetilde{v}}}.
\end{equation}

\begin{remark}
In practice, the local problems~\eqref{eq:locpbparchphii} are solved on the polynomial space $W^{p+k}(\Omega_i)=\mathcal{P}^{p+k}(\Omega_i) \cap W(\Omega_i)$ of polynomials with degree up to $p+k$, leading to the estimate $\eta^\star$. It can be shown (see~\cite{PRU04}) that $\dis \vertiii{e} \le \eta^\star + 2 \inf_{v\in V_h^{p+k}} \vertiii{u-v}$ and that there exists a constant $C>1$, independent of $h$ and $k$, such that $\eta^\star \le C \vertiii{e}$.
\end{remark}

In~\cite{PAR06}, yet another variation of the subdomain residual method (later named ``flux free'') is considered, with another local problem (again complemented, as in~\cite{CAR00b,MOR03,PRU04}, with Neumann boundary conditions). Defining $V(\Omega_i)=\left\{v \in H^1(\Omega_i), \ \ \text{$v=0$ on $\partial \Omega_i \cap \Gamma_D$} \right\}$, local problems on patches are introduced as:
\begin{equation}\label{eq:locpbparchphii2}
\text{Find $z_i \in V(\Omega_i)$ such that, for any $v \in V(\Omega_i)$,} \quad B_{\Omega_i}(z_i,v)=R(v \, \phi_i),
\end{equation}
where $B_{\Omega_i}$ is defined as in~\eqref{eq:reslocpatch}, i.e. $\dis B_{\Omega_i}(u,v) = \int_{\Omega_i} \Aa\nab u \cdot \nab v$ (note the difference between $B_{\Omega_i}$ and the bilinear form used in~\eqref{eq:locpbparchphii}). For patches $\Omega_i$ which are not connected to $\Gamma_D$, these local problems are again of Neumann type. The condition $\dis \int_{\Omega_i} z_i=0$ is imposed for these patches $\Omega_i$. The problems~\eqref{eq:locpbparchphii2} are then well-posed.

To build an error estimator, we proceed as follows. For any $v \in V$, we see that $v_{|\Omega_i} \in V(\Omega_i)$ for any $i$. We thus infer from~\eqref{eq:ressplitpatch} and~\eqref{eq:locpbparchphii2} that
\begin{multline*}
B(e,v)
= 
\sum_i R(v \, \phi_i) 
= 
\sum_i B_{\Omega_i}(z_i,v)
= 
\sum_i B_{\rm brok}(z_i,v)
\\ =
B_{\rm brok} \left( \sum_i z_i,v \right)
\le 
\sqrt{B_{\rm brok}\left(\sum_i z_i,\sum_i z_i\right)} \ \sqrt{B(v,v)},
\end{multline*}
where $\dis B_{\rm brok}(u,v) = \sum_{K \in \mT_h} B_K(u,v) = \sum_{K \in \mT_h} \int_K \Aa \nab u \cdot \nab v$. Taking $v=e$, we deduce the upper bound
\begin{equation}\label{eq:flux_free2_UB}
\vertiii{e} \leq \eta = \sqrt{B_{\rm brok}\left(\sum_i z_i,\sum_i z_i\right)}.
\end{equation}
Numerical experiments show that the estimate~\eqref{eq:flux_free2_UB} obtained by solving~\eqref{eq:locpbparchphii2} is more accurate than the estimate~\eqref{eq:retour} obtained by solving~\eqref{eq:locpbparchphii} (i.e. when weighting the operator) (see~\cite{PAR06}). A lower bound on the error can be obtained as in the previous approach, using~\eqref{eq:c_est_le_sup} and $\widetilde{v} = \sum_i z_i \, \phi_i \in V$.

\begin{remark}
For linear elasticity problems, $B_{\Omega_i}(\bz_i,\bv)$ also vanishes when $\bv$ is a rigid body motion, that is when $\bv(\bx) = \bv_0 + \MM \bx$ where $\bv_0$ is constant and $\MM$ is a skew-symmetric matrix. It is necessary to change the right-hand side of~\eqref{eq:locpbparchphii2} in $\dis R\big( (\bv-\Pi^1_h\bv) \, \phi_i \big)$, where $\Pi^1_h \bv$ is the Cl\'ement interpolant of $\bv$ on piecewise affine FE functions, to ensure the solvability of~\eqref{eq:locpbparchphii2} in this more complex setting (see~\cite{PAR06}). In practice, $\Pi^1_h \bv$ is often replaced by the nodal interpolant of $\bv$.
\end{remark}

\subsubsection{Element residual methods}\label{section:elementresidualmeth}

Element implicit residual methods require solving auxiliary problems on each element $K$ rather than on patches of elements as in Section~\ref{section:locpbpatches} (see~\cite{BAN85,DEM84,STR92}). Defining $V(K)=\left\{v \in H^1(K), \ \ \text{$v=0$ on $\partial K \cap \Gamma_D$} \right\}$ and the broken space
$$
V_{\rm brok} = \oplus_K V(K),
$$
the local problems typically read:
\begin{multline}\label{locpbelem}
\text{Find $e_K \in V(K)$ such that, for any $v \in V(K)$,} \\ B_K(e_K,v)=\int_K r_K \, v + \sum_{\Gamma \subset \partial K \setminus \Gamma_D} \int_\Gamma R_\Gamma \, v, 
\end{multline}
with $R_\Gamma=-t_\Gamma/2$ (resp. $R_\Gamma=-t_\Gamma$) if $\Gamma \subset \Gamma_{\rm int}$ (resp. $\Gamma \subset \Gamma_N$), where $r_K$ and $t_\Gamma$ are defined by~\eqref{eq:def_rK_tGamma}.

For any $v \in V \subset V_{\rm brok}$, we have
$$
B(e,v) = R(v) = \sum_{K \in \mT_h} B_K(e_K,v),
$$
which implies, taking $v=e$, that the local estimates $\eta^2_K=B_K(e_K,e_K)$ provide for a guaranteed upper bound on the error:
\begin{equation}\label{eq:UB_elt}
\vertiii{e} \le \eta=\sqrt{\sum_{K \in \mT_h} \eta^2_K}.
\end{equation}
In addition, a lower bound on $\vertiii{e}$ can be derived from $e_{\rm est} = \sum_K e_K$, see~\cite{DIE03}. Note that $e_{\rm est} \in V_{\rm brok}$ but that, in general, $e_{\rm est} \not\in V$. In order to use~\eqref{eq:c_est_le_sup}, one has to introduce a correction $e_{\rm cor} \in V_{\rm brok}$ such that $e_{\rm est}+e_{\rm cor}\in V$.

\medskip

However, the local problems~\eqref{locpbelem} are not necessarily well-posed. For elements $K$ which are not connected to $\Gamma_D$, the problem~\eqref{locpbelem} is of Neumann type, and the compatibility relation is not satisfied (taking $v=1$ in~\eqref{locpbelem}, the left-hand side vanishes but the right-hand side does not, in the general case). Keeping $R_\Gamma$ defined from a trivial flux averaging as above, there are two possibilities to circumvent this issue:
\begin{itemize}
\item search $e_K$ in a regularizing subspace of bubble functions $V^{\rm comp}(K)$ (see~\cite{AIN00,BAN85}) vanishing at vertex nodes, so that the bilinear form is coercive. The quality of the obtained estimates of course depends on the choice of the subspace. In addition, the upper bound is not guaranteed any more since $e$ cannot usually be chosen as a test function in this variant of~\eqref{locpbelem}. 
\item search $e_K$ in $V(K)$ but change the right-hand side of~\eqref{locpbelem} in $\dis \int_K r_K (v-\Pi_h^1 v) + \sum_{\Gamma \subset \partial K \setminus\Gamma_D} \int_\Gamma R_\Gamma \, (v-\Pi_h^1v)$, where $\Pi_h^1v$ is the Cl\'ement interpolant of $v$ on piecewise affine FE functions (see~\cite{BAN85}). The Neumann compatibility relation is now satisfied and this variant of~\eqref{locpbelem} is well-posed, $e_K$ being defined up to the addition of a constant. The upper bound~\eqref{eq:UB_elt} again holds~\cite{BAN85}. 
\end{itemize}

An alternative and more effective possibility~\cite{AIN93}, which also provides for an upper bound on the error, is to consider a variant of~\eqref{locpbelem} where $R_\Gamma$ is replaced by some carefully chosen boundary data $\widehat{R}_{\Gamma,K}$ on $\partial K$. These data $\widehat{R}_{\Gamma,K}$ are constructed so that equilibrium over each element $K$ is ensured, i.e. such that
\begin{equation}
\label{eq:retour2}
\int_K r_K + \sum_{\Gamma \subset \partial K \setminus \Gamma_D}\int_\Gamma \widehat{R}_{\Gamma,K}=0.
\end{equation}
This is the basis of the so-called equilibrated element residual estimates.

The following variant of~\eqref{locpbelem} is then considered:
\begin{multline}
\label{eq:retour3}
\text{Find $\widehat{e}_K \in V(K)$ such that, for any $v \in V(K)$,} \\ B_K(\widehat{e}_K,v)=\int_K r_K v + \sum_{\Gamma \subset \partial K \setminus \Gamma_D}\int_\Gamma \widehat{R}_{\Gamma,K} \, v.
\end{multline}
This again corresponds to a problem with Neumann boundary conditions. In view of~\eqref{eq:retour2}, the compatibility condition holds and the solution $\widehat{e}_K$ to~\eqref{eq:retour3} is therefore well-defined, up to the addition of a constant.

The tractions $\widehat{R}_{\Gamma,K}$ are often chosen of the following form: on each element $K$,
\begin{equation*}
\widehat{R}_{\Gamma,K}(\bx)=\sigma_{\Gamma,K} \ \widehat{g}_\Gamma (\bx) -\Aa\nab u_{h|K}\cdot \bn_K
\end{equation*}
where $\sigma_{\Gamma,K}=\pm 1$ (to ensure the continuity of the normal flux across the element edges) and where $\widehat{g}_\Gamma$ is an approximation of the exact normal boundary flux $\Aa\nab u_{|\Gamma}\cdot \bn$ to be appropriately chosen (see Section~\ref{section:SA} for technicalities when computing $\widehat{g}_\Gamma$). The fact that $\widehat{R}_{\Gamma,K}$ satisfies the equilibrium equation~\eqref{eq:retour2} provides for more realistic boundary conditions than the two approaches described at the beginning of this Section~\ref{section:elementresidualmeth}. The approach eventually provides a guaranteed error estimate that is observed to be more accurate. 

\begin{remark}
The change of the right-hand side of~\eqref{locpbelem} performed in~\cite{BAN85} can actually be seen as an implicit way of recovering equilibrated tractions.
\end{remark}

% ---------------------------------------------------------------------
% ---------------------------------------------------------------------

\section{Duality-based methods}\label{section:CRE}

We now turn to a different class of methods. The \textit{a posteriori} error estimation methods that we present in this section are built from a residual on the constitutive relation $\bq=\Aa\nab u$ rather than from a balance residual. More precisely, the estimation methods presented in Section~\ref{section:residuals} are based on the residuals $r_K = f +\dive \left[ (\Aa\nab u_h)_{|K} \right]$ in each element $K$ and on the flux jumps $t_\Gamma$, two quantities that enter in the right-hand side of~\eqref{eq:ressplit} and of the auxiliary local problems~\eqref{eq:reslocpatch}, \eqref{locpbelem}, \eqref{eq:retour3}, \dots The spirit of the approaches presented here is different. Recalling that $f \in L^2(\Omega)$, the problem~\eqref{eq:refpbstrong} is written, somewhat as for mixed methods, in the form
\begin{eqnarray}
\label{eq:mixed1}
-\dive \bq &=& f \ \ \text{in $\Omega$}, \qquad \bq \cdot \bn = g \ \ \text{on $\Gamma_N$}, \qquad \bq \in H(\dive,\Omega),
\\
\label{eq:mixed2}
u &=& 0 \ \ \text{on $\Gamma_D$}, \qquad u \in H^1(\Omega),
\\
\label{eq:mixed3}
\Aa\nab u &=& \bq \ \ \text{in $\Omega$},
\end{eqnarray}
where we recall that $H(\dive,\Omega)$ is defined by~\eqref{eq:def_Hdiv} and that, for any $\bq \in H(\dive,\Omega)$, the trace of $\bq\cdot \bn$ on $\partial \Omega$ is well-defined.

\medskip

The methods described in this Section~\ref{section:CRE} aim at
\begin{itemize}
\item building a function $\widehat{u}$ that satisfies~\eqref{eq:mixed2} (one usually chooses $\widehat{u} = u_h$);
\item building a vector-valued function $\widehat{\bq}$ that satisfies the equilibrium equation~\eqref{eq:mixed1}; this is often done by post-processing $\Aa \nab u_h$;
\item estimating the error $e=u-u_h$ in term of the error in~\eqref{eq:mixed3}, that is $\Aa \nab \widehat{u} - \widehat{\bq}$.
\end{itemize}

\subsection{Dual approach}\label{section:dualapp}

Starting from the residual equation~\eqref{eq:errorprop} verified by $e\in V$, namely
\begin{equation*}
\forall v \in V, \qquad B(e,v)=F(v)-B(u_h,v),
\end{equation*}
we observe that $e$ is equivalently the solution of the following (so-called primal) variational problem
\begin{equation*}
J(e) = \inf \left\{ J(w), \quad w\in V \right\},
\end{equation*}
where $J$ is the quadratic functional
\begin{multline*}
J(w) = \frac{1}{2} B(w,w)-F(w)+B(u_h,w) \\ = \frac{1}{2} \intO \Aa\nab w\cdot \nab w -\intO f \, w -\int_{\Gamma_N} g \, w + \intO \Aa\nab u_h \cdot \nab w.
\end{multline*}
We also note that $\dis J(e)=-\frac{1}{2}\intO \Aa\nab e\cdot \nab e$. We therefore have
\begin{equation*}
 \forall w \in V, \qquad \vertiii{e}^2 = -2J(e) \ge -2J(w). 
\end{equation*}
An interesting consequence is that we can easily compute a lower bound on the error $\vertiii{e}$, namely $\sqrt{-2 J(w)}$ for any $w \in V$ such that $J(w) < 0$. However, this lower bound is usually poor unless $w$ is a suitably chosen representation of $e$ as in~\eqref{eq:LBZZ}, \eqref{eq:LBres1} or~\eqref{eq:LBres2}.

\begin{remark}
An alternative way to obtain that lower bound is as follows. Introduce the potential energy functional $J_1$ associated with the reference problem, which is defined by $\dis J_1(w)=\frac{1}{2}\intO \Aa\nab w\cdot \nab w -\intO f \, w -\int_{\Gamma_N} g \, w$. We recall that the solution $u$ to~\eqref{eq:refpbstrong} satisfies $\dis J_1(u)=\inf_{w\in V}J_1(w)$. We then have that $\vertiii{e}^2 = 2(J_1(u_h)-J_1(u))$, and hence the lower bound
\begin{equation}
\label{eq:lower}
\forall w \in V, \quad \vertiii{e}^2 \ge -2(J_1(w)-J_1(u_h)) = -2J(w-u_h).
\end{equation}
We note that, in view of~\eqref{eq:lower}, we would like to compute $w$ such that $J_1(w)$ is as small as possible. In the space $V_h^p$, we know that $u_h$ is the unique minimizer of $J_1$. Using~\eqref{eq:lower} with $w \in V_h^p$ is thus not informative. For~\eqref{eq:lower} to be useful, $w$ should be searched in a space larger than $V_h^p$, as performed in~\eqref{eq:LBZZ}, \eqref{eq:LBres1} or~\eqref{eq:LBres2}.
\end{remark}

A complementary variational principle can be associated to the primal variational principle and may be used to get an upper bound on $\vertiii{e}$. To that aim, we introduce the space
\begin{equation*}
W=\{\bp \in H(\dive,\Omega), \quad \dive \bp + f=0 \; \text{in $\Omega$}, \quad \bp\cdot \bn = g\; \text{on $\Gamma_N$} \},
\end{equation*}
where we recall that the Hilbert space $H(\dive,\Omega)$ is defined by~\eqref{eq:def_Hdiv}. 

Consider the quadratic functional
\begin{equation*}
G(\bp) = \frac{1}{2} \intO \Aa^{-1} (\bp-\Aa \nab u_h) \cdot (\bp-\Aa \nab u_h) = \frac{1}{2} \vertiii{\bp-\Aa \nab u_h}^2_q
\end{equation*}
and the so-called complementary variational problem
\begin{equation}\label{eq:pbcomplement}
\inf \left\{ G(\bp), \ \ \bp \in W \right\}.
\end{equation}
The minimization problem~\eqref{eq:pbcomplement} is well posed, and it is easy to see that the solution flux $\bq=\Aa\nab u$ (where $u$ is the solution to the reference problem~\eqref{eq:refpbstrong}) is the solution to~\eqref{eq:pbcomplement}. We obviously have $\vertiii{e}^2 = 2G(\bq)$. We thus deduce the following upper bound:
\begin{equation}\label{eq:upper}
\forall \bp \in W, \quad \vertiii{e}^2 \le 2G(\bp).
\end{equation}

\begin{remark}
An alternative way to obtain that upper bound is as follows. Introduce the complementary energy functional $J_2$ associated with the reference problem, which is defined by $\dis J_2(\bp) = \frac{1}{2} \intO \Aa^{-1}\bp \cdot \bp$. We recall that the exact flux $\bq=\Aa\nab u$ satisfies $\dis J_2(\bq)=\inf_{\bp\in W} J_2(\bp)$. We then have that $\vertiii{e}^2 = 2(J_2(\bq)+J_1(u_h))$, from which we deduce the upper bound
\begin{equation}\label{eq:majJ2}
\forall \bp \in W, \quad \vertiii{e}^2 \le 2(J_2(\bp)+J_1(u_h)) = 2 G(\bp).
\end{equation}
In view of~\eqref{eq:majJ2}, we would like to compute $\bp$ such that $J_2(\bp)$ is as small as possible. 
\end{remark}

\subsection{Constitutive relation error functional}
\label{sec:cre}

A flux field that belongs to $W$ will be said to be statically admissible (in the sense that it verifies the equilibrium equations) and denoted $\widehat{\bq}$ in the following. For the pair $(u_h,\widehat{\bq}) \in V_h^p \times W$, we define the constitutive relation error (CRE) functional $E_{\rm CRE}$ by
\begin{equation*}
E^2_{\rm CRE}(u_h,\widehat{\bq})=\frac{1}{2} \, \vertiii{\widehat{\bq}-\Aa\nab u_h}^2_q.
\end{equation*}
We of course note that $E^2_{\rm CRE}(u_h,\widehat{\bq}) = G(\widehat{\bq}) = J_1(u_h)+J_2(\widehat{\bq})$. In view of~\eqref{eq:upper}, we then get that, for any $\widehat{\bq} \in W$, $\sqrt{2} \, E_{\rm CRE}(u_h,\widehat{\bq})$ is an upper-bound on $\vertiii{e}$ (without any generic multiplicative constant), and that $\dis \vertiii{e} = \inf_{\widehat{\bq} \in W}\sqrt{2} \, E_{\rm CRE}(u_h,\widehat{\bq})$. We show in what follows how to obtain more precise relations.

\medskip

We first have the following result (the so-called Prager-Synge equality), which will be most useful in what follows.
\begin{lemma}
For any $\widehat{\bq} \in W$, we have
\begin{equation}\label{eq:propertiesCRE2}
2 \, E^2_{\rm CRE}(u_h,\widehat{\bq}) = \vertiii{e}^2 + \vertiii{\bq-\widehat{\bq}}^2_q.
\end{equation}
\end{lemma}

\begin{proof}
Let $u$ be the solution to the reference problem~\eqref{eq:refpbstrong}. We write
\begin{eqnarray*}
2 \, E_{\rm ERC}^2(u_h,\widehat{\bq})
&=&
\vertiii{\widehat{\bq} - \Aa \nab u_h}^2_q
\\
&=&
\vertiii{(\widehat{\bq} - \Aa \nab u) + (\Aa \nab u - \Aa \nab u_h)}^2_q
\\
&=&
\vertiii{\widehat{\bq} - \Aa \nab u}^2_q
+
\vertiii{u - u_h}^2
+
2 \intO (\widehat{\bq} - \Aa \nab u) \cdot  \nab (u - u_h).
\end{eqnarray*}
The last term in the above right-hand side vanishes by integration by part, using that both $\widehat{\bq}$ and $\bq = \Aa \nab u$ belong to $W$ and that $u=u_h$ on $\Gamma_D$. We hence obtain~\eqref{eq:propertiesCRE2}.
\end{proof}

We also have the following properties (see~\cite{LAD04}):
\begin{equation}\label{eq:propertiesCRE1}
\begin{aligned}
(u_h,\widehat{\bq}) \in V_h^p \times W \ \text{s.t.} \ E^2_{\rm CRE}(u_h,\widehat{\bq}) &= 0 \quad \Longleftrightarrow \quad u_h = u \ \text{and} \ \widehat{\bq} = \bq, \\
\text{Hypercircle property:} \quad
E^2_{\rm CRE}(u_h,\widehat{\bq}) &= 2 \, \vertiii{\bq-\widehat{\bq}^m}^2_q \ \ \text{where} \ \ \widehat{\bq}^m=\frac{1}{2}(\widehat{\bq}+\Aa\nab u_h).
\end{aligned}
\end{equation}
The hypercircle property can easily be shown as a consequence of the Prager-Synge equality~\eqref{eq:propertiesCRE2}, as illustrated on Figure~\ref{fig:hypercircle}. Introducing $\delta \bq = (\widehat{\bq}-\bq_h)/2$, we indeed infer from~\eqref{eq:propertiesCRE2} that
\begin{align*}
  2 \, E^2_{\rm CRE}(u_h,\widehat{\bq})
  &=
  \vertiii{\bq-\bq_h}^2_q + \vertiii{\bq-\widehat{\bq}}^2_q
  \\
  &=
  \vertiii{\bq-(\widehat{\bq}^m-\delta \bq)}^2_q + \vertiii{\bq-(\widehat{\bq}^m+\delta \bq)}^2_q
  \\
  &=
  2 \, \vertiii{\bq-\widehat{\bq}^m}^2_q + 2 \, \vertiii{\delta \bq}^2_q,
\end{align*}
and we observe that $2 \, \vertiii{\delta \bq}^2_q = \vertiii{\widehat{\bq}-\bq_h}^2_q/2 = E^2_{\rm CRE}(u_h,\widehat{\bq})$, which yields the hypercircle property.

\begin{figure}[H]
\begin{center}
\includegraphics[width=70mm]{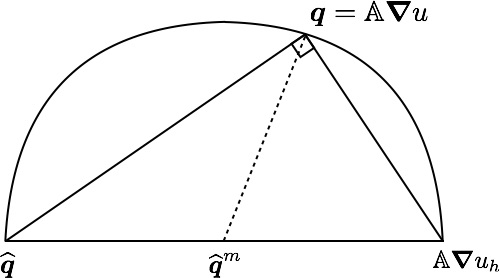}
\end{center}
\caption{Geometrical illustration of the Prager-Synge equality and the hypercircle property. \label{fig:hypercircle}}
\end{figure}

\medskip

A practical consequence of~\eqref{eq:propertiesCRE2} is that the quantity $\sqrt{2} \, E_{\rm CRE}(u_h,\widehat{\bq})$ is an upper bound on the error $\vertiii{e}$ for any $\widehat{\bq}\in W$. The accuracy of this bound of course depends on the choice of $\widehat{\bq}$. As shown below, it is possible to efficiently build some $\widehat{\bq} \in W$ such that this upper bound is accurate in the sense that $\sqrt{2} \, E_{\rm CRE}(u_h,\widehat{\bq}) / \vertiii{e}$ is close to 1.

In addition, depending on the precise way the flux $\widehat{\bq}\in W$ is constructed, a lower bound on $\vertiii{e}$ can also be shown (see e.g.~\cite{LAD83,LAD04}). This lower bound is usually of the form $E_{\rm CRE}(u_h,\widehat{\bq}) \le C \, \vertiii{e}$, where $C$ is a constant independent of the mesh size $h$, so that $E_{\rm CRE}(u_h,\widehat{\bq})$ and $\vertiii{e}$ have the same asymptotic convergence rate. The precise statement and the proof of such a bound for a 2D problem discretized by piecewise affine finite elements is given in Appendix~\ref{section:appendixA} below.

\subsection{Construction of a statically admissible flux field}\label{section:SA}

In order to compute an upper bound on the error, one possibility is to use~\eqref{eq:upper} and to approximate the complementary problem~\eqref{eq:pbcomplement} using a FE discretization with equilibrium elements, namely functions in $W$ (see~\cite{DEB95,FRA01,KEM09,MOI17}). Even though this procedure usually provides for the sharpest error bounds, it is in practice very challenging for a twofold reason. First, the constraints in the space $W$ make the practical construction of feasible functions $\widehat{\bq}$ not straightforward (and barely implementable in commercial finite element softwares since it refers to a non-conventional numerical architecture). Second, such an approach amounts to solve an additional global problem, and thus requires a large computational effort. In addition, it is unclear whether it is actually necessary to carry out any further global computations, since there is already some global information in the FE approximation $u_h$ (and in the flux $\bq_h = \Aa \nab u_h$) that we have at our disposal. 

\medskip

In the following, we proceed differently. We detail various techniques that enable to recover a flux $\widehat{\bq}_h \in W$ from local independent computations and from postprocessing operations on $u_h$ and $\bq_h$. Note that we are able to build functions in $W$ (a space that carries a {\em global} contraint) while only solving {\em local} problems because we are going to use the global information we already have, through the knowledge of $\bq_h = \Aa \nab u_h$.

\subsubsection{Hybrid-flux approach}\label{section:hybridfluxapp}

In this approach, also referred to as Element Equilibration Technique (EET), the computation of $\widehat{\bq}_h \in W$ is performed in two steps~\cite{AIN93,LAD83,LAD96,LAD04,PLE11}:
\begin{itemize}
\item[(i)] construction of functions $\widehat{g}_K$ along element edges (with $\widehat{g}_K=g$ on $\Gamma_N$) that should satisfy the so-called element equilibrium equation:
\begin{equation}\label{eq:equilelement}
\forall K \in \mT_h, \quad \int_K f + \int_{\partial K}\widehat{g}_K =0,
\end{equation}
and should be continuous across adjacent element edges (equilibrium of edges). In analogy to mechanical problems, the functions $\widehat{g}_K$ are called equilibrated tractions. We will use the numerical solution ($u_h$, $\bq_h$) to build these $\widehat{g}_K$.
\item[(ii)] local construction of $\widehat{\bq}_{h|K}$ over each element $K$, with prescribed tractions $\widehat{g}_K$ and source term $f$; they should be a solution to the local Neumann problem
\begin{equation}\label{eq:localpbelement}
-\dive \widehat{\bq}_{h|K} = f \quad \text{in $K$}, \qquad \widehat{\bq}_{h|K}\cdot \bn_K = \widehat{g}_K \quad \text{on $\partial K$}.
\end{equation}
\end{itemize}
We now successively detail these two steps.

\medskip

\noindent
\textbf{Step 1: construction of equilibrated tractions $\widehat{g}_K$ on the element edges.} The compatibility (equilibrium) conditions on $\widehat{g}_K$ are a global constraint, involving tractions on boundaries of all elements in the mesh $\mT_h$. Nevertheless, we show below that the computation of the set of $\widehat{g}_K$ can be performed locally, by solving small and independent linear systems of equations.

After ordering the elements, we define the function $\sigma_{\Gamma,K} : \partial K \rightarrow \left\{+1,-1\right\}$ by
\begin{equation} \label{eq:def_sigma_gamma_K}
\sigma_{\Gamma,K}=\left\{
\begin{array}{l}
+1\quad \text{if} \ \Gamma = \overline{K}\cap\overline{J} \ \text{for some element $J<K$}, \\ 
-1 \quad \text{if} \ \Gamma = \overline{K}\cap\overline{J} \ \text{for some element $J>K$}, \\ 
+1 \quad \text{if} \ \Gamma \subset \Gamma_D. 
\end{array}
\right.
\end{equation}
We then construct the tractions $\widehat{g}_K$ in the form $\widehat{g}_K = \sigma_{\Gamma,K} \, \widehat{g}_\Gamma$ where $\widehat{g}_\Gamma$ is a smooth function defined on each element interface $\Gamma$. This ensures that the normal flux is continuous across the edges.

\medskip

We request that $\widehat{\bq}_h$ is a solution to~\eqref{eq:localpbelement} and satisfies the following energy condition (called \textit{strong prolongation condition}): for each element $K$ and each node $i$ connected to $K$,
\begin{equation} \label{eq:prolong}
\int_K(\widehat{\bq}_h-\bq_h) \cdot \nab \varphi_i = 0,
\end{equation}
where $\{ \varphi_i \}$ are the basis functions of the approximation space $V_h^p$. This can be recast as 
\begin{equation}
\label{eq:helas}
\sum_{\Gamma \subset \partial K} \int_\Gamma \sigma_{\Gamma,K} \, \widehat{g}_\Gamma \, \varphi_i = Q_i^K 
\qquad \text{with} \qquad
Q_i^K = \int_K (\bq_h \cdot \nab \varphi_i - f \, \varphi_i).
\end{equation}
Note that, if~\eqref{eq:helas} is satisfied, then, by summing over $i$ and by using the partition of unity property $\dis \sum_i \varphi_{i|K}=1$, we check that~\eqref{eq:equilelement} is satisfied.

We are now left with building $\widehat{g}_\Gamma$ satisfying~\eqref{eq:helas}. For any adjacent elements $K_{\alpha}$ and $K_{\beta}$ (sharing the interface $\Gamma_{\alpha,\beta}$), introduce the quantity
\begin{equation}
  \label{eq:def_b_hat}
\widehat{b}_{\alpha,\beta}(i)
=
\int_{\Gamma_{\alpha,\beta}} \sigma_{\Gamma_{\alpha,\beta},K_\alpha} \, \widehat{g}_{\Gamma_{\alpha,\beta}} \, \varphi_i
=
\int_{\Gamma_{\alpha,\beta}} \widehat{g}_{K_\alpha} \, \varphi_i
=
-\int_{\Gamma_{\alpha,\beta}} \widehat{g}_{K_\beta} \, \varphi_i,
\end{equation}
which is the projection of the traction $\widehat{g}_{\Gamma_{\alpha,\beta}}$ over the FE shape function $\varphi_i$. For each node $i$, the equation~\eqref{eq:helas} can be recast as a linear system involving the quantities $\widehat{b}_{\alpha,\beta}(i)$ for the elements $K_\alpha$ and $K_\beta$ that are connected to node $i$. The number of such elements is limited, and hence the dimension of this linear system is limited. As an illustration, in the 2D case and for a vertex node $i$ which is not on the boundary $\partial \Omega$ (see Figure~\ref{fig:on_tourne}), the system is of the form
\begin{equation}
  \label{eq:helas3}
\begin{array}{rcl}
\widehat{b}_{1,2}(i) - \widehat{b}_{N_i,1}(i) &=& Q_i^{K_1}, \\
\widehat{b}_{2,3}(i) - \widehat{b}_{1,2}(i) &=& Q_i^{K_2}, \\
\ldots &=& \ldots \\
\widehat{b}_{N_i,1}(i) - \widehat{b}_{N_i-1,N_i}(i) &=& Q_i^{K_{N_i}},
\end{array}
\end{equation}
where $N_i$ is the number of elements connected to node $i$ (for other nodes, the system is of a similar form).

\begin{figure}[H]
\begin{center}
\includegraphics[width=60mm]{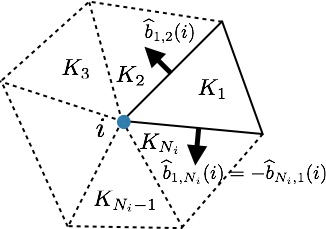}
\end{center}
\caption{Illustration of the prolongation condition~\eqref{eq:prolong} applied to a 2D triangular mesh in the conforming case, yielding a local linear system on traction projections $\widehat{b}_{\alpha,\beta}(i)$ associated with vertex node $i$. \label{fig:on_tourne}}
\end{figure}

With obvious notations, the linear system~\eqref{eq:helas3} can be written in the compact form $\MM \, \widehat{\bb}(i) = \bQ_i$, where $\widehat{\bb}(i)$ and $\bQ_i$ are $N_i$-dimensional vectors and $\MM$ is a $N_i \times N_i$ matrix, the elements of which are 0, $-1$ or $1$. It is easily checked that the kernel of $\MM$ is of dimension 1. The system~\eqref{eq:helas3} has solutions if and only if $\dis \sum_{j=1}^{N_i}Q_i^{K_j}=0$. This property indeed holds because it is exactly the weak formulation of the approximate solution $u_h$ for the test function $\varphi_i$. 

\medskip

Once the projections $\widehat{b}_{\alpha,\beta}(i)$ are computed for all nodes $i$, the tractions $\widehat{g}_\Gamma$ are next built as linear combinations of the basis functions $\varphi_j$. For instance, in a P1 setting, and letting $K_{\alpha}$ and $K_{\beta}$ be the elements sharing the edge $\Gamma_{\alpha,\beta}$ that connects the nodes $\ell$ and $m$, we write
\begin{equation}
\label{eq:combi-lin}
\widehat{g}_{\Gamma_{\alpha,\beta}} = c_\ell \, \varphi_\ell + c_m \, \varphi_m
\end{equation}
and determine the coefficients $c_\ell$ and $c_m$ such that
$$
\begin{cases}
\dis \int_{\Gamma_{\alpha,\beta}} \sigma_{\Gamma_{\alpha,\beta},K_\alpha} \, \left[ c_\ell \, \varphi_\ell + c_m \, \varphi_m \right] \, \varphi_\ell
=
\widehat{b}_{\alpha,\beta}(\ell),
\\ \noalign{\vskip 3pt}
\dis \int_{\Gamma_{\alpha,\beta}} \sigma_{\Gamma_{\alpha,\beta},K_\alpha} \, \left[ c_\ell \, \varphi_\ell + c_m \, \varphi_m \right] \, \varphi_m
=
\widehat{b}_{\alpha,\beta}(m).
\end{cases}
$$

\begin{remark}
\label{rem:underdetermined}
When the node $i$ is not connected to $\Gamma_N$, the associated local system, arising from~\eqref{eq:helas}, is underdetermined. In practice, the solution which is chosen is the one that minimizes a function of the form $\dis \sum_{\alpha,\beta} \frac{(\widehat{b}_{\alpha,\beta}(i)-b^m_{\alpha,\beta}(i))^2}{l_{\alpha,\beta}^2}$, where $l_{\alpha,\beta} = |\Gamma_{\alpha,\beta}|$ and where
\begin{equation}
\label{eq:def_b_m}
b^m_{\alpha,\beta}(i) = \frac{1}{2} \int_{\Gamma_{\alpha,\beta}} \varphi_i \, (\bq_{h|K_{\alpha}}+\bq_{h|K_{\beta}})\cdot \bn_{K_\alpha}
\end{equation}
is the average FE normal flux on $\Gamma_{\alpha,\beta}$. Consider for instance an internal node $i$, for which the system to solve is~\eqref{eq:helas3}. We introduce the following $N_i$-dimensional vectors:
\begin{gather*}
  \widehat{\bX} = \left[ \widehat{b}_{N_i,1}(i),\widehat{b}_{1,2}(i),\dots,\widehat{b}_{N_i-1,N_i}(i) \right]^T,
  \\
  \bX^m = \left[ b^m_{N_i,1}(i),b^m_{1,2}(i),\dots,b^m_{N_i-1,N_i}(i) \right]^T,
  \\
  \bX_0 = \left[ 0,Q_i^{K_1},Q_i^{K_2},\dots,\sum_{j=1}^{N_i-1}Q_i^{K_j} \right] \quad \text{and} \quad \bA=[1,1,\dots,1]^T.
\end{gather*}
Then the set of solutions to~\eqref{eq:helas3} is the set $\left\{ \widehat{\bX} = s \, \bA + \bX_0, \ \ s \in \RR \right\}$, $s$ representing the unknown value $\widehat{b}_{N_i,1}(i)$. Introducing the matrix
$$
\DD=diag\left(l_{N_i,1}^{-2},l_{1,2}^{-2},\ldots,l_{N_i-1,N_i}^{-2}\right),
$$
the cost function mentioned above reads
\begin{multline*}
f(s)=\left(\widehat{\bX}-\bX^m\right)^T \DD \left(\widehat{\bX}-\bX^m\right) \\ = s^2\bA^T\DD\bA-2s\bA^T\DD(\bX^m-\bX_0)+(\bX^m-\bX_0)^T\DD(\bX^m-\bX_0).
\end{multline*}
Its minimization among the solutions to~\eqref{eq:helas3} leads to the choice
$$
s=(\bA^T\DD\bA)^{-1}\bA^T\DD(\bX^m-\bX_0).
$$
We thus notice that $\widehat{\bX}$ is obtained explicitly with very little computational effort. 
\end{remark}

\begin{remark}
Several variants of the above method, which is based on~\eqref{eq:prolong}, have been proposed in the literature~\cite{FLO02,PLE12}. For instance, a ``weak'' prolongation condition may be applied to shape functions associated with non-vertex nodes alone. Tractions are then constructed as $\widehat{g}_K=L+H$ where $H$ (high-degree component) is fully computed from the weak prolongation condition whereas $L$ (low-degree component) is obtained by minimizing a global complementary energy. This procedure is in practice applied in zones with large gradients or large element aspect ratio in order to optimize the estimate. In the other zones, the above method, based on~\eqref{eq:prolong}, is used. 

Another variant is to construct equilibrated tractions using the Partition of Unity Method (PUM)~\cite{LAD10c,PLE11}. This variant provides results similar to those obtained with the method based on~\eqref{eq:prolong}, but is easier to implement in simulation softwares.
\end{remark}

\medskip

\noindent
\textbf{Step 2: local construction of the equilibrated flux $\widehat{\bq}_{h|K}$.} We first consider the specific case when the source term $f$ is a polynomial function. In that case, it is possible to analytically identify a solution $\widehat{\bq}_{h|K}$ to~\eqref{eq:localpbelement} in the form of a polynomial function with a degree consistent with those of $f$ and $\widehat{g}_K$ (recall that $\widehat{g}_K$ is a linear combination of the basis functions in $V_h^p$, in view of~\eqref{eq:combi-lin}). We refer to~\cite{LAD97} for details.

\begin{remark}
  A specific difficulty arises when the above analytical method is used for elasticity problems. Due to the symmetry of the stress tensor, tractions should be compatible at the element vertices (i.e. we should have $\dis \sigma_{\Gamma_1,K} \, \widehat{\bg}_{\Gamma_1} \cdot \bn_2 = \sigma_{\Gamma_2,K} \, \widehat{\bg}_{\Gamma_2} \cdot \bn_1$). However, this compatibility condition is usually not satisfied when the tractions $\widehat{\bg}$ are built as explained in Step~1 above. To circumvent this difficulty, a possibility is to split the element $K$ in sub-elements and to search for an equilibrated stress field as a polynomial function in each sub-element~\cite{LAD04}.
\end{remark}

We next turn to the case of a general right hand side $f$, in which case~\eqref{eq:localpbelement} is going to be numerically solved. Once the tractions $\widehat{g}_K$ have been computed, the best admissible flux $\widehat{\bq}_{h|K}$ is the one that minimizes $\vertiii{\widehat{\bq}-\bq_h}_{q,K}$ among all fluxes $\widehat{\bq}$ satisfying~\eqref{eq:localpbelement}, where $\bq_h = \Aa\nab u_h$ is the numerical flux. Recall indeed (see~\eqref{eq:upper}) that $\vertiii{\widehat{\bq}-\bq_h}_q$ is an upper bound on the error for any $\widehat{\bq} \in W$. 

For any $\widehat{\bq}$ satisfying~\eqref{eq:localpbelement}, we have
\begin{equation*}
\vertiii{\widehat{\bq}-\bq_h}_{q,K}^2 
= 
\vertiii{\widehat{\bq}}^2_{q,K} + \vertiii{\bq_h}_{q,K}^2 - 2 \left[ \int_K f \, u_h + \int_{\partial K} \widehat{g}_K \, u_h \right]
\end{equation*}
and it is thus equivalent to minimize $\vertiii{\widehat{\bq}}^2_{q,K} = 2 J_{2|K}(\widehat{\bq})$ over fluxes $\widehat{\bq}$ satisfying~\eqref{eq:localpbelement}. This problem can be written as the minimization of $\dis \int_K \Aa^{-1} \widehat{\bq} \cdot \widehat{\bq}$ under the constraint that $\dis \int_K \nabla v \cdot \widehat{\bq} - \int_K f \, v - \int_{\partial K} \widehat{g}_K \, v = 0$ for any $v \in H^1(K)$. Introducing the Lagrangian functional
$$
{\cal L}(\widehat{\bq},v) = \int_K \Aa^{-1} \widehat{\bq} \cdot \widehat{\bq} - \left[ \int_K \nabla v \cdot \widehat{\bq} - \int_K f \, v - \int_{\partial K} \widehat{g}_K \, v \right],
$$
we deduce from the saddle-point equation $\dis \frac{\partial {\cal L}}{\partial \widehat{\bq}} = 0$ that the minimizer $\widehat{\bq}_{h|K}$ should satisfy $\Aa^{-1} \widehat{\bq}_{h|K} = \nabla w_K$, where $w_K$ is the Lagrange multiplier. We therefore obtain that $\widehat{\bq}_{h|K} = \Aa \nabla w_K$. The above minimization problem is thus equivalent to finding a function $w_K \in H^1(K)$ such that
\begin{equation}\label{eq:dualelempb}
\forall v \in H^1(K), \qquad B_K(w_K,v)=\int_K f \, v + \int_{\partial K} \widehat{g}_K \, v,
\end{equation}
and then taking $\widehat{\bq}_{h|K} =\Aa \nab w_K$. 

In practice, a numerical approximation of $\widehat{\bq}_{h|K}$ is obtained by solving~\eqref{eq:dualelempb} using an enriched FE method. The basis functions are polynomial functions over the whole element $K$, with a degree $p+k$ (recall that $p$ is the order of the polynomial functions used to discretize $u_h$). A Galerkin approximation of~\eqref{eq:dualelempb} in $V_h^{p+k}(K)$ is thus performed. Orthogonal hierarchical subspaces $\widetilde{V}_{h|K}^q$ (for $q=p+1$, $p+2$, \dots, $p+k$) can be introduced to solve~\eqref{eq:dualelempb} effectively~\cite{AIN00}. The numerical comparisons performed in~\cite{BAB94} show that the approach based on exactly solving~\eqref{eq:dualelempb} and the numerical approach we have just described (approximating~\eqref{eq:dualelempb} in $V_h^{p+k}(K)$) provide similar CRE values (i.e. error bounds) when choosing $k\ge3$, even though the flux fields $\widehat{\bq}_{h|K}$ obtained in the latter approach do not strictly satisfy the equation~\eqref{eq:localpbelement} in its strong form (and therefore do not provide for a mathematically guaranteed upper error bound). We also refer to~\cite{STR92} for similar investigations.

\subsubsection{Flux-free approach}\label{section:fluxfreeapp}

As an alternative to the hybrid-flux approach described in Section~\ref{section:hybridfluxapp}, an admissible flux $\widehat{\bq}_h \in W$ can also be obtained using a flux-free technique~\cite{COT09,GAL09,MOI09,PAR09}. This approach is simpler to analyze and implement than the hybrid-flux approach since it circumvents the necessity of constructing equilibrated tractions on element edges (the quantities $\widehat{g}_K$ in the above approach, which are boundary conditions for the local problems). Here the local boundary conditions appear naturally. However, the flux-free approach requires to solve local problems over patches of elements (in contrast to the problem~\eqref{eq:dualelempb}, set over a single element) and is therefore more expensive. For instance, the local problem complexity is at least 3 times larger in 2D~\cite{CHO04}.

\medskip

We recall that $\Omega_i$ is the patch of elements associated with the vertex $i$. Introduce the space
\begin{equation*}
W_{\Omega_i} = \left\{ 
\begin{array}{c}
\bp \in H(\dive,K) \ \ \text{and} \ \ \dive \bp + r_K \, \phi_i =0 \ \ \text{on each $K \subset \Omega_i$}, 
\\ \noalign{\vskip 3pt}
\bp \cdot \bn = 0 \ \text{on $\partial \Omega_i \setminus (\Gamma_D \cup \Gamma_N)$}, 
\quad
\bp \cdot \bn = - t_\Gamma \, \phi_i \ \text{on $\partial \Omega_i \cap \Gamma_N$},
\\ \noalign{\vskip 3pt}
[[ \bp\cdot \bn ]] = - t_\Gamma \, \phi_i \ \text{on $\Gamma_{\rm int} \subset \Omega_i$}
\end{array}
\right\},
\end{equation*}
where $r_K$ and $t_\Gamma$ are defined by~\eqref{eq:def_rK_tGamma}, $[[ \ \cdot \ ]]$ is the jump across an edge and where $\{ \phi_i \}$ are the piecewise affine basis functions. Note that we do not impose any condition on $\bp \cdot \bn$ on $\partial \Omega_i \cap \Gamma_D$. For a vertex $i$ inside $\Omega$, the function $\phi_i$ vanishes on $\partial \Omega_i$, and thus the boundary condition on $\partial \Omega_i \cap \Gamma_N$ simply reads $\bp \cdot \bn = 0$. However, for a vertex $i$ on the boundary of $\Omega$, the function $\phi_i$ does not identically vanish on $\partial \Omega_i$.

A finite dimensional subspace $\widetilde{W}_{\Omega_i} \subset W_{\Omega_i}$ is next introduced and the following local problem is solved:
\begin{equation}\label{pbfluxfreedual}
\widehat{\bp}_i = \text{argmin} \left\{ J_{2|\Omega_i}(\bp), \quad \bp \in \widetilde{W}_{\Omega_i} \right\}.
\end{equation}
We then set $\widehat{\bq}_h = \bq_h + \sum_i \widehat{\bp}_i$ which is an element of $W$ (in view of the definition of $W_{\Omega_i}$). In practice, $\widetilde{W}_{\Omega_i}$ is defined using polynomial functions on each element, and constraints associated with this space are imposed using Lagrange multipliers. 

\begin{remark}
To solve~\eqref{pbfluxfreedual} in practice, it may be convenient to change variable and to work with $\widehat{\bp}^\ast_i = \widehat{\bp}_i + \phi_i \, \Aa\nab u_h$ rather than $\widehat{\bp}_i$. The field $\widehat{\bp}^\ast_i$ should then satisfy
\begin{gather*}
\dive \widehat{\bp}^\ast_i + f \phi_i - \Aa\nab u_h \cdot \nab \phi_i = 0 \ \ \text{on each $K \subset \Omega_i$}, \\
\widehat{\bp}^\ast_i\cdot \bn = 0 \ \ \text{on $\partial \Omega_i \setminus (\Gamma_D \cup \Gamma_N)$}, \qquad
\widehat{\bp}^\ast_i\cdot \bn = g \, \phi_i \ \ \text{on $\partial \Omega_i \cap \Gamma_N$}, \\
[[ \, \widehat{\bp}^\ast_i \cdot \bn \, ]] = 0 \ \ \text{on $\Gamma_{\rm int} \subset \Omega_i$}.
\end{gather*}
We eventually set $\widehat{\bq}_h = \sum_i \widehat{\bp}^\ast_i$, and we have $\widehat{\bq}_h \in W$. 
\end{remark}

In~\cite{PAR17}, a variant of the original flux-free technique was investigated, with the goal of substantially reducing the computational cost while retaining a good quality of the bounds. 

\subsubsection{Full polynomial construction using a dual mesh}

We now describe yet another method to build equilibrated flux fields $\widehat{\bq}_h \in W$ (see~\cite{ERN15} for an overview), based on a mixed formulation and the use of Raviart-Thomas-N\'ed\'elec (RTN) elements~\cite{BOF13}. This method has been used in e.g.~\cite{DES99,LUC05,ERN10,VOH08,VOH11}, where a simple construction of $\widehat{\bq}_h$ is proposed using local computations (element by element construction) and low-order RTN elements.

\medskip

Assume that $u_h$ is a piecewise affine function and that $\dive \Aa \nab u_h = 0$ in each element $K$ (this is for instance the case when $\Aa$ is constant in each $K$). Changing of unknown function, consider $\widehat{\bp}_K = \widehat{\bq}_{h|K}-\bq_{h|K}$ over each element $K$. We wish to build $\widehat{\bp}_K$ such that
\begin{equation}
\label{eq:helas2}
\dive \widehat{\bp}_K + f = 0 \quad \text{in $K$}.
\end{equation} 
We decompose $f$ in $K$ as
\begin{equation*}
f = \overline{f}_K + \delta f \qquad \text{with} \qquad \overline{f}_K = \frac{1}{K} \int_K f \qquad \text{and} \qquad \int_K \delta f = 0. 
\end{equation*}
Denoting by $\bx_K$ the geometric center of $K$, we observe that
\begin{equation*}
\widehat{\bp}_K = - \frac{\bx-\bx_K}{d} \ \overline{f}_K + \rott \boldsymbol{\psi}_K + \nab \tau_K 
\end{equation*}
is a solution to~\eqref{eq:helas2}, where $\boldsymbol{\psi}_K$ is an arbitrary function in $\left( H^1(K) \right)^d$ and $\tau_K \in H^1(K)$ is such that $-\Delta \tau_K = \delta f$. We complement the previous equation on $\tau_K$ by homogeneous Neumann boundary conditions $\nab \tau_K \cdot \bn = 0$ on $\partial K$ as well as the condition $\dis \int_K \tau_K = 0$. Since the integral over $K$ of $\delta f$ vanishes, the above problem is well-posed and $\tau_K$ is well-defined. An analytical solution can be found using the Green kernel.

\begin{remark}
Note that $\tau_K = 0$ when $f$ is constant in $K$. Furthermore, if $f\in H^1(K)$, we have, using the Poincar\'e Wirtinger inequality on $f$ and $\tau_K$, that $\|\delta f\|_{0,K} \le Ch \|f\|_{1,K}$, which implies that $\|\nab \tau_K\|_{0,K} \le Ch^2 \|f\|_{1,K}$. We hence see that, for small $h$, the quantity $\nab \tau_K$ is small in comparison to the error $e$ (which is expected to be of the order of $h$ since $u_h$ is a piecewise affine function), so that the computation of $\tau_K$ may not be needed.

More generally, if $f \not\in \oplus_K H^1(K)$, we see that the computation of $\tau_K$ is needed only in the elements $K$ where $f$ does not belong to $H^1(K)$.
\end{remark}

The choice of $\boldsymbol{\psi}_K$ is guided by the fact that normal fluxes of the field $\widehat{\bq}_{h|K} = \widehat{\bp}_K + \bq_{h|K}$ must be continuous. In view of the homogeneous Neumann boundary condition on $\tau_K$, the set of functions $\boldsymbol{\psi}_K \in \left( H^1(K) \right)^d$ should satisfy, on all element edges,
\begin{equation} \label{eq:mathieu4}
[[ \rott \boldsymbol{\psi}_K ]] \cdot \bn + \left[\left[ \bq_{h|K} - \frac{\bx-\bx_K}{d} \, \overline{f}_K \right]\right] \cdot \bn = 0.
\end{equation}
Following~\cite{DES99b} and restricting our presentation to the two-dimensional case, we observe that $\rott \boldsymbol{\psi}_K \cdot \bn_K = - \nab_{\bs_K} \Psi_K$, where $\bn_K$ is the outgoing normal vector to $\partial K$, $\bs_K$ is the curvilinear coordinate along $\partial K$ (oriented in the trigonometric, counterclockwise sense) and $\Psi_K \in H^1(K)$ is such that $\dis \rott \boldsymbol{\psi}_K = \left( \begin{array}{c} \partial_2 \Psi_K \\ -\partial_1 \Psi_K \end{array} \right)$. A possibility to ensure~\eqref{eq:mathieu4} is to choose $\Psi_K \in H^1(K)$ such that, along each edge in $\partial K$,
\begin{equation} \label{eq:mathieu5}
\nab_{\bs_K} \Psi_K = \bq_{h|K} \cdot \bn_K - \frac{(\bx-\bx_K) \cdot \bn_K}{d} \, \overline{f}_K.
\end{equation}
A necessary condition for the existence of such a function $\Psi_K$ is of course that the integral over $\partial K$ of $\bq_{h|K} \cdot \bn_K - (\bx-\bx_K)\cdot \bn_K \, \overline{f}_K/d$ vanishes, because we need to ensure that $\dis \int_{\partial K} \nab_{\bs_K} \Psi_K = 0$. To that aim, the function $u_h$ (and thus $\bq_h$) may have to be locally modified (see~\cite[p.~207]{DES99b}). Under that assumption, and noticing that the right-hand side of~\eqref{eq:mathieu5} is constant on each edge, there exists a unique affine function $\Psi_K$ with vanishing average over $K$ and satisfying~\eqref{eq:mathieu5} (the expression of which can of course be given in closed form). We eventually observe that, up to the contribution $\nabla \tau_K$, the function $\widehat{\bp}_K$ (and also $\widehat{\bq}_h$ if $\Aa$ is piecewise constant) belongs to the lowest-order RTN element.

\begin{remark}
Higher-order RTN finite-dimensional subspaces of $H(\dive,\Omega)$ can also be constructed. They are useful to construct equilibrated fields for elasticity problems~\cite{ARN84,BRA08,NIC08}, where the situation is more difficult due to the symmetry of the stress tensor. They are denoted $RTN_k(\mT_h)$ and defined as
\begin{equation*} 
RTN_k(\mT_h) = \{\bp_h \in H(\dive,\Omega); \quad \bp_{h|K} \in RTN_k(K) \ \ \text{for any $K$} \},
\end{equation*}
where $RTN_k(K) = [P_k(K)]^d +\bx \, P_k(K)$. In particular, any function $\bp_h \in RTN_k(\mT_h)$ is such that $\dive \bp_h \in P_k(K)$ for all $K$ and $\bp_h \cdot \bn \in P_k(\Gamma)$ for all edges $\Gamma$.
\end{remark}

\subsection{Numerical illustration}\label{section:illustration}

We report in this section on a numerical experiment borrowed from~\cite{PLE11}. A linear elasticity problem is considered on a three-dimensional physical domain (which represents the hub of the main rotor of an helicopter). The mechanical structure is clamped on part of its boundary, and is subjected to a unit traction force density $\boldsymbol{g}$, normal to the boundary surface, on another part of its boundary. The loading plane is not exactly orthogonal to the main axis of the structure. The considered geometry and mesh, made of 1978 linear tetrahedral elements (17694 dofs), are shown in Figure~\ref{fig:3Dmesh}. 

\begin{figure}[H]
\begin{center}
\includegraphics[width=120mm]{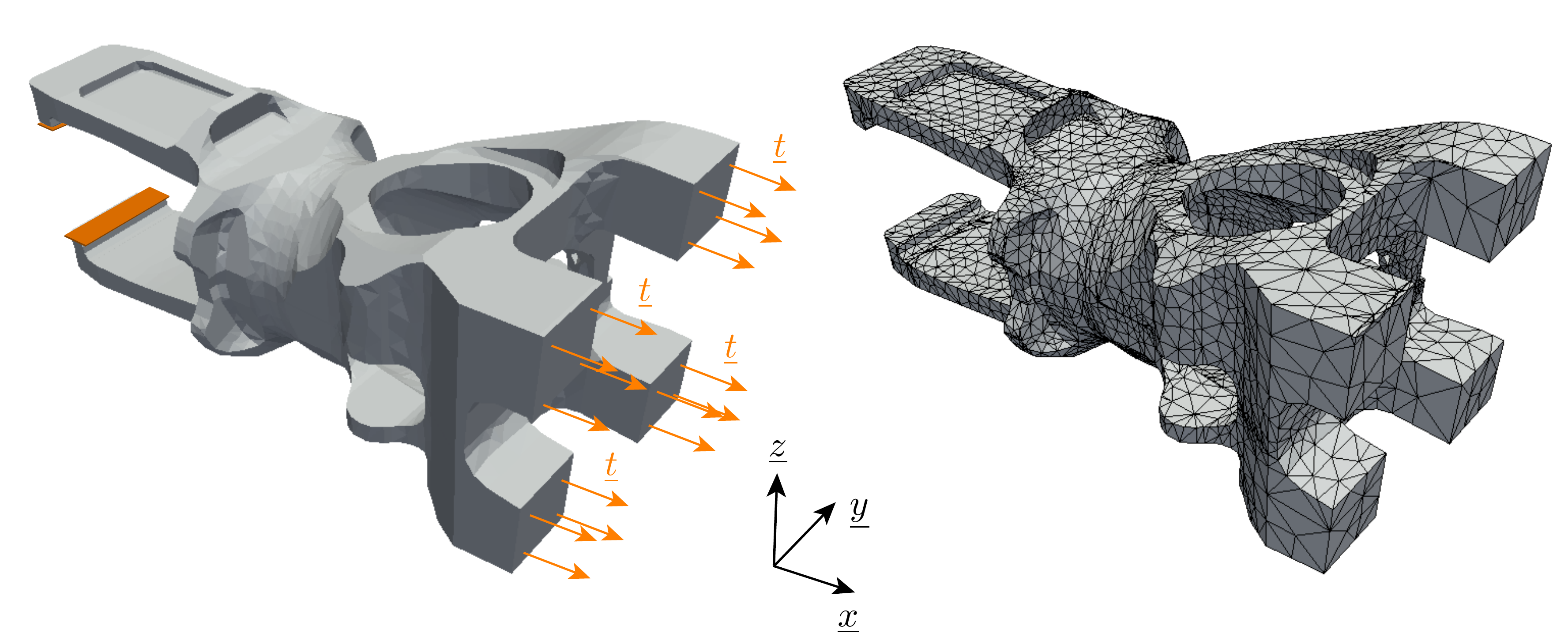}
\end{center}
\caption{Model problem (left) and associated FE mesh (right). The colored planes represent clamped boundary parts. \label{fig:3Dmesh}}
\end{figure}

A conforming FEM using a direct solver with Cholesky factorization is used to compute an approximate solution. In this example, a specific region of the structure plays an essential role for design purposes and engineering interest. The magnitude (in Frobenius norm) of the FE stress field in this region is shown on Figure~\ref{fig:3Dstress}.

\begin{figure}[H]
\begin{center}
\includegraphics[width=100mm]{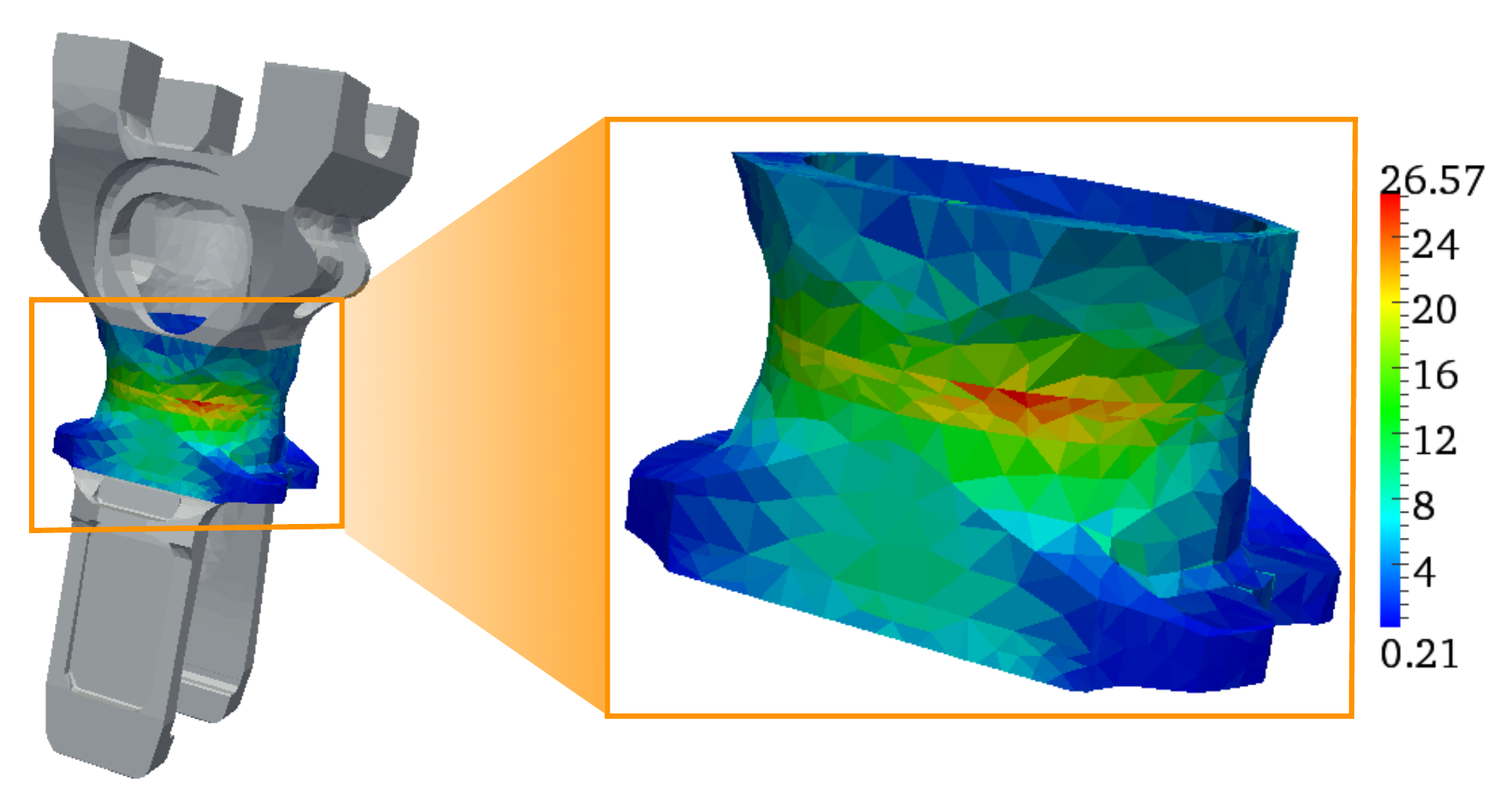}
\end{center}
\caption{Magnitude of the FE stress field in the region of interest. \label{fig:3Dstress}}
\end{figure}

A posteriori error estimation is performed using the CRE concept presented in Section~\ref{section:CRE}. The magnitude of the admissible stress field, obtained using hybrid-flux equilibration techniques, as well as the spatial distribution of the obtained CRE error estimate, are shown on Figure~\ref{fig:3Derror}.

\begin{figure}[H]
\begin{center}
\includegraphics[width=50mm]{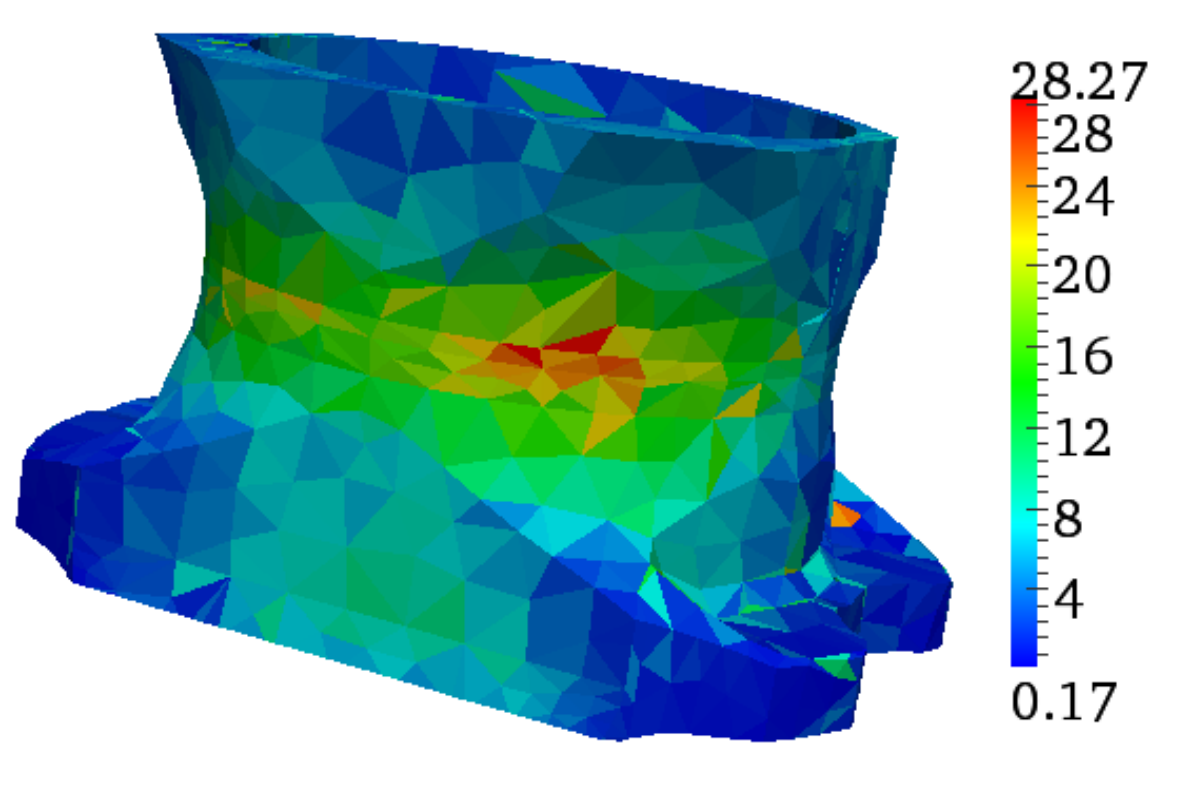} \hspace{5.em}
\includegraphics[width=50mm]{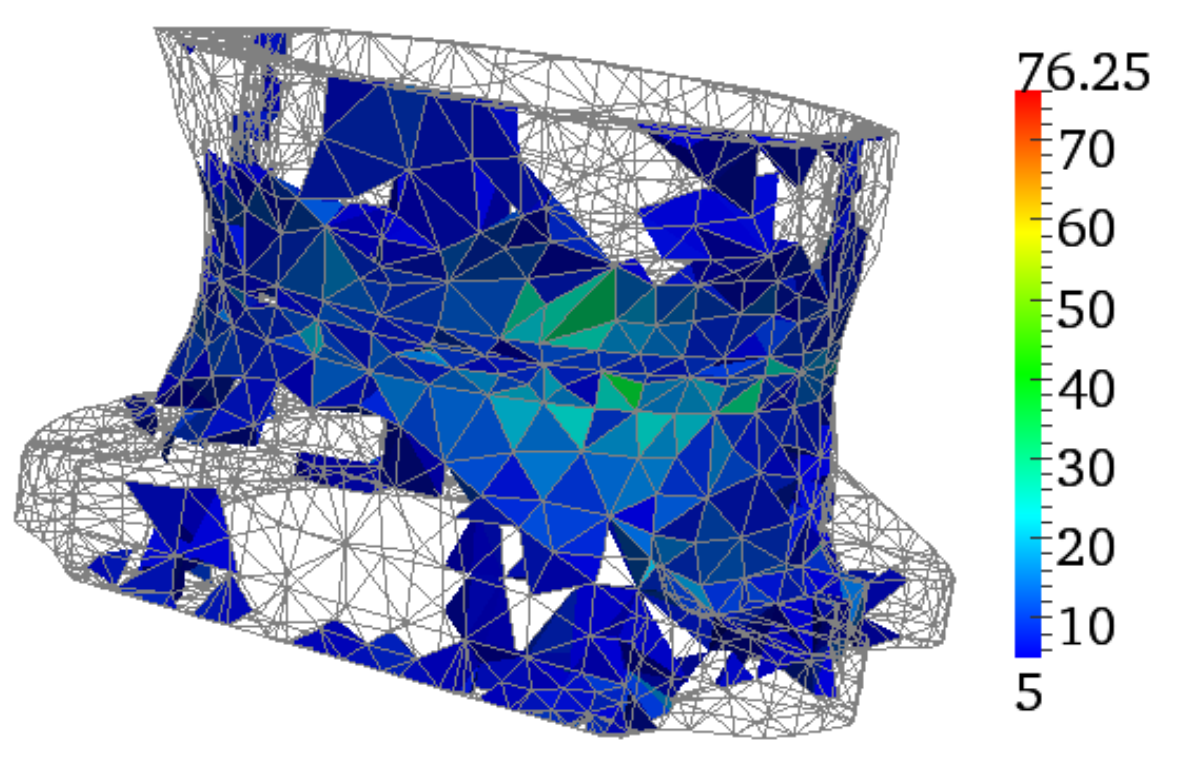}
\end{center}
\caption{Magnitude of the admissible stress field (left) and spatial distribution of relevant local contributions to the error estimate (right). \label{fig:3Derror}}
\end{figure}

\section{Unified point of view on \textit{a posteriori} error estimation methods}\label{section:connexions}

Among the error estimation methods presented in Sections~\ref{section:recovery}, \ref{section:residuals} and~\ref{section:CRE}, and even though the philosophies behind these methods are pretty different, connections can be easily found around the concept of equilibrium (see~\cite{CAR12} where such a unified perspective is also described). In the duality-based approach, this concept naturally appears to be fundamental to recover, from a statically admissible flux field $\widehat{\bq}$, an upper bound on the global error $\vertiii{e}$ measured in the energy norm (see Section~\ref{section:dualapp}).

This link between equilibrium and guaranteed upper error bounds is also exhibited in advanced recovery methods (see Section~\ref{section:advancedZZ}): choosing in~\eqref{eq:boundZZ} a field $\bq_h^\ast$ that is fully equilibrated leads to $\dive \bq^\ast_h + f = 0$ and yields an upper error bound $\vertiii{\bq_h^\ast-\bq_h}_q$ which is equivalent to the one given by the constitutive relation error functional.

\medskip

In addition, it is noticeable that the constitutive relation functional can also be understood from the residual functional. Indeed, for any flux $\bp \in H(\dive,\Omega)$ and any $v \in V$, we have
\begin{equation*} %\label{eq:Rdev}
R(v) = B(e,v) = \intO (\bq-\bq_h) \cdot \nab v = \intO (\bp-\bq_h) \cdot \nab v + \intO (f + \dive \bp) \, v + \int_{\Gamma_N} (g-\bp\cdot \bn) \, v,
\end{equation*}
with $\bq = \Aa \nab u$ and $\bq_h = \Aa \nab u_h$. Writing the above equation for $\bp = \widehat{\bq} \in W$ and $v = e$, we get
\begin{equation*}
R(e) = \vertiii{e}^2 = \intO (\widehat{\bq}-\bq_h) \cdot \nab e \leq \vertiii{\widehat{\bq}-\bq_h}_q \, \vertiii{e},
\end{equation*}
and thus $\vertiii{e} \le \vertiii{\widehat{\bq}-\bq_h}_q = \sqrt{2} \, E_{\rm CRE}(u_h,\widehat{\bq})$. Consequently, strong similarities can be found between duality-based methods and some of the implicit residual methods (i.e. those providing upper error bounds) described in Section~\ref{sec:implicit}. We illustrate these similarities for hybrid-flux equilibration techniques first, and second for flux-free techniques. 

\medskip

First, the hybrid-flux equilibration technique presented in Section~\ref{section:hybridfluxapp} is used to construct tractions $\widehat{g}_\Gamma$ in the equilibration procedure of the element residual method (see Section~\ref{section:elementresidualmeth}), in which the local problem to solve over each element $K$ reads (see~\eqref{eq:retour3})
\begin{multline*}
\text{Find $\widehat{e}_K \in V(K)$ such that, for any $v \in V(K)$,} \\ B_K(\widehat{e}_K,v) = \int_K r_K \, v + \sum_{\Gamma \subset \partial K \setminus \Gamma_D} \int_\Gamma \widehat{R}_{\Gamma,K} \, v, 
\end{multline*}
with $r_K = f + \dive (\Aa\nab u_h)_{|K}$ and $\widehat{R}_{\Gamma,K} = \sigma_{\Gamma,K} \, \widehat{g}_\Gamma - \Aa \nab u_{h|K} \cdot \bn_K$. Consequently, and except for elements $K$ connected to $\Gamma_D$, the problem~\eqref{eq:retour3} is similar to the one given in~\eqref{eq:dualelempb}. In other words, from the solutions $\widehat{e}_K$ to~\eqref{eq:retour3}, an equilibrated flux field $\bq_h^\ast$ such that $\bq_{h|K}^\ast=\Aa\nab (\widehat{e}_K+u_{h|K})$ is obtained. In practice, solving~\eqref{eq:retour3} or~\eqref{eq:dualelempb} with an enriched FE method (polynomial functions of degree $p+k$) leads to the same error bounds. Alternatively, introducing the local space
\begin{equation*}
W_K=\{\bp \in H(\dive,K), \quad \dive \bp + r_K = 0 \; \text{in $K$}, \quad \bp \cdot \bn = \widehat{R}_{\Gamma,K}\; \text{on $\Gamma \subset \partial K$} \},
\end{equation*}
an equilibrated flux field $\widehat{\bq}_h \in W$ can be constructed as $\widehat{\bq}_{h|K}=\widehat{\bp} + \Aa\nab u_{h|K}$ with $\widehat{\bp} \in W_K$~\cite{SAU04}. This is the approach followed when solving~\eqref{eq:retour3} with a dual approach, or when solving~\eqref{eq:localpbelement} with polynomial flux functions.

\medskip

Second, problems~\eqref{pbfluxfreedual} stemming from the flux-free technique presented in Section~\ref{section:fluxfreeapp} are the dual problems to the local problems~\eqref{eq:locpbparchphii2}, that are to be solved over each patch in the subdomain residual method. Solving~\eqref{eq:locpbparchphii2} with an enriched FE discretization (we denote $\widetilde{z}_i$ its solution) provides for a flux $\sum_i \Aa \nab \widetilde{z}_i + \bq_h$ which is (up to the small numerical error introduced when solving~\eqref{eq:locpbparchphii2}) statically admissible.

%------------------------------------------------------------------
%------------------------------------------------------------------

\section{Mesh adaptation}\label{section:adaptivity}

Motivated by the fact that computational resources are limited, mesh-adaptive finite element algorithms have been developed. The adaptive procedures aim at automatically refining (or coarsening) a mesh and/or the FE basis (using polynomial functions of various degrees) to obtain a numerical solution having a specified accuracy while requiring the smallest possible computational cost~\cite{DIE99,DOR96,ONA93,VER94}. They may lead to non-conforming discretizations. They are usually based on local contributions $\eta_K$ of \textit{a posteriori} error estimates, even though other indicators providing guidance on where refinement should occur may be considered (the use of solution gradients is popular in fluid dynamics where error estimates are not readily available). Even though adaptive procedures have been applied for a long time, optimal strategies for deciding where and how to refine or change the basis are rare, and proofs of the optimality of enrichment in the vicinity of the largest error contributions have been investigated only recently~\cite{BIN04,CAS08,MOR02,STE07}.

\medskip

By far, $h$-adaptivity is the most popular strategy (see~\cite{FUE96} for an overview and comparison between various $h$-adaptivity strategies). Conserving element type and interpolation, its interest comes from the fact that it can increase the asymptotic convergence rate, particularly when singularities are present. Indeed, \textit{a priori} error analysis shows that a FE method based on polynomial functions of degree $p$ used on a uniform mesh of size $h$ converges as $\vertiii{e} \le C_1(u,p) \, h^{\min(p,q)} = C_2(u,p) \, N^{-\min(p,q)/d}$ where $q>0$ depends on the solution smoothness ($u$ should belong to $H^{q+1}(\Omega)$ for the estimate to hold) and $N$ is the number of degrees of freedom. In contrast, the optimal adaptive $h$-refinement restores the optimal convergence $\vertiii{e} \le C_1(u,p) \, h^p = C_2(u,p) \, N^{-p/d}$.

The simplest $h$-adaptivity approach is $h$-refinement that consists in dividing some selected elements, keeping other elements unchanged (as performed for instance in quadtree/octree meshes). The adaptive process is based on the cycle
\begin{equation*}
\boxed{
\text{solve} \rightarrow \text{estimate} \rightarrow \text{mark} \rightarrow \text{refine}
}
\end{equation*}
and refines a fixed percentage of elements having the largest error indicators. According to a marking threshold of the form $\dis \theta= \lambda \max_j \eta_j$, where $\lambda \in [0,1]$ is a parameter (a typical choice is $\lambda=0.8$), elements $K$ such that $\eta_K \ge \theta$ are marked. The refinement level (namely the number of small elements in which a selected element is divided) is then determined trying to equi-distribute discretization error contributions among the new elements. To that aim, the total number of elements in the refined mesh needs to be estimated. This is done from asymptotic convergence rate and prescribed error tolerance $\epsilon_0$. A drawback of the $h$-refinement method is that it may lead to non-conforming meshes with hanging nodes.

Another more flexible approach, which also enables for mesh unrefinement, is $h$-remeshing. In this approach, the mesh is totally rebuilt following a remeshing map based on \textit{a posteriori} error estimates. It thus requires a mesh generator able to respect a given remeshing map, which is the main drawback of this approach. To build this new mesh, one possibility is to set the problem as a constrained minimization, in order to define an optimal mesh minimizing the number of elements while ensuring a given accuracy tolerance $\epsilon_0$.

Denoting $r_K=h_K^\star/h_K$ the local refinement ratio in the element $K$, where $h_K^\star$ is the size of the small elements $K^\star \subset K$ after refinement, and noticing that $\eta_K^\star/\eta_K \approx r_K^{\alpha_K}$ with $\dis \eta_K^\star = \sqrt{\sum_{K^\star \subset K} (\eta_{K^\star})^2}$ and $\alpha_K$ the local convergence rate predicted by \textit{a priori} error estimation, the problem reads as
\begin{equation}
  \label{eq:cons-min}
\min_{r_K} \ \sum_{K \in \mT_h} \frac{1}{r_K^d} \quad \text{under the constraint} \quad \sum_{K \in \mT_h} (\eta_K^\star)^2 = \sum_{K \in \mT_h} r_K^{2\alpha_K} \, \eta_K^2 = \epsilon_0^2.
\end{equation}
When the exact solution is smooth, we have $\alpha_K=p$ for all elements $K$, and the constrained minimization problem can be solved analytically. It yields 
\begin{equation*}
r_K = \frac{\epsilon_0^{1/p}}{\eta_K^{2/(2p+d)}\left[\sum_{K \in \mT_h} \eta_K^{2d/(2p+d)}\right]^{1/2p}}.
\end{equation*}
If the exact solution is not smooth enough, then $\alpha_K<p$ in some given zones. The standard strategies then consist in first empirically determining the rates $\alpha_K$, and next numerically solving the constrained minimization problem~\eqref{eq:cons-min} (e.g. by using an iterative algorithm). It can be shown that this $h$-remeshing procedure leads to an evenly-distributed accuracy of the approximate solution over the whole computational domain.

\medskip

As an illustration, we show in Figure~\ref{fig:meshref} some adaptivity results obtained on a two-dimensional mechanical structure when using $h$-remeshing. The obtained sequence of meshes enables one to progressively decrease the relative (global, i.e. in energy norm) discretization error, starting from $73\%$ and reaching $5\%$.

\begin{figure}[H]
\begin{center}
\includegraphics[width=110mm]{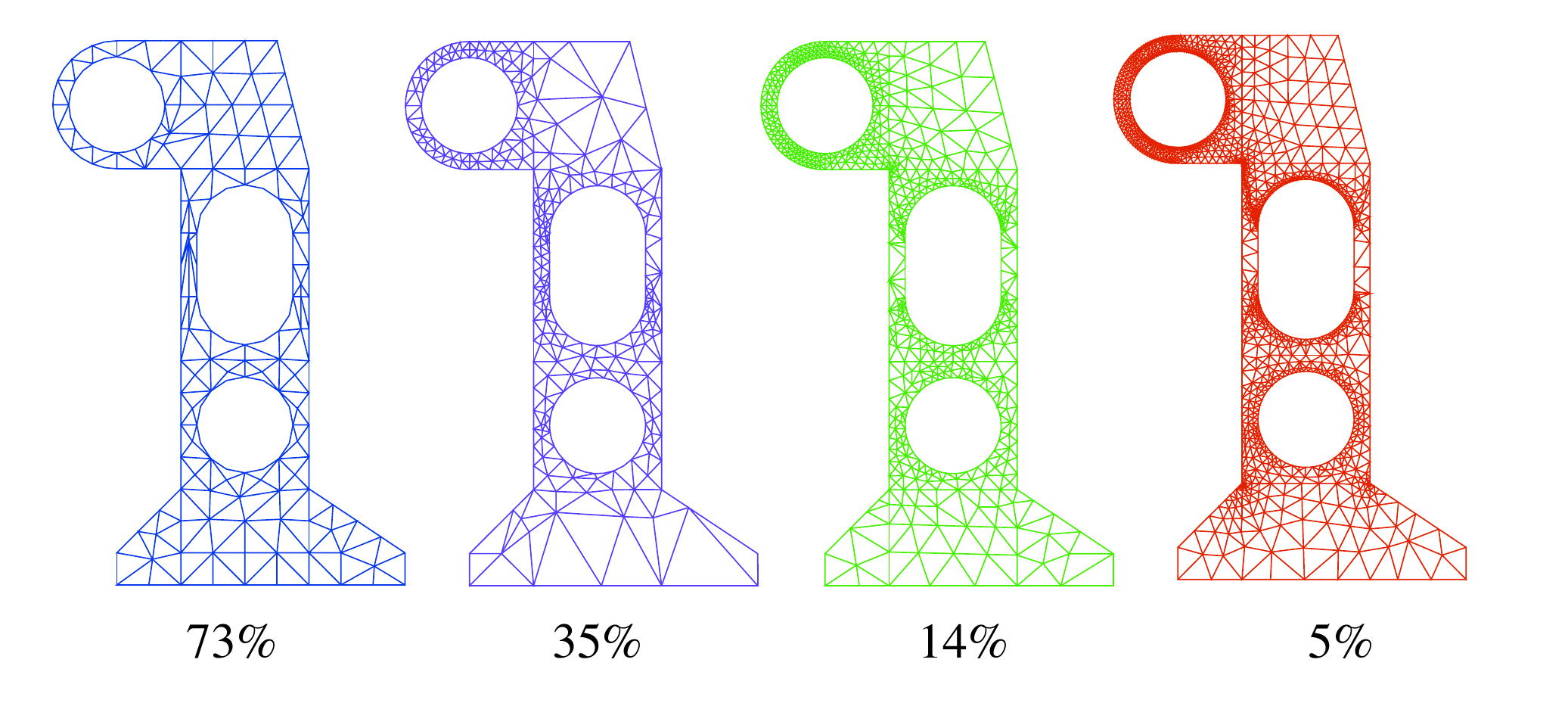}
\end{center}
\caption{Sequence of meshes using a $h$-remeshing procedure, with values of the relative error estimate. \label{fig:meshref}}
\end{figure}

Another possibility for adaptivity, which is more powerful when the solution $u$ is smooth, is to perform $p$-refinement. If the solution $u$ is $C^\infty$, then it allows for an exponential convergence rate. Otherwise the convergence rate is limited by the solution smoothness. Again, this may yield a non-conforming discretization space. 

\medskip

Of course, $hp$-refinement strategies, combining $h$ and $p$ refinements, are also possible. They allow to recover an exponential convergence rate $\vertiii{e} \le Ce^{-q_1 N^{q_2}}$, where $q_1$ and $q_2$ are positive constants that depend on the smoothness of $u$ and on the FE mesh, even when singularities are present. This idea is (i) to increase the degree of polynomial basis functions in regions where the solution $u$ is smooth and (ii) to decrease $h$ near singularities of $u$.

\begin{remark}
Though being well understood and available in various academic codes, adaptive methods with \textit{a posteriori} error control have so far not become fully widespread in commercial softwares. One of the reasons for this shortcoming is probably the complex task of mesh generation, in particular for 3D domains where the practitioners have little trust in automatic refinement procedures.
\end{remark}

% ---------------------------------------------------------------------
% ---------------------------------------------------------------------

\section{Goal-oriented error estimation}\label{section:goaloriented}

In practical applications, scientists and engineers are often interested in errors on some specific outputs of the computation, i.e. quantities of interest $Q(u)$ which are functionals of the solution $u$ (typical examples include the average of the normal flux $q_n = \bq \cdot \bn = (\Aa\nab u) \cdot \bn$ over a part $\Gamma_Q \subset \partial \Omega$ of the boundary of the domain, or the average of the solution $u$ in a local region of interest $\omega_Q \subset \Omega$). In such cases, energy-norm driven error estimation and mesh adaptation tools (presented in the above sections) fail to provide the required accuracy in the chosen quantity of interest using limited computational resources. We show in the following that, again using these tools but complementing them with additional information, goal-oriented error estimation can be performed and adaptive techniques can be extended to that setting. We refer to~\cite{AIN12,BAB84,CAO03,CIR98,GRA05,LAD99b,MOZ15,ODE01,OHN01,PAR97,PAR06b,PER98,PRU99,RAN97,STR00} for more details on what we present below. Such strategies are highly beneficial compared to global error estimation and the associated mesh adaptivity, since complex features of the solution on some parts of the domain might not influence the quantity of interest (and thus do not need to be accurately captured by the numerical solution). 

\medskip

In this section, we consider a linear and continuous functional $Q: V \to \RR$, and we define from this a quantity of interest $Q(u)$ (depending linearly on $u$). In practice, the functional $Q$ is often defined in a global manner as
\begin{equation} \label{eq:rep_Q}
Q(v) = \intO \left( \widetilde{f}_Q \, v + \widetilde{\bq}_Q \cdot \nab v \right) + \int_{\Gamma_N} \widetilde{g}_Q \, v + \intO \Aa \nab \widetilde{u}_Q \cdot \nab v,
\end{equation}
where $\left(\widetilde{f}_Q, \widetilde{\bq}_Q, \widetilde{g}_Q, \widetilde{u}_Q\right)$ is a set of extraction functions which can be mechanically interpreted in an adjoint problem (see below) as body, pre-flux, traction, and pre-primal field loadings, respectively. These are defined explicitly or implicitly, depending on the quantity $Q$. The last extraction function $\widetilde{u}_Q$, vanishing on $\Gamma_N$, is a regular field that enables one to extract components of the normal flux $\bq \cdot \bn$ on $\Gamma_D$ (reaction forces) by imposing a non-homogeneous Dirichlet condition. More precisely, considering momentarily the case when $\widetilde{f}_Q$, $\widetilde{\bq}_Q$ and $\widetilde{g}_Q$ vanish, evaluating the quantity of interest on the solution $u$ to~\eqref{eq:refpbstrong} and using that $\widetilde{u}_Q$ vanishes on $\Gamma_N$, we compute that $\dis Q(u) = \intO \Aa \nab \widetilde{u}_Q \cdot \nab u = \int_{\Gamma_D} \widetilde{u}_Q \, \bq \cdot \bn + \intO \widetilde{u}_Q \, f$. Up to the second term in $Q(u)$ (which is independent of the solution), we indeed see that $\widetilde{u}_Q$ allows to extract components of the normal flux $\bq \cdot \bn$ on $\Gamma_D$.

\medskip

The quantity of interest $Q(u)$ can be written $\langle Q, u \rangle$ where $Q$ belongs to the dual space $V^\star$ of $V$. Consequently, we have
\begin{equation*} 
\left|Q(u)-Q(u_h)\right| = |Q(e)| \le \| Q \|_\star \ \vertiii{e},
\end{equation*}
where $\dis \|Q\|_\star = \sup_{v\in V, \, \vertiii{v}=1} |Q(v)|$ is the dual norm of $Q$. A target accuracy $\epsilon$ for the error on $Q(u)$ can thus be achieved by ensuring that $\vertiii{e} \le \epsilon/ \|Q\|_\star$. However, such a method based on a mesh adaptation with respect to the error in the energy norm usually leads to an overestimation (and requires useless computational effort) since norm approximation does not take the locality of $Q$ into account. We show in the following, using duality arguments, that a target accuracy $\epsilon$ for the error on $Q(u)$ can be achieved by only requesting that $\vertiii{e}$ is of the order of $\sqrt{\epsilon}$.

\subsection{Adjoint problem and error indicator}

Considering in a general setting the problem written in its weak form~\eqref{eq:refpb}, and in order to give an exact representation of the error $Q(u)-Q(u_h)=Q(e)$, the following adjoint problem is introduced~\cite{BEC01,GIL02}:
\begin{equation}
\label{eq:adjoint}
\text{Find $\widetilde{u} \in V$ such that, for any $v \in V$,}\quad B^\ast(\widetilde{u},v) = Q(v),
\end{equation}
where $B^\ast(u,v)$ is the formal adjoint of the primal form $B(u,v)$, defined by $B^\ast(u,v)=B(v,u)$. In the current symmetric case, the forms $B$ and $B^\ast$ are identical. We also define the adjoint flux $\widetilde{\bq}=\Aa \nab \widetilde{u}$.

\medskip

The adjoint solution $\widetilde{u}$ yields the exact representation (primal-dual equivalence) of the error on $Q(u)$:
\begin{equation*}
Q(e) = B^\ast(\widetilde{u},e) = B(e,\widetilde{u}) = R(\widetilde{u}) = \langle R_{|u_h},\widetilde{u}\rangle.
\end{equation*}
This shows that $\widetilde{u}$ provides for the sensitivity of the discretization error on $Q(u)$ to the local sources of discretization errors in the whole domain $\Omega$. The adjoint solution, often termed generalized Green's function~\cite{EST05}, conveys the locality of the targeted information. It indicates the influence of parts of the primal solution away from the spatial location of the target functional. 

\begin{remark}
In the case of a nonlinear quantity of interest, the general approach consists in performing a first order linearization $Q(u)-Q(u_h) \approx Q'_{|u_h}(e)$ and therefore considering $Q'_{|u_h}(v)$ as the right-hand side of the adjoint problem~\eqref{eq:adjoint}. Another approach that avoids linearization and thus keeps strict error bounds for nonlinear quantities of interest is presented in~\cite{LAD10a}.
\end{remark}

A classical numerical approach then consists in computing an approximate solution $\widetilde{u}_h$ of $\widetilde{u}$ using the reference mesh $\mT_h$. We thus get, using the Galerkin orthogonality, that
\begin{equation}
\label{eq:rep}
Q(e)=B(e,\widetilde{u}-\widetilde{u}_h)=B(e,\widetilde{e}).
\end{equation}
When a different mesh (e.g. with a local enrichment) is used to solve the adjoint problem, yielding an approximate adjoint solution $\widetilde{u}_+$, the relation~\eqref{eq:rep} is changed in
\begin{equation}
\label{eq:rep2}
Q(e)=B(e,\widetilde{u}-\widetilde{u}_+) + B(e,\widetilde{u}_+),
\end{equation}
where the second term $B(e,\widetilde{u}_+) = R(\widetilde{u}_+) = F(\widetilde{u}_+)-B(u_h,\widetilde{u}_+)$ is a fully computable correcting term.

\begin{remark}
A local enrichment technique for the solution of the adjoint problem is proposed in~\cite{CHA08,GRA03}, using analytical or pre-computed numerical functions as an additive term to a classical FE solution computed on the mesh $\mT_h$. This technique, referred to as non-intrusive in the sense that the global mesh is unchanged between the primal and the adjoint problems, is particularly interesting in the case of pointwise quantities of interest such as $u(\bx_0)$ for some $\bx_0 \in \Omega$ (assuming that the reference solution $u$ is continuous at point $\bx_0$). In such a case, analytical Green's functions can be introduced as an enrichment. This is an alternative to the method proposed in~\cite{PRU99}, based on using a regularized functional instead of the original pointwise functional. The regularized functional yields a modified quantity of interest that involves a smooth function $k_{\delta,\bx_0}$ (characterized by the length-scale $\delta >0$ and such that $\dis \intO k_{\delta,\bx_0}=1$) which performs a weighted average over a small neighborhood of the point $\bx_0$.
\end{remark}

An error indicator may be derived from~\eqref{eq:rep2} using the dual weighted residual (DWR) method proposed in~\cite{BEC96} and based on hierarchical \textit{a posteriori} error analysis. After computing an approximate adjoint solution $\widetilde{u}_+$ from a hierarchically refined FE space $V_+ \supset V_h^p$ (higher-order basis functions and/or refined mesh), and assuming that the error between $\widetilde{u}$ and $\widetilde{u}_+$ is small, the DWR method then consists in writing
\begin{equation}\label{eq:DWRindicator}
Q(u)-Q(u_h) \approx R(\widetilde{u}_+).
\end{equation}
The quantity $R(\widetilde{u}_+)$ is then taken as an error indicator on $Q$. 

\begin{remark}
Due to the Galerkin orthogonality, choosing $V_+ = V_h^p$ would of course lead to a meaningless indicator (in that case, we would have $R(\widetilde{u}_+) = 0$). Consequently, a convenient approximation of the adjoint solution should involve a subspace in $(V_h^p)^\perp$.
\end{remark}

The above argument can be somewhat quantified. Let $u_+ \in V_+$ be the approximation (in the refined space) of the primal solution. We assume that the saturation assumption $|Q(u)-Q(u_+)| \le \beta \, |Q(u)-Q(u_h)|$ is satisfied for some $0 \le \beta < 1$. Using the fact that $u_+-u_h \in V_+$ and next the fact that $\widetilde{u}_+ \in V_+$, we obtain
$$
Q(u_+)-Q(u_h)=B(u_+-u_h,\widetilde{u}_+)=B(u-u_h,\widetilde{u}_+)=R(\widetilde{u}_+).
$$
From the triangle inequality, we eventually get
$$
\frac{|R(\widetilde{u}_+)|}{1+\beta} \le |Q(u)-Q(u_h)| \le \frac{|R(\widetilde{u}_+)|}{1-\beta},
$$
which quantifies the approximation~\eqref{eq:DWRindicator}. The indicator defined in~\eqref{eq:DWRindicator} may be used to drive adaptive algorithms since it gives a clear and accurate distribution of error sources, but it does not provide for {\em guaranteed} error bounds on $Q$. Alternative techniques addressing this question are proposed in the next section.

\subsection{Computation of upper error bounds}

In the following, we bound the error on $Q(u)$ using~\eqref{eq:rep}. A first possibility~\cite{AIN00} is to directly use the Cauchy-Schwarz inequality:
\begin{equation}\label{eq:localbound1}
|Q(u)-Q(u_h)|=|B(e,\widetilde{e})| \le \vertiii{e} \ \vertiii{\widetilde{e}} \le \eta \ \widetilde{\eta},
\end{equation}
where $\eta$ and $\widetilde{\eta}$ are upper error bounds (in the energy norm) for the primal and adjoint problems, respectively. This shows that any guaranteed global error estimate can be used to assess the error on $Q(u)$ and give a conservative bound. It also shows that the error on $Q(u)$ decreases with a rate which is twice larger than that for the error on $u$ in the energy norm (thus our above claim that an error of the order of $\sqrt{\epsilon}$ on $\vertiii{e}$, and similarly on $\vertiii{\widetilde{e}}$, is enough to yield an error of the order of $\epsilon$ on $Q(u)$). However, the above estimate may be crude (and may not exploit the locality of $Q(u)$) due to the use of the Cauchy-Schwarz inequality, even though improvement techniques may be found in~\cite{LAD13}.

It is also possible, still starting from~\eqref{eq:rep}, to use a local Cauchy-Schwarz inequality, writing that
\begin{equation}
\label{eq:adapt_goal1}
|Q(u)-Q(u_h)| = \left| \sum_{K \in \mT_h} B_K(e,\widetilde{e}) \right| \le \sum_{K \in \mT_h} \vertiii{e}_{|K} \ \vertiii{\widetilde{e}}_{|K}, 
\end{equation}
but bounds on $\vertiii{e}_{|K}$ and $\vertiii{\widetilde{e}}_{|K}$ (using for instance local error indicators $\eta_K$ and $\widetilde{\eta}_K$) are heuristic. In addition, a bound such as~\eqref{eq:adapt_goal1} is again very conservative, in the sense that it does not take into account possible cancellation of errors over the domain. 

\begin{remark}
A quantitative evaluation of $Q(e)$ involving the elementary components of $e$ and $\widetilde{e}$, but not providing for an upper error bound, can also be exhibited when using the DWR method and the approximate adjoint solution $\widetilde{u}_+ \in V_+$ obtained with a fine mesh and/or a high-order method (see~\eqref{eq:DWRindicator}). We can indeed write
\begin{equation*}
Q(u)-Q(u_h) \approx R(\widetilde{u}_+) = R(\widetilde{e}_+) = \sum_{K\in \mT_h} \left[ \int_K r_K \, \widetilde{e}_+ + \int_{\partial K} R_{\partial K} \, \widetilde{e}_+ \right],
\end{equation*}
where $\widetilde{e}_+ = \widetilde{u}_+ - \widetilde{u}_h$ is an approximation of $\widetilde{e}$, and where $R_{\partial K}$ is defined on $\partial K$ by $R_{\partial K}(\bx) = R_\Gamma(\bx)$ for any $\bx \in \Gamma \subset \partial K \setminus \Gamma_D$ and $R_{\partial K}(\bx) = 0$ for any $\bx \in \partial K \cap \Gamma_D$ (we recall that $r_K$ and $R_\Gamma$ are defined in Section~\ref{section:elementresidualmeth}). Consequently, an error indicator can be defined as
\begin{equation}\label{eq:adapt_goal2}
\left|Q(u)-Q(u_h)\right| \approx \sum_{K \in \mT_h} \rho_K \, \omega_K \quad \text{with} \quad
\left\{
\begin{array}{l}
\rho_K = \|r_K\|_{0,K} + h_K^{-1/2} \, \|R_{\partial K}\|_{0,\partial K},
\\ \noalign{\vskip 3pt}
\omega_K = \|\widetilde{e}_+\|_{0,K} + h_K^{1/2} \, \|\widetilde{e}_+\|_{0,\partial K}.
\end{array}
\right.
\end{equation}
Note that $\rho_K$ and $\omega_K$ are computable terms.
\end{remark}

An alternative technique consists in using the parallelogram identity to get an upper bound on the error~\cite{ODE01,PRU99}. Starting from the expression
\begin{equation} \label{eq:mathieu2}
B(e,\widetilde{e})
=
B\left(s \, e,\frac{1}{s} \, \widetilde{e}\right)
=
\frac{1}{4} \left[\vertiii{s \, e+\frac{1}{s} \, \widetilde{e}}^2-\vertiii{s \, e-\frac{1}{s} \, \widetilde{e}}^2\right]
=
\frac{1}{4} \left[\chi_+ - \chi_-\right],
\end{equation}
where $s \in \RR$ is a scaling factor (to be optimized) and
$$
\chi_+ = \vertiii{s \, e + \frac{1}{s} \, \widetilde{e}}^2, \qquad
\chi_- = \vertiii{s \, e - \frac{1}{s} \, \widetilde{e}}^2,
$$
we directly get 
\begin{equation}\label{eq:boundparallel}
\frac{1}{4} \left[\chi_+^{\rm low}-\chi_-^{\rm upp}\right] \le Q(u)-Q(u_h) \le \frac{1}{4} \left[\chi_+^{\rm upp}-\chi_-^{\rm low} \right],
\end{equation}
where $\chi_-^{\rm upp}$ and $\chi_-^{\rm low}$ (resp. $\chi_+^{\rm upp}$ and $\chi_+^{\rm low}$) are upper and lower bounds on $\chi_-$ (resp. $\chi_+$). These can be derived from the various approaches detailed in the above sections.

Even though the error bounds defined in~\eqref{eq:boundparallel} are more accurate than those given by~\eqref{eq:localbound1}, error cancellations over the domain are still not captured by the approach. It appears that there is no method in the literature to conciliate the two requirements (i)~take into account error cancellation (as done by the DWR method) and (ii)~get guaranteed error bounds, even though recent works such as~\cite{LAD13} have proposed first ideas in that direction.

\medskip

Another alternative bounding technique (see~\cite{CHA08,LAD08}) is based on the properties of the duality-based approaches presented in Section~\ref{section:CRE}. Introduce the space of equilibrated fluxes (with respect to the adjoint problem)
$$
\widetilde{W} = \left\{ \widetilde{\bp} \in H(\dive,\Omega) \quad \text{such that, for any $v \in V$}, \quad \intO \widetilde{\bp} \cdot \nab v = Q(v) \right\}.
$$
After constructing admissible flux fields $\widehat{\bq}_h \in W$ and $\widehat{\widetilde{\bq}}_h \in \widetilde{W}$ as a post-processing of $\bq_h$ and $\widetilde{\bq}_h$ (as detailed in Section~\ref{section:SA}), a direct consequence of~\eqref{eq:localbound1} and~\eqref{eq:upper} is
\begin{equation}\label{eq:CREgobound1}
\left| Q(u)-Q(u_h)\right| \le 2 \, E_{\rm CRE}(u_h,\widehat{\bq}_h) \, E_{\rm CRE}(\widetilde{u}_h,\widehat{\widetilde{\bq}}_h).
\end{equation}
An \textit{a posteriori} error estimate on $Q$ which is more accurate than~\eqref{eq:CREgobound1} can be obtained using the hypercircle property~\eqref{eq:propertiesCRE1} verified by $\dis \widehat{\bq}^m_h=\frac{1}{2}(\widehat{\bq}_h+\bq_h)$. Indeed, we infer from~\eqref{eq:rep} and the fact that $\widehat{\widetilde{\bq}}_h \in \widetilde{W}$ that
\begin{multline*}
Q(u)-Q(u_h) = B\left(u-u_h,\widetilde{u}-\widetilde{u}_h\right) 
= \intO \nab(u-u_h) \cdot \left(\widetilde{\bq}-\Aa \nab \widetilde{u}_h \right) 
\\ = \intO \nab(u-u_h) \cdot \left(\widehat{\widetilde{\bq}}_h-\Aa \nab \widetilde{u}_h \right).
\end{multline*}
We consequently obtain
\begin{multline*}
Q(u)-Q(u_h) = \intO \Aa^{-1}(\bq-\bq_h) \cdot \left(\widehat{\widetilde{\bq}}_h-\Aa \nab \widetilde{u}_h \right) 
\\ = \intO \Aa^{-1} \left( \bq-\widehat{\bq}^m_h \right) \cdot \left(\widehat{\widetilde{\bq}}_h-\Aa \nab \widetilde{u}_h \right) + C_h,
\end{multline*}
where $\dis C_h = \frac{1}{2} \intO \Aa^{-1} \left( \widehat{\bq}_h - \bq_h \right) \cdot \left(\widehat{\widetilde{\bq}}_h-\widetilde{\bq}_h\right)$ is a computable term. From the Cauchy-Schwarz inequality and the use of~\eqref{eq:propertiesCRE1}, we eventually obtain the bound
\begin{equation}\label{eq:boundERCopti}
|Q(u)-Q(u_h)-C_h| \le E_{\rm CRE}(u_h,\widehat{\bq}_h) \, E_{\rm CRE}(\widetilde{u}_h,\widehat{\widetilde{\bq}}_h),
\end{equation}
which is twice sharper than~\eqref{eq:CREgobound1}, and partially takes into account error cancellations (through the term $C_h$).

\medskip

The quantity $Q(u_h)+C_h$ can be interpreted as a corrected approximation of the quantity of interest. Furthermore, the estimate~\eqref{eq:boundERCopti} can also be written in the form of an error bound on $Q$ as
\begin{equation}\label{eq:goestimate}
|Q(u)-Q(u_h)| \le \eta^Q := \max_{\theta=\pm 1} \left| C_h + \theta \, E_{\rm CRE}(u_h,\widehat{\bq}_h) \, E_{\rm CRE}(\widetilde{u}_h,\widehat{\widetilde{\bq}}_h) \right|.
\end{equation}

\begin{remark}
It can be shown that the bound~\eqref{eq:boundERCopti} is equivalent to the one given by~\eqref{eq:boundparallel} when taking $\chi_-^{\rm low} = \chi_+^{\rm low} = 0$ and for the optimal value $\dis s_{\rm opt} = \sqrt{\frac{\vertiii{\widetilde{e}}}{\vertiii{e}}} \approx \sqrt{\frac{E_{\rm CRE}(\widetilde{u}_h,\widehat{\widetilde{\bq}}_h)}{E_{\rm CRE}(u_h,\widehat{\bq}_h)}}$, in the following sense. We infer from~\eqref{eq:rep} and~\eqref{eq:mathieu2}, with the lower bound $\chi_\pm \geq 0$, that, for any $s$,
\begin{equation} \label{eq:mathieu}
- \frac{1}{4} \, \chi_-(s) \le Q(u)-Q(u_h) \le \frac{1}{4} \, \chi_+(s),
\end{equation}
where $\dis \chi_\pm(s) = \vertiii{s \, e \pm \frac{1}{s} \, \widetilde{e}}^2 = s^2 \, \vertiii{e}^2 + \frac{1}{s^2} \, \vertiii{\widetilde{e}}^2 \pm 2\, B(e,\widetilde{e})$. The interval given by~\eqref{eq:mathieu} is the smallest for $s$ such that $\chi_+(s) + \chi_-(s)$ is the smallest, thus the choice $s = s_{\rm opt}$. The width of the interval in~\eqref{eq:mathieu} is then $\vertiii{e} \, \vertiii{\widetilde{e}}$, which is bounded from above by $2 \, E_{\rm CRE}(u_h,\widehat{\bq}_h) \, E_{\rm CRE}(\widetilde{u}_h,\widehat{\widetilde{\bq}}_h)$. The center of the interval given by~\eqref{eq:mathieu} is $(\chi_+(s) - \chi_-(s))/8 = B(e,\widetilde{e})/2$, which is exactly equal to $C_h$. For the choice $s = s_{\rm opt}$, the bound~\eqref{eq:mathieu} thus implies~\eqref{eq:boundERCopti}.
\end{remark}

\begin{remark} \label{rem:QoI_enrich}
A still sharper bound can be obtained from the previous approach when using an enriched adjoint solution $\widetilde{u}_+$ (see~\cite{CHA08}). Starting from~\eqref{eq:rep2}, we have (using an admissible flux field $\widehat{\widetilde{\bq}}_+ \in \widetilde{W}$ obtained as a postprocessing of $\widetilde{\bq}_+$) that
\begin{multline*}
Q(u)-Q(u_h)-R(\widetilde{u}_+) = B(u-u_h,\widetilde{u}-\widetilde{u}_+) = \intO \nab(u-u_h) \cdot \left(\widehat{\widetilde{\bq}}_+-\Aa \nab \widetilde{u}_+ \right) \\ = \intO \Aa^{-1} \left( \bq-\widehat{\bq}^m_h \right) \cdot \left(\widehat{\widetilde{\bq}}_+-\Aa \nab \widetilde{u}_+ \right) + C_+,
\end{multline*}
with $\dis C_+ = \frac{1}{2} \intO \Aa^{-1} \left( \widehat{\bq}_h - \bq_h \right) \cdot \left(\widehat{\widetilde{\bq}}_+-\widetilde{\bq}_+\right)$. This leads to the bound
\begin{equation} \label{eq:QoI_enrich}
|Q(u)-Q(u_h)-\overline{C}_+| \le E_{\rm CRE}(u_h,\widehat{\bq}_h) \, E_{\rm CRE}(\widetilde{u}_+,\widehat{\widetilde{\bq}}_+),
\end{equation}
where $\dis \overline{C}_+ = R(\widetilde{u}_+) + C_+ = \frac{1}{2} \intO \Aa^{-1} \left( \widehat{\bq}_h - \bq_h \right) \cdot \left(\widehat{\widetilde{\bq}}_+ +\widetilde{\bq}_+ \right)$ is a new computable correction term.
\end{remark}

\medskip

Once the error on $Q(u)$ has been estimated, an adaptive strategy similar to the one presented in Section~\ref{section:adaptivity} can be used to reach a given error tolerance. The adaptive algorithm, which should now be based on local error contributions given by~\eqref{eq:adapt_goal1}, \eqref{eq:adapt_goal2}, \eqref{eq:boundparallel} or~\eqref{eq:goestimate}, enables one to recover optimal convergence rates~\cite{MOM09}. For instance, when using~\eqref{eq:goestimate}, and denoting $\theta_{\rm max}$ the maximizer there, we can proceed as follows. We first recast~\eqref{eq:goestimate} as
\begin{multline} \label{eq:mathieu6}
|Q(u)-Q(u_h)|
\le
\left| C_h + \theta_{\rm max} \, E_{\rm CRE}(u_h,\widehat{\bq}_h) \, E_{\rm CRE}(\widetilde{u}_h,\widehat{\widetilde{\bq}}_h) \right|
\\=
\left| \sum_{K \in \mT_h} C_h^K + \theta_{\rm max} \sqrt{\sum_{K \in \mT_h} (\eta_K)^2} \right|,
\end{multline}
where $\dis C_h^K = \frac{1}{2} \int_K \Aa^{-1} \left( \widehat{\bq}_h - \bq_h \right) \cdot \left(\widehat{\widetilde{\bq}}_h-\widetilde{\bq}_h\right)$ is the local contribution to $C_h$ (by construction, $C_h^K$ is an integral over $K$ and $\dis \sum_{K \in \mT_h} C_h^K  = C_h$) and where $\eta_K$ is defined by
\begin{multline*} 
(\eta_K)^2 = \frac{1}{2} \Big( E_{\rm CRE}(u_h,\widehat{\bq}_h) \Big)^2 \Big( E_{{\rm CRE},K}(\widetilde{u}_h,\widehat{\widetilde{\bq}}_h) \Big)^2 \\ + \frac{1}{2} \Big( E_{\rm CRE}(\widetilde{u}_h,\widehat{\widetilde{\bq}}_h) \Big)^2 \Big( E_{{\rm CRE},K}(u_h,\widehat{\bq}_h) \Big)^2,
\end{multline*}
where $E_{{\rm CRE},K}(u_h,\widehat{\bq}_h) = \vertiii{\widehat{\bq}_h-\Aa\nab u_h}_{q,K}/\sqrt{2}$ is the local contribution to $E_{\rm CRE}(u_h,\widehat{\bq}_h)$, in the sense that $\dis \big( E_{\rm CRE}(u_h,\widehat{\bq}_h) \big)^2 = \sum_{K \in \mT_h} \big( E_{{\rm CRE},K}(u_h,\widehat{\bq}_h) \big)^2$ (and likewise for $E_{{\rm CRE},K}(\widetilde{u}_h,\widehat{\widetilde{\bq}}_h)$). We thus observe that
$$
\sum_{K \in \mT_h} (\eta_K)^2 = \big( E_{\rm CRE}(u_h,\widehat{\bq}_h) \, E_{\rm CRE}(\widetilde{u}_h,\widehat{\widetilde{\bq}}_h) \big)^2.
$$
In view of~\eqref{eq:mathieu6}, we then take $C_h^K + \theta_{\rm max} \, \eta_K$ as a local indicator of the error in the element $K$, and proceed with an adaptive algorithm based on these local contributions. 

\subsection{Numerical illustration} \label{sec:goal_num_illus}

To illustrate the performance of the approach, we again consider the numerical example described in Section~\ref{section:illustration}, now focusing on a specific quantity of interest, which we choose as the average (over a critical three-dimensional subdomain $\omega$ of the structure) of the $zz$ component of the stress tensor $\sigma$ (where $z$ is the axial direction of the structure): $\dis Q(u) = \frac{1}{|\omega|} \int_\omega \sigma_{zz}(u)$. The subdomain $\omega$ is shown on Figure~\ref{fig:QoI} and includes around 15 finite elements.

\begin{figure}[H]
\begin{center}
\includegraphics[width=90mm]{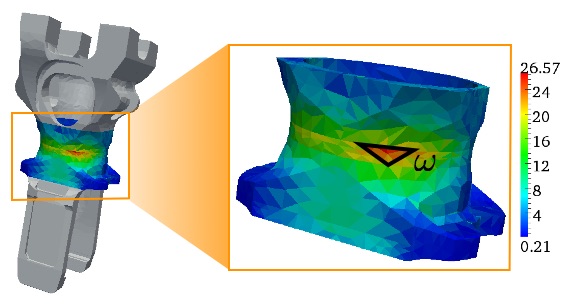}
\end{center}
\caption{Representation of the subdomain of interest $\omega$. Our quantity of interest is the average of $\sigma_{zz}$ over $\omega$. The value of the stress field (in Frobenius norm) in the neighborhood of $\omega$ is indicated on the figure. \label{fig:QoI}}
\end{figure}

The reference value of the quantity of interest (computed using a very fine mesh) is $Q(u) = 33.24$. Using the mesh presented in Section~\ref{section:illustration}, we obtain the approximate value $Q(u_h) = 28.18$. To obtain a bound on $Q(u)-Q(u_h)$, we follow the procedure explained in Remark~\ref{rem:QoI_enrich}. The bound~\eqref{eq:QoI_enrich} yields that the exact value $Q(u)$ satisfies $31.17 \leq Q(u) \leq 34.29$ (the correction term $\overline{C}_+$ is such that $Q(u_h) + \overline{C}_+ = 32.73$). We observe that the actual exact value indeed satisfies these bounds, and that the lower (resp. upper) bound underestimates (resp. overestimates) the exact value of $Q$ by only 6\% (resp. 3\%). The bound~\eqref{eq:QoI_enrich} thus yields a very accurate estimate.

% ---------------------------------------------------------------------
% ---------------------------------------------------------------------

\section{Extensions to other FE schemes}\label{section:extensions1}

Pioneering works on model verification for non-conforming FE schemes were developed in~\cite{BEC03,CAR02,DAR96,KAN99,RIV03} (recent contributions include e.g.~\cite{AIN18}). It was then shown that \textit{a posteriori} error estimation for conforming and non-conforming FE schemes could be phrased in a united framework~\cite{AIN05,AIN06,AIN07,BEC16,ERN07,ERN15,MAL20}. In both cases, one can get computable upper bounds on the error measured in the energy norm. In this section, we restrict our presentation to the reference problem presented in Section~\ref{section:pbreffort}, discretized using piecewise affine, but not necessarily conforming, FE approximations. We recall that
$$
V = \{v \in H^1(\Omega), \ \ v=0 \ \text{on $\Gamma_D$} \}.
$$

\subsection{Splitting of the error}

A conforming Galerkin discretization consists in introducing the finite dimensional subspace $\dis V_h^1 \subset V$. However, there may be advantages in using non-conforming discretization spaces that are not a subspace of $V$. These spaces are still constructed on families of regular partitions $\mT_h$ of $\Omega$. In the following, as in Section~\ref{section:elementresidualmeth}, we refer to a broken Sobolev space 
$$
X_{\rm brok} = \oplus_K H^1(K)
$$
and we introduce, for each element $K$, the function $\sigma_{\Gamma,K} : \partial K \mapsto \{-1,+1\}$ defined by~\eqref{eq:def_sigma_gamma_K}. For any $v \in X_{\rm brok}$, we denote by $[[v]]$ (resp. $\langle v \rangle$) its jump across (resp. its average on) the internal element edges. We also introduce $\bn = \sigma_{\Gamma,K} \, \bn_K$ (which is a well-defined vector on each edge) and the normal flux $q_n(v) = \bn \cdot \bq(v) = \bn \cdot \Aa \nab v$.

For any given parameter $\kappa>0$, we introduce the forms
\begin{multline*}
  B_{\rm brok}(w,v)
  = \sum_{K\in \mT_h} \int_K \Aa \nab w \cdot \nab v - \sum_{\Gamma \subset \Gamma_{\rm int} \cup \Gamma_D} \int_\Gamma \Big([[w]] \, \langle q_n(v) \rangle + \langle q_n(w) \rangle \, [[v]] \Big) \\ + \sum_{\Gamma \subset \Gamma_{\rm int} \cup \Gamma_D} \frac{\kappa}{h_\Gamma} \int_\Gamma [[w]] \, [[v]]
\end{multline*}
and
$$
F_{\rm brok}(v) = \sum_{K\in \mT_h} \int_K f \, v + \sum_{\Gamma \subset \Gamma_N} \int_\Gamma g \, v. 
$$
Note that, for any $v$ and $w$ in $X_{\rm brok}$, an integral over an edge of the product $\langle q_n(v) \rangle \, [[w]]$ is interpreted as a duality pairing. Furthermore, on edges $\Gamma \subset \Gamma_D$, $[[v]]$ and $\langle v \rangle$ simply denote the function $v$ itself: $[[v]] = \langle v \rangle = v$ on $\Gamma \subset \Gamma_D$. Note also that functions in $X_{\rm brok}$ do not necessarily vanish on $\Gamma_D$.

\medskip

We consider the following problem:
\begin{equation}\label{eq:pb_DG}
\text{Find $U \in X_{\rm brok}$ such that, for any $v \in X_{\rm brok}$,} \quad B_{\rm brok}(U,v) = F_{\rm brok}(v).
\end{equation}
By choosing appropriate test functions $v \in X_{\rm brok}$ (e.g. supported on a single element $K$, or non-zero on a single edge $\Gamma$), it can be easily shown~\cite{ARN02} that the unique solution $U$ to~\eqref{eq:pb_DG} corresponds to the solution $u$ to~\eqref{eq:refpbstrong}. We have thus recast the problem of interest as~\eqref{eq:pb_DG}.

\medskip

We then introduce a finite dimensional subspace $X_h \subset X_{\rm brok}$, based on the partition $\mT_h$ of $\Omega$, and construct an approximation $U_h \in X_h$ by considering the finite-dimensional problem
\begin{equation}\label{eq:weakNC}
\forall v \in X_h, \quad B_{\rm brok}(U_h,v)=F_{\rm brok}(v).
\end{equation}
If $\kappa$ sufficiently large, then the above problem has a unique solution (see~\cite{AIN05}). In the following, we consider two choices of $X_h$, which differ by the conditions which are applied on $\Gamma_{\rm int} \cup \Gamma_D$:
\begin{itemize}
\item Discontinuous Galerkin FEM:
$$
X_h^{\rm DG} = \{ v \in X_{\rm brok}, \quad v_{|K} \in P_1(K) \ \ \text{for any $K \in \mT_h$} \};
$$
\item Crouzeix-Raviart FEM:
$$
X_h^{\rm CR} = \{v \in X_h^{\rm DG}, \quad \int_\Gamma [[v]] = 0 \ \ \text{for any $\Gamma \subset \Gamma_{\rm int} \cup \Gamma_D$} \}.
$$
\end{itemize}
We recall that the conforming FEM corresponds to the choice
$$
X_h^{\rm CG}= V_h^1 = \{v\in X_h^{\rm DG}, \quad [[v]] = 0 \ \ \text{on any $\Gamma \subset \Gamma_{\rm int} \cup \Gamma_D$} \}.
$$
We of course have $X_h^{\rm CG} \subset X_h^{\rm CR} \subset X_h^{\rm DG} \subset X_{\rm brok}$. Note that functions in $X_h^{\rm CR}$ are continuous at the midpoints of the interior edges, and vanish at the midpoints of edges on $\Gamma_D$.

\medskip

Since the error $e=u-U_h$ is in general not in $V$, a natural idea is to split it into a part that belongs to $V$ and a remainder. We thus consider the projection $\phi\in V$ of $e$ (which is the so-called conforming error), defined by
\begin{equation}
  \label{eq:def_phi}
\forall v \in V, \quad B(\phi,v)=\intO \Aa \nab_h e\cdot \nab v = \intO f \, v + \int_{\Gamma_N} g \, v - \intO \Aa \nab_h U_h\cdot \nab v,
\end{equation}
where, for any $w \in X_{\rm brok}$, $\nab_h w$ is equal, on each element $K$, to $\nab w$. The remaining part $\rho=e-\phi$ is orthogonal to $V$, in the sense that
\begin{equation}
  \label{eq:helas5}
\forall v \in V, \quad \intO \Aa \nab_h \rho \cdot \nab v = 0.
\end{equation}
Denoting $\bq(\rho) = \Aa \nab_h \rho$, we thus see that $\bq(\rho)$ is divergence free in $\Omega$ (at least in the sense of distributions), and that $\bq(\rho) \in (L^2(\Omega))^d$. There thus exists $\boldsymbol{\psi} \in (H^1(\Omega))^d$ such that $\bq(\rho) = \rott \boldsymbol{\psi}$. We next infer from~\eqref{eq:helas5}, using an integration by parts, that $\bn \cdot \rott \boldsymbol{\psi} = 0$ on $\Gamma_N$. Summarizing, $\bq(\rho)$ can be expressed as the curl of a function $\boldsymbol{\psi}$ belonging to
\begin{equation*}
H=\left\{ \bw \in (H^1(\Omega))^d, \quad \bn \cdot \rott \bw = 0 \ \text{on $\Gamma_N$} \right\},
\end{equation*}
and we have the splitting
\begin{equation}
  \label{eq:helas4}
\bq(e) = \Aa \nab_h e = \bq(\phi) + \rott \boldsymbol{\psi},
\end{equation}
which is an orthogonal decomposition in the sense that
\begin{equation}
  \label{eq:helas7}
\intO \Aa^{-1}\bq(e)\cdot\bq(e) = \intO \Aa^{-1}\bq(\phi)\cdot\bq(\phi) + \intO \Aa^{-1} (\rott \boldsymbol{\psi}) \cdot (\rott \boldsymbol{\psi}),
\end{equation}
since $\dis \intO \Aa^{-1} \bq(\phi) \cdot \rott \boldsymbol{\psi} = 0$ in view of~\eqref{eq:helas5}. The above equation can equivalently be written $\vertiii{e}^2_h = \vertiii{\bq(\phi)}_q^2 + \vertiii{\bq(\rho)}_{q,h}^2 = \vertiii{\phi}^2 + \vertiii{\rho}^2_h$, with the notation $\dis \vertiii{e}^2_h = \int_\Omega \Aa \nab_h e \cdot \nab_h e$.

\begin{remark}
If $\Aa$ is the identity matrix, the splitting~\eqref{eq:helas4} is the classical Helmholtz decomposition of the vector $\nab_h e$.
\end{remark}

\subsection{Estimation of the non-conforming part of the error}

The non-conforming part $\rho$ of the error may be estimated using the identity
\begin{equation}
  \label{eq:helas6}
\vertiii{\rho}_h = \vertiii{\bq(\rho)}_{q,h} = \min_{v\in V} \vertiii{v-U_h}_h,
\end{equation}
which can be established as follows. We first observe that, for any $v \in V$, we have, in view of~\eqref{eq:helas5},
\begin{equation*}
\vertiii{v-U_h}^2_h = \vertiii{v-u+\phi+\rho}^2_h = \vertiii{v-u+\phi}^2 + \vertiii{\rho}^2_h.
\end{equation*}
This implies that $\vertiii{v-U_h}_h \ge \vertiii{\rho}_h$ for any $v \in V$. Furthermore, $\vertiii{v-U_h}_h = \vertiii{\rho}_h$ for the specific choice $v = u - \phi \in V$, which shows~\eqref{eq:helas6}.

\medskip

In order to use~\eqref{eq:helas6} to estimate $\vertiii{\rho}_h$, we need to find a suitable choice for $\widehat{u}_h \in V$. This function can be constructed by smoothing the approximation $U_h$ with an averaging operator, i.e. defining $\widehat{u}_h$ as a continuous and piecewise affine function over $\mT_h$ with value at node $i$ taken as the average of the values of $U_h$ at this node (except for nodes on $\Gamma_D$, where the value of $\widehat{u}_h$ is taken to be zero). This gives a computable upper bound on the non-conforming part of the error:
\begin{equation*}
\vertiii{\rho}^2_h \le \sum_{K \in \mT_h} \eta^2_{{\rm NC},K} \qquad \text{with} \qquad \eta^2_{{\rm NC},K} = \int_K \Aa \nab (\widehat{u}_h-U_h)\cdot \nab (\widehat{u}_h-U_h).
\end{equation*} 
It can be shown~\cite{AIN06,AIN07} that this estimator is efficient in the sense that there exists $C$, independent of the mesh size, such that
\begin{equation*}
\eta^2_{{\rm NC},K} \le C\int_{U(K)} \Aa \nab_h \rho \cdot \nab_h \rho,
\end{equation*} 
where $U(K)$ is the patch of elements (defined in~\eqref{eq:neighbor}) composed of $K$ and its neighbors. 

\begin{remark}
The estimate on the non-conforming part of the error is consistent with a conforming Galerkin discretization. Indeed, in this case, our above construction yields $\widehat{u}_h = U_h \in V$ and thus $\eta_{{\rm NC},K} = 0$.
\end{remark}

\subsection{Estimation of the conforming part of the error}

To estimate the conforming part $\phi$ of the error, the classical \textit{a posteriori} error estimation tools developed in the above sections for conforming Galerkin discretization can be reused. We here follow the strategy of the equilibrated element residual method (see Section~\ref{section:elementresidualmeth}). Introduce a set of tractions $\widehat{g}_{\Gamma,K}$ defined on element boundaries, such that
\begin{itemize}
\item they are equilibrated at the element level, that is,
  \begin{equation}\label{eq:equilelement_DG}
    \forall K \in \mT_h, \quad \int_K f + \sum_{\Gamma \subset \partial K} \int_\Gamma \widehat{g}_{\Gamma,K} = 0;
  \end{equation}
\item they satisfy
\begin{equation} \label{eq:helas15}
\widehat{g}_{\Gamma,K_1}+\widehat{g}_{\Gamma,K_2} = 0 \quad \text{on $\partial K_1 \cap \partial K_2$}, \qquad \widehat{g}_{\Gamma,K} = g \quad \text{on $\partial K \cap \Gamma_N$}.
\end{equation}
\end{itemize}
Simply using~\eqref{eq:helas15}, we obtain from~\eqref{eq:def_phi} that
\begin{equation} \label{eq:helas16}
\forall v \in V, \quad B(\phi,v) = \sum_{K \in \mT_h}\left( \int_K f \, v + \sum_{\Gamma \subset \partial K} \int_\Gamma \widehat{g}_{\Gamma,K} \, v -\int_K \Aa\nab U_h \cdot \nab v\right),
\end{equation}
a relation that will be useful below.

The technique for the construction of the flux functions $\widehat{g}_{\Gamma,K}$ is again based on a post-processing of the numerical solution $U_h$ to~\eqref{eq:weakNC}, and it differs depending on the non-conforming FE scheme. However, it is in any case less complex than the technique introduced in Section~\ref{section:SA} for a conforming FE scheme. We detail the procedure in the two cases that we consider, the Discontinuous Galerkin FEM and the Crouzeix-Raviart FEM:
\begin{itemize}
\item For Discontinuous Galerkin (DG) FEM, the absence of any continuity requirement on the test functions $v\in X_h^{\rm DG}$ involved in~\eqref{eq:weakNC} makes the construction of $\widehat{g}_{\Gamma,K}$ easy. On any edge $\Gamma \subset \Gamma_{\rm int} \cup \Gamma_D$ and for any element $K$ such that $\Gamma \subset \partial K$, we define $\widehat{g}_{\Gamma,K}$ by the simple closed form expression
\begin{equation*}
\widehat{g}_{\Gamma,K} = \sigma_{\Gamma,K} \left( \langle q_n(U_h) \rangle - \frac{\kappa}{h_\Gamma} [[U_h]] \right).
\end{equation*}
The key point to check that~\eqref{eq:equilelement_DG} indeed holds is that the characteristic function $\chi_K\in X_h^{\rm DG}$ of any element $K$ is an admissible test function for the DG FEM. Taking $v = \chi_K$ in~\eqref{eq:weakNC}, we easily obtain that the tractions defined above indeed satisfy~\eqref{eq:equilelement_DG}. In addition, the property~\eqref{eq:helas15} obviously holds. 
\item In contrast to DG, the Crouzeix-Raviart scheme imposes a weak continuity condition on the test functions $v\in X_h^{\rm CR}$ between neighboring elements. We can thus not choose a test function supported on a single element. For a given edge $\Gamma$, let $\theta_\Gamma \in X_h^{\rm CR}$ be the function whose value is unity at the midpoint of $\Gamma$ and which vanishes at the midpoint of all the other edges. The support of $\theta_\Gamma$ is the set of elements of which $\Gamma$ is an edge. In the two-dimensional case, the support of $\theta_\Gamma$ is $K_1 \cup K_2$, where $K_1$ and $K_2$ are the two triangles sharing the edge $\Gamma$, and $\theta_\Gamma$ is actually the unique piecewise affine function which vanishes everywhere in $\Omega \setminus (K_1 \cup K_2)$, and which is equal to 1 at both endpoints of the edge $\Gamma$ and equal to -1 at the third vertex of $K_1$ and of $K_2$. We then get $\theta_\Gamma = 1$ on $\Gamma$ (and thus $[[ \theta_\Gamma ]] = 0$ on $\Gamma$). This property also holds in dimension $d \geq 3$. 
In addition, for any element $K$, we have $\dis \sum_{\Gamma \subset \partial K} \theta_\Gamma = 1$ on $K$ since $\dis \sum_{\Gamma \subset \partial K} \theta_\Gamma$ is an affine function equal to 1 at the midpoint of each edge of $K$.

On any edge $\Gamma \subset \Gamma_{\rm int} \cup \Gamma_D$ and for any element $K$ such that $\Gamma \subset \partial K$, we now define $\widehat{g}_{\Gamma,K}$ as a constant function with value given by  
\begin{equation*}
h_\Gamma \, \widehat{g}_{\Gamma,K} = \int_K \Aa \nab U_h \cdot \nab \theta_\Gamma - \int_K f \, \theta_\Gamma + \sum_{\Gamma' \subset \partial K} \frac{\kappa}{h_{\Gamma'}}\int_{\Gamma'} [[U_h]] \ [[ \theta_\Gamma]].
\end{equation*}
We now check that these tractions indeed satisfy~\eqref{eq:helas15}, under the additional assumption that $\Aa$ is piecewise constant on the mesh. Let $K_1$ and $K_2$ be the two elements sharing the edge $\Gamma$ (for brevity, we only consider the case when $\partial K_1 \cup \partial K_2$ does not contain any edge of $\partial \Omega$). We then have, using~\eqref{eq:weakNC} with $v=\theta_\Gamma$, that
\begin{align}
  & h_\Gamma \, \left( \widehat{g}_{\Gamma,K_1}+\widehat{g}_{\Gamma,K_2} \right)
  \nonumber
  \\
  &=
  \int_\Omega \Aa \nab_h U_h \cdot \nab_h \theta_\Gamma - \int_\Omega f \, \theta_\Gamma + \sum_{\Gamma' \subset \partial K_1 \cup \partial K_2} \frac{\kappa}{h_{\Gamma'}}\int_{\Gamma'} [[U_h]] \ [[ \theta_\Gamma]]
  \nonumber
  \\
  &=
  \sum_{\Gamma' \subset \Gamma_{\rm int}} \int_{\Gamma'} \Big([[U_h]] \, \langle q_n(\theta_\Gamma) \rangle + \langle q_n(U_h) \rangle \, [[\theta_\Gamma]] \Big),
  \label{eq:mathieu3}
\end{align}
where we have used, in the first line, that $\theta_\Gamma$ is supported in $K_1 \cup K_2$. Since $\theta_\Gamma$ and $U_h$ are piecewise affine functions and since $\Aa$ is piecewise constant, we have that $\bq(\theta_\Gamma)$ and $\bq(U_h)$ are piecewise constant. We therefore have that, on each edge, $\langle q_n(\theta_\Gamma) \rangle$ and $\langle q_n(U_h) \rangle$ are constant. We thus compute that $\dis \int_{\Gamma'} [[U_h]] \, \langle q_n(\theta_\Gamma) \rangle = \langle q_n(\theta_\Gamma) \rangle_{|\Gamma'} \int_{\Gamma'} [[U_h]] = 0$, where the last equality stems from the fact that $U_h \in X_h^{\rm CR}$. The second term in the right-hand side of~\eqref{eq:mathieu3} vanishes for the same reasons. We hence deduce~\eqref{eq:helas15}.

We next check that the local equilibrium property~\eqref{eq:equilelement_DG} holds. For any $K$, we compute that
\begin{align}
  & \int_K f + \sum_{\Gamma \subset \partial K} \int_\Gamma \widehat{g}_{\Gamma,K}
  \nonumber
  \\
  &=
  \int_K f + \sum_{\Gamma \subset \partial K} \left( \int_K \Aa \nab U_h \cdot \nab \theta_\Gamma - \int_K f \, \theta_\Gamma + \sum_{\Gamma' \subset \partial K} \frac{\kappa}{h_{\Gamma'}}\int_{\Gamma'} [[U_h]] \ [[ \theta_\Gamma]] \right)
  \nonumber
  \\
  &=
  \sum_{\Gamma \subset \partial K} \left( \sum_{\Gamma' \subset \partial K} \frac{\kappa}{h_{\Gamma'}}\int_{\Gamma'} [[U_h]] \ [[ \theta_\Gamma]] \right) \qquad \text{[since $\sum_{\Gamma \subset \partial K} \theta_\Gamma = 1$ on $K$]}
  \nonumber
  \\
  &=
  \sum_{\Gamma' \subset \partial K} \frac{\kappa}{h_{\Gamma'}}\int_{\Gamma'} [[U_h]] \ [[ \sum_{\Gamma \subset \partial K} \theta_\Gamma]].
  \label{eq:helas17}
\end{align}
For any edge $\Gamma' \subset \partial K$, let us denote $K$ and $K'$ the two elements sharing that edge. We know that $\dis \sum_{\Gamma \subset \partial K} \theta_\Gamma = 1$ in $K$. In $K'$, we have $\theta_\Gamma = 0$ for all $\Gamma \subset \partial K \setminus \Gamma'$, hence $\dis \sum_{\Gamma \subset \partial K} \theta_\Gamma = \theta_{\Gamma'}$, which is continuous across $\Gamma'$ and equal to 1 there. We hence get that $\dis \sum_{\Gamma \subset \partial K} \theta_\Gamma$ is equal to 1 on both sides of the edge $\Gamma'$, and we thus infer~\eqref{eq:equilelement_DG} from~\eqref{eq:helas17}. 
\end{itemize}
Once we have defined equilibrated tractions $\widehat{g}_{\Gamma,K}$, we proceed similarly to~\eqref{eq:retour3} and introduce local residual problems. Here, they consist in finding $w_K \in H^1(K)$ such that
\begin{equation*}
\forall v \in H^1(K), \quad \int_K \Aa \nab w_K \cdot \nab v =\int_K f \, v + \sum_{\Gamma \subset \partial K} \int_\Gamma \widehat{g}_{\Gamma,K} \, v - \int_K \Aa\nab U_h \cdot \nab v.
\end{equation*}
Since~\eqref{eq:equilelement_DG} is satisfied, the above Neumann problem is well-posed and uniquely defines $w_K$, up to an additive constant. In view of~\eqref{eq:helas16}, we deduce that
\begin{equation*}
\forall v \in V, \quad B(\phi,v) = \sum_{K \in \mT_h} \int_K \Aa \nab w_K \cdot \nab v,
\end{equation*}
and we get, applying the Cauchy-Schwarz inequality, an upper bound on the conforming part of the error:
\begin{equation*}
\vertiii{\phi}^2 \le \sum_{K \in \mT_h} \eta^2_{{\rm CF},K} \qquad \text{with} \qquad \eta^2_{{\rm CF},K} = \int_K \Aa \nab w_K \cdot \nab w_K. 
\end{equation*} 
It can be shown~\cite{AIN05} that this estimator is efficient in the sense that there exists $C$, independent of the mesh size, such that
\begin{equation*}
\eta^2_{{\rm CF},K} \le C\int_{U(K)} \Aa \nab \phi\cdot \nab \phi,
\end{equation*} 
where $U(K)$ is again the patch of elements composed of $K$ and its neighbors.

\begin{remark}
A similar bound can be obtained using duality arguments and the CRE concept~\cite{DES98b,ERN15}, constructing an equilibrated flux field $\widehat{\bq}_h \in W$ and defining $\eta_{{\rm CF},K} = \vertiii{\widehat{\bq}_h-\Aa \nab U_h}_{q,K}$. As highlighted previously, the procedure to recover $\widehat{\bq}_h$ (e.g. using the hybrid-flux approach or RTN elements) is here less technical than in the classical conforming FEM.
\end{remark}

\subsection{A posteriori bounds on the total error}

In view of~\eqref{eq:helas7}, and collecting the bounds from above and below established on the two error components, we eventually get
\begin{equation*}
\frac{1}{C} \sum_{K \in \mT_h} (\eta^2_{{\rm CF},K} + \eta^2_{{\rm NC},K}) \le \vertiii{e}^2_h \le \sum_{K \in \mT_h} (\eta^2_{{\rm CF},K} + \eta^2_{{\rm NC},K}),
\end{equation*}
for some $C$ independent of $h$, where we recall that $\vertiii{\cdot}_h$ is the broken energy norm.

% ---------------------------------------------------------------------
% ---------------------------------------------------------------------

\section{Extensions to other mathematical problems}\label{section:extensions2}

In all the above sections, we have considered the problem~\eqref{eq:refpbstrong}, namely a time-independent, linear elliptic problem in divergence form, with a symmetric matrix $\Aa$. In order to show the versatility of the approaches we have presented, we turn in this section to more general problems. We consider a non-symmetric problem in Section~\ref{sec:adv-diff}. Time-dependent and nonlinear problems are next briefly discussed in Section~\ref{sec:non-lin}.

\subsection{Advection-diffusion-reaction problems}
\label{sec:adv-diff}

We consider a scalar advection-diffusion-reaction problem of the form
\begin{equation} \label{eq:advec}
-\dive (\Aa \nab u) + \bc \cdot \nab u + r \, u = f \ \ \text{in $\Omega$}, \quad u=0 \ \ \text{on $\Gamma_D$}, \quad (\Aa \nab u) \cdot \bn = g \ \ \text{on $\Gamma_N$}.
\end{equation}
We introduce
\begin{equation*}
B(u,v)=\intO \Aa\nab u \cdot \nab v + (\bc \cdot \nab u) v + r \, u \, v, \qquad F(v) = \intO f \, v + \int_{\Gamma_N} g \, v,
\end{equation*}
and assume that the matrix $\Aa$ satisfies~\eqref{eq:elliptic}, that $\dis r-\frac{1}{2} \dive \bc \ge 0$ in $\Omega$ and that $\bc \cdot \bn \geq 0$ on $\Gamma_N$. The bilinear form $B$ is hence coercive on $V$. For the sake of simplicity, we again assume that $\Aa$ is symmetric. Note however that $B$ is not symmetric, due to the presence of the advection term. 

\medskip

The weak formulation of~\eqref{eq:advec} is
\begin{equation*}
\text{Find $u \in V$ such that, for any $v \in V$,}\quad B(u,v)=F(v).
\end{equation*}
It has a unique solution in view of the above assumptions. This solution may show boundary layers, namely subregions of $\Omega$ where the derivatives of $u$ are very large~\cite{JOH00,STE98}. We define the numerical problem as:
\begin{equation} \label{eq:FEpb_adv}
\text{Find $u_h \in V_h^p$ such that, for any $v \in V_h^p$,} \quad B(u_h,v)=F(v).
\end{equation}

\begin{remark}
\label{rem:supg}
Even though the bilinear form $B$ is coercive, the problem~\eqref{eq:FEpb_adv} may be ill-conditioned, typically in the case of advection-dominated problems (i.e. when the P\'eclet number $P_e$ is large; considering for instance the simple case $r=0$, we recall that the P\'eclet number is $P_e = h \| \bc \|_{L^\infty(\Omega)}/a_{\rm min}$, where $a_{\rm min}$ is the coercivity constant in~\eqref{eq:elliptic}). Stabilization techniques (see e.g.~\cite{BRE92,FRA92,HUG95,QUA94,STY05}) are then required. The streamline diffusion FE method (SDFEM, or SUPG) is one of them. It consists in adding to the standard Galerkin FEM some weighted residuals that can be interpreted as artificial diffusion along streamlines. This leads to the new forms
\begin{align*}
B_{\rm stab}(u,v) &= B(u,v) + \sum_{K \in \mT_h} \delta_K \int_K \Big( {\cal L} u \Big) \, \Big( {\cal L}_{ss} v \Big), \\
F_{\rm stab}(v) &= F(v) + \sum_{K \in \mT_h} \delta_K \int_K f \, \Big( {\cal L}_{ss} v \Big), 
\end{align*}
where $\delta_K$ is a local stabilization parameter, $\dis {\cal L} u = -\dive (\Aa \nab u) + \bc \cdot \nab u + r \, u$ is the operator of the problem and ${\cal L}_{ss}$ is the skew-symmetric part of ${\cal L}$. The stabilized problem then consists in finding $u_h^{\rm stab} \in V_h^p$ such that $B_{\rm stab}(u_h^{\rm stab},v) = F_{\rm stab}(v)$ for any $v \in V_h^p$. As is well-known, this approach is strongly consistent, in the sense that the exact solution $u$ satisfies $B_{\rm stab}(u,v) = F_{\rm stab}(v)$ for any $v \in V_h^p$.
\end{remark}

\subsubsection{\textit{A priori} error estimation}

We introduce the Cl\'ement interpolant $\Pi_h u$ of the exact solution $u$, as defined in Section~\ref{section:apriori}, and decompose the overall error as $e=(u-\Pi_h u)+(\Pi_h u-u_h)=\mu+e_h$. The term $\mu$ corresponds to the interpolation error, while $e_h$ belongs to $V_h^p$. 

\medskip

When $\bc=0$, the bilinear form $B$ is symmetric and Galerkin orthogonality yields
\begin{equation}
\label{eq:le_cas_symm}
\vertiii{e}^2 = B(e,\mu+e_h) = B(e,\mu) \le \vertiii{e} \ \vertiii{\mu},
\end{equation}
and thus $\vertiii{e} \le \vertiii{\mu}$, which corresponds to the best approximation property. The error (in the energy norm) can thus be bounded as in Section~\ref{section:apriori}, from a standard interpolation estimate~\cite{CIA78,STR73}. In particular, when $u \in H^q(\Omega)$ for some $q \ge p+1$, we get 
\begin{equation}\label{eq:aprioriadv}
\vertiii{e} \le C \, h^p \, \|u\|_{p+1},
\end{equation}
where $C$ depends on $p$ and the geometry of the mesh (but is independent of $h$ and $u$).

\medskip

In the general case (i.e. for a non-symmetric bilinear form $B$), the estimate $|B(e,\mu)| \le \vertiii{e} \ \vertiii{\mu}$ used in~\eqref{eq:le_cas_symm} does not hold. We follow here the procedure introduced in~\cite{STE98} to estimate the error. The form $B$ is split in its symmetric and skew symmetric parts:
\begin{multline*}
B(w,v)=\frac{1}{2} \big( B(w,v)+B(v,w) \big) + \frac{1}{2} \big( B(w,v)-B(v,w) \big) \\ = B_{\rm symm}(w,v) + B_{\rm skew}(w,v),
\end{multline*}
with 
\begin{align*}
B_{\rm symm}(w,v)&=\intO \left[\Aa\nab w \cdot \nab v + (r-\frac{1}{2} \dive \bc) \, w \, v \right] + \frac{1}{2} \int_{\Gamma_N} \bc\cdot \bn \ w \, v, \\
B_{\rm skew}(w,v)&=\frac{1}{2}\intO \bc \cdot (v \nab w -w \nab v). 
\end{align*}

\begin{remark}
From this decomposition of $B$, it is straightforward to understand why the inequality in~\eqref{eq:le_cas_symm} does not hold. The skew symmetric part $B_{\rm skew}$ does not contribute to the ``energy'' norms $\vertiii{e}$ and $\vertiii{\mu}$, but it contributes to $B(e,\mu)$.
\end{remark}

We next define a dual norm (the so-called skew norm) to accomodate for the skew symmetric part of $B$:
\begin{equation*}
\vertiii{w}_{\rm skew} := \sup_{v\in V}\frac{|B_{\rm skew}(w,v)|}{\vertiii{v}}.
\end{equation*}
For any $w$ and $v$ in $V$, we thus have $|B_{\rm skew}(w,v)| \le \vertiii{w}_{\rm skew} \, \vertiii{v}$. Using that $B_{\rm skew}(w,v) = -B_{\rm skew}(v,w)$, we also have $|B_{\rm skew}(v,w)| \le \vertiii{w}_{\rm skew} \, \vertiii{v}$. We thus obtain
\begin{multline*}
\vertiii{e}^2 
= 
|B(e,\mu)|
=
|B_{\rm symm}(e,\mu)+B_{\rm skew}(e,\mu)| 
\\ \le 
|B_{\rm symm}(e,\mu)|+|B_{\rm skew}(\mu,e)| 
\le 
\vertiii{\mu} \, \vertiii{e} + \vertiii{\mu}_{\rm skew} \, \vertiii{e},
\end{multline*}
hence
\begin{equation} \label{eq:helas18}
\vertiii{e} \le \vertiii{\mu}+\vertiii{\mu}_{\rm skew}.
\end{equation}
Focusing on $\vertiii{\mu}_{\rm skew}$, we get, using the Poincar\'e inequality in $V$, that
\begin{equation*}
\vertiii{\mu}_{\rm skew} = \sup_{v\in V}\frac{|B_{\rm skew}(v,\mu)|}{\vertiii{v}} \le C \|\mu\|_1,
\end{equation*}
where the constant $C$ is large for advection-diffusion problems with large P\'eclet number $P_e$. Inserting the above bound in~\eqref{eq:helas18} and using an interpolation estimate, we recover~\eqref{eq:aprioriadv} when $u \in H^q(\Omega)$ for some $q \ge p+1$.

\begin{remark}
An alternative approach to estimate the error consists in using the coercivity of $B$ in $V$: there exists some $c>0$ such that $B(v,v) \geq c \, \| v \|_1^2$ for any $v \in V$. Writing this coercivity estimate for $e \in V$ and using next the variational formulations that $u$ and $u_h$ satisfy, we obtain that, for any $v_h \in V_h^p$,
$$
c \, \| e \|_1^2 \leq B(e,e) = B(u-u_h,u-u_h) = B(u-u_h,u-v_h).
$$
Using next that $B$ is continuous in $V$, we deduce that
$$
c \, \| e \|_1^2 \leq C \, \| u-u_h \|_1 \ \| u-v_h \|_1,
$$
and therefore $\| e \|_1 \leq C \, \| u-v_h \|_1$ for any $v_h \in V_h^p$. Using classical approximation results, we hence infer that $\| e \|_1 \leq C \, h^p \, \| u \|_{p+1}$. We eventually write that $\vertiii{e} = \sqrt{B(e,e)} \leq C \, \| e \|_1$ and deduce~\eqref{eq:aprioriadv}.
\end{remark}

\begin{remark}
Suppose that $u \in H^2(\Omega)$, that $P_e>1$ and that we use piecewise affine finite elements. We can then show that $\| u - u_h \|_1 \le Ch(1+P_e) \|u\|_2$ when using the (original, non-stabilized) bilinear form $B$, whereas $\| u - u_h^{\rm stab} \|_1\le Ch(1+\sqrt{P_e})\|u\|_2$ when using the stabilized formulation described in Remark~\ref{rem:supg}.
\end{remark}

\subsubsection{\textit{A posteriori} error estimation}

Several \textit{a posteriori} error estimates have been developed for advection-diffusion-reaction problems in the literature. We detail some of them below. They basically follow similar methodologies as for pure diffusion problems (i.e. with a symmetric operator), since the Cauchy-Schwarz inequality (in the sense of~\eqref{eq:le_cas_symm}) is not needed when establishing \textit{a posteriori} estimates.

\medskip

In~\cite{JOH00}, the use of the explicit residual method (in the vein of Section~\ref{sec:resu_exp}) leads to the general form $\eta = \sqrt{\sum_{K \in \mT_h} \eta^2_K}$ for the estimate of the error in the energy norm, with
\begin{multline*}
\eta^2_K = \alpha_K \, \|f+\dive (\Aa \nab u_h)- \bc \cdot \nab u_h - r \, u_h\|^2_{0,K} \\ + \sum_{\Gamma \subset \partial K \setminus \Gamma_N} \frac{\beta_\Gamma}{2} \, \Big\| [[ \Aa \nab u_h \cdot \bn ]]_\Gamma \Big\|_{0,\Gamma}^2 + \sum_{\Gamma \subset \partial K \cap \Gamma_N} \beta_\Gamma \, \|\Aa \nab u_h \cdot \bn-g\|_{0,\Gamma}^2,
\end{multline*}
where $\alpha_K = \min(h_K^2 \, a_{\rm max}^{-1},1)$ and $\beta_\Gamma = \min(h_\Gamma \, a_{\rm max}^{-1},a_{\rm max}^{-1/2})$. It is shown that $\eta$ provides for a global upper estimate, and $\eta_K$ provides for a local lower estimate, up to unknown multiplicative constants. These estimates are robust in the advection-dominated regime. We refer to~\cite{JOH00} for more details.

\medskip

Implicit methods such as the element residual method (in the vein of Section~\ref{section:elementresidualmeth}) may also be used. They involve the resolution of local problems of the following form (compare with~\eqref{locpbelem}): find $e_K\in V(K)$ such that, for any $v \in V(K)$,
\begin{multline*}
B_K(e_K,v) = \int_K (f+ \dive \Aa \nab u_h - \bc \cdot \nab u_h - r \, u_h) \, v \\ -\frac{1}{2} \sum_{\Gamma \subset \partial K \setminus \partial \Omega} \int_\Gamma [[ \Aa\nab u_h\cdot \bn ]] \, v - \sum_{\Gamma \subset \partial K \cap \Gamma_N} \int_\Gamma (\Aa \nab u_h \cdot \bn-g) \, v,
\end{multline*}
where we recall that $V(K) = \left\{v \in H^1(K), \ \text{$v=0$ on $\partial K \cap \Gamma_D$} \right\}$. Local estimates $\eta_K$ are next computed using $e_K$. The above problem may be in practice solved with higher-order elements (bubble functions). A difficulty comes from the fact that the solution $e_K$ may have oscillations in convection-dominated problems. This can again be addressed using stabilization techniques.

\medskip

Duality-based approaches (in the spirit of the approaches presented in Section~\ref{section:CRE}) can also be employed for advection-diffusion-reaction problems~\cite{ERN10b,ERN15}. Defining the flux as $\bq = \Aa \nab u - \bc \, u$, the CRE framework provides for a guaranteed estimate of the global error measured using the norm $\vertiii{v}_{\oplus} := \vertiii{v} + \vertiii{v}_{\rm skew}$. It can be shown that this norm, which is an augmented energy norm, is equivalent to the dual norm of the residual $R(v) = F(v)-B(u_h,v)$.

\begin{remark}
When the hybrid-flux technique of Section~\ref{section:hybridfluxapp} is considered to compute an admissible flux field, the associated prolongation condition should be consistent with the discretization of the problem (in particular when stabilization is used) in order to recover required properties, and in particular the solvability condition on the linear system~\eqref{eq:helas3}. We refer to~\cite{PAR13} for details. 
\end{remark}

\subsubsection{Goal-oriented error estimation}

The splitting between the symmetric and the skew symmetric parts of $B$ can also be used for goal-oriented error estimation~\cite{PAR09b,PAR13,SAU06}. Introducing the adjoint problem $B(v,\widetilde{u}) = Q(v)$ for any $v \in V$ (possibly solved in practice using a stabilization technique), that is
\begin{align*}
-\dive (\Aa \nab \widetilde{u}) - \bc \cdot \nab \widetilde{u} + (r-\dive \bc) \, \widetilde{u} &= \widetilde{f}_Q -\dive (\widetilde{\bq}_Q + \Aa \nab \widetilde{u}_Q) \quad \text{in $\Omega$}, \\ 
\widetilde{u}&=0 \quad \text{on $\Gamma_D$}, \\ 
\Aa\nab \widetilde{u} \cdot \bn + \bc \cdot \bn \, \widetilde{u} &= \widetilde{g}_Q + (\widetilde{\bq}_Q + \Aa \nab \widetilde{u}_Q) \cdot \bn \quad \text{on $\Gamma_N$},
\end{align*}
where $\widetilde{\bq}_Q$, $\widetilde{f}_Q$ and $\widetilde{g}_Q$ are the extraction functions used in the definition~\eqref{eq:rep_Q} of the quantity of interest $Q$, we get
\begin{equation*}
Q(e)=B(e,\widetilde{u})=B(e,\widetilde{e})+R(\widetilde{u}_h),
\end{equation*}
with $R(\widetilde{u}_h)=F(\widetilde{u}_h)-B(u_h,\widetilde{u}_h) \neq 0$ if different discretizations are used for $u_h$ and $\widetilde{u}_h$, or if a stabilized discretization is used for one of the two problems.

\medskip

Denoting by $e^s$ and $\widetilde{e}^s$ the solutions in $V$ to the symmetrized residual equations,
\begin{align*}
\forall v \in V, \qquad B_{\rm symm}(e^s,v)&=F(v)-B(u_h,v), \\
\forall v \in V, \qquad B_{\rm symm}(\widetilde{e}^s,v)&=Q(v)-B(v,\widetilde{u}_h),
\end{align*}
we get that
\begin{equation*}
|Q(u)-Q(u_h)-R(\widetilde{u}_h)| = |B_{\rm symm}(e^s,\widetilde{e}^s)| \le \vertiii{e^s} \ \vertiii{\widetilde{e}^s}.
\end{equation*}
Classical estimates in energy norm (induced by the symmetric part of the operator) can next be used to bound $\vertiii{e^s}$ and $\vertiii{\widetilde{e}^s}$.

\subsection{Time-dependent and nonlinear problems}
\label{sec:non-lin}

For linear time-dependent problems, effective \textit{a posteriori} error estimation tools (both for the global error and for errors on quantities of interest) have been proposed. We refer to~\cite{BER05} for parabolic problems, to~\cite{AUB99,COM99,VER12,WAE12,WIB92} for transient elastodynamics, and to~\cite{LAD89,MAD99,ODE03,WAN16} for vibratory dynamics with error estimates on the eigenfrequencies. A large set of applications of \textit{a posteriori} error estimates for both parabolic and hyperbolic problems is also described in~\cite{ERI95}.

\medskip

For nonlinear problems, there are much fewer contributions than for linear problems. It is important to distinguish between nonlinear {\em time-dependent} and nonlinear {\em time-independent} problems:
\begin{itemize}
\item For the latter case, we mention~\cite{BAB82,LAR02} for the design of estimates for nonlinear elasticity problems (where $\nab u$ is considered small, as in linear elasticity, but where the constitutive relation between stress and strain is nonlinear), \cite{BRI98} for large strain elasticity (where $\nab u$ is not considered to be small), \cite{JOH92} for Hencky-type plasticity problems, \cite{CIR00,GAL96,GHO17,PER94,RAN99} for elastoplasticity, and~\cite{BEN12,COO00,ELB20,LOU03,WEI09} for contact-friction problems.
  \item For the former case, viscoplasticity problems have been considered in~\cite{FOU95,LAR03,PEL00,PIJ95}, Navier-Stokes equations have been addressed in~\cite{AIN00,SEG10,VER96} with residual methods, and nonlinear dynamics has been studied in~\cite{RAD99}. Other nonlinear contexts have also been investigated in~\cite{HUE00}. In most cases, techniques devised for linear problems or time-independent nonlinear problems are used at each time step, and the estimation is thus limited to spatial error. For such nonlinear time-dependent problems, error estimators based on the constitutive law residual have been introduced and take all error sources into account~\cite{LAD00,LAD01,LAD98,LAD99}. They refer to the concept of dissipation error which is constructed from dual convex potentials (in the Legendre-Fenchel sense) that describe the material behavior~\cite{REP99}. This concept also enables to get upper bounds for quantities of interest in complex nonlinear cases~\cite{LAD08,LAD12}.
\end{itemize}

\medskip

Let us also add that specific indicators for various error sources (space/time discretizations, linearization procedure and stopping criteria for iterative algorithms introduced to solve the nonlinear problems, \dots) have also been established, see~\cite{DAB20,ELA11,ERN13,PEL00,VOH13}.

% ---------------------------------------------------------------------
% ---------------------------------------------------------------------
\section{Conclusion}\label{section:conclusion}

We have presented a review on \textit{a posteriori} error estimation tools applied to linear elliptic problems solved with FEM. On the one hand, we have shown that inexpensive methods (such as explicit residual methods) may be sufficient to obtain an indication on the discretization error or to drive mesh adaptation. On the other hand, we have emphasized that more advanced error estimators, which require more expensive computations, are able to provide accurate information on the error value. We have also highlighted that dual analysis and equilibrium concepts are at the heart of all robust estimates that provide guaranteed error bounds, and that can be extended to general engineering problems. For the sake of brevity, we have often limited ourselves in terms of technical details. We refer the interested reader to reference textbooks on this topic, such as~\cite{AIN00,BAB01,LAD04,STE03,VER96,VER13}.

In addition to FEM, verification has now become a challenging issue in various innovating numerical strategies that have emerged during the last decades. In this context, some first tools have been developed to assess discretization errors in association with errors coming from other sources. These tools use concepts similar to those presented in this article. Among all \textit{a posteriori} error estimates that have appeared in these new contexts, we wish to cite those related to
\begin{itemize}
\item mixed FE discretizations~\cite{BRA95,VOH07};
\item domain decomposition techniques~\cite{BER02,PAR10,REY14};
\item virtual element methods~\cite{BEI15,BER17,CAN17};
\item solution to stochastic problems~\cite{BAB05,BUT11,CHA12,LAD06,MAC15,ODE05,SCA19};
\item fracture mechanics, possibly with enrichment techniques such as XFEM~\cite{GAL06,PAN10,ROD08,RUT06,STO98,STR00,XUA06};
\item isogeometric analysis~\cite{BAZ06,BUF16,DOR10,KLE15,KUM17,KUR14,THA19};
\item the use of surrogate mathematical models, leading to modeling errors when compared to a reference model~\cite{BAB96b,ODE02,TIR20};
\item the use of reduced order modeling by means of Reduced Basis methods~\cite{GRE05,HOA18,MAC01,POR85,PRU02,QUA11,ROV06,ROZ08}, Proper Generalized Decomposition~\cite{ALL18,AMM10,CHA16,CHA17,CHA19,LAD11,REI20} or alternative methods~\cite{AGG17,CHA15,EKR20,KER14};
\item multiscale modeling as in the Heterogeneous Multiscale Method (HMM)~\cite{ABD13,ABD09,JHU12}, the Variational Multiscale Method (VMM)~\cite{LAR07}, the Generalized Finite Element Method (GFEM)~\cite{STR06} or the Multiscale Finite Element Method (MsFEM)~\cite{CHA18,CHA21,CHU16,HEN14}.
\end{itemize}

\appendix
%\appendixpage
%\addappheadtotoc

\section{Appendix: Proof of the lower bound for CRE error estimates}\label{section:appendixA}

We have seen in Section~\ref{sec:cre} that the Constitutive Relation Error (CRE) estimate is an upper bound on the numerical error. More precisely, we have seen (see~\eqref{eq:upper} or~\eqref{eq:propertiesCRE2}) that
$$
\vertiii{ u - u_h }^2 \leq E_{\rm CRE}^2(u_h,\bp)
$$
for any statically admissible $\bp \in W$. We show here a converse inequality. In contrast to some other methods (see e.g.~\eqref{eq:LBres2} for the subdomain residual methods described in Section~\ref{section:locpbpatches}), obtaining such a lower bound is not straightforward. Establishing this lower bound is also enlightning because the proof below uses the specificities of the construction of the equilibrated flux field (in contrast to the above upper bound).

For the sake of simplicity, we consider the problem~\eqref{eq:refpbstrong} in dimension $d=2$, although our arguments carry over to the three-dimensional case. We assume that we use a mesh made of triangles, and that the mesh is regular in the sense that
\begin{itemize}
\item the number of triangles in $\Omega_i$ (the set of elements having $i$ as a vertex) is bounded independently of the mesh size $h$; 
\item in the union $\dis \cup_{K \in \mT_h} \cup_{i \in K}\Omega_i$, each triangle is accounted for a number of times which is bounded from above independently of $h$. 
\end{itemize}
We also assume that we use P1 finite elements, and denote $\{ \phi_i \}$ the basis functions. Our arguments presumably carry over to more general cases, at the possible price of additional work.

\begin{theorem}\label{theo:borneinf}
Under the above assumptions, let $u$ be the solution to~\eqref{eq:refpb} and $\widehat{\bq}_h \in [L^2(\Omega)]^2$ be the equilibrated flux field constructed with the hybrid-flux technique (see Section~\ref{section:hybridfluxapp}). We assume that the numerical solution $u_h\in H^1(\Omega)$ to~\eqref{eq:FEpb} is such that
\begin{equation}
\label{eq:hyp1}
\text{on any triangle $K$, \quad $\dive (\Aa \nab u_h)=0$}.
\tag{H1}
\end{equation}
Denoting $\bq_h = \Aa \nab u_h$, we also assume that
\begin{equation}
  \label{eq:hyp2}
  \begin{array}{c}
    \text{on any edge $\Gamma_{\alpha,\beta}$ separating the elements $K_\alpha$ and $K_\beta$}, \\
    \text{$(\bq_{h|K_{\alpha}}+\bq_{h|K_{\beta}})\cdot \bn$ is constant}
  \end{array}
\tag{H2}
\end{equation}
and that
\begin{equation}
  \label{eq:hyp3}
  \begin{array}{c}
    \text{on any edge $\Gamma_{\alpha,\beta}$ separating the elements $K_\alpha$ and $K_\beta$}, \\
    \text{$(\bq_{h|K_{\alpha}}-\bq_{h|K_{\beta}})\cdot \bn$ is constant}.
  \end{array}
\tag{H3}
\end{equation}
We furthermore assume that
\begin{equation}
\label{eq:hyp4}
\text{on any triangle $K$, \quad $f$ is constant}.
\tag{H4}
\end{equation}
Then, there is a constant $C$ independent of $h$ such that
\begin{equation}\label{eq:borneinf}
E_{\rm CRE}^2(u_h,\widehat{\bq}_h) \leq C \, \vertiii{ u - u_h }^2.
\end{equation}
\end{theorem}
Of course, Assumptions~\eqref{eq:hyp2} and~\eqref{eq:hyp3} are equivalent to the fact that, on each edge $\Gamma_{\alpha,\beta}$ separating $K_\alpha$ and $K_\beta$, the normal fluxes $\bq_{h|K_{\alpha}} \cdot \bn$ and $\bq_{h|K_{\beta}} \cdot \bn$ are constant. We note that assumptions~\eqref{eq:hyp1}, \eqref{eq:hyp2} and~\eqref{eq:hyp3} are satisfied as soon as $\Aa$ is constant on each triangle $K$ (recall that we use P1 elements). Our proof of Theorem~\ref{theo:borneinf} essentially follows the arguments of~\cite{LAD83}. In Corollary~\ref{coro:erc_autre} below, we relax Assumption~\eqref{eq:hyp4}.

\medskip

\begin{proof}
The proof of Theorem~\ref{theo:borneinf} falls in 6 steps.

\medskip

\noindent \textbf{Step 1.} We follow the notations introduced in Section~\ref{section:hybridfluxapp}. Consider the local linear system~\eqref{eq:helas3} associated with an interior node $i$ and define the quantity
\begin{equation*}%\label{eq:def_d}
\widehat{d}_{\alpha,\beta}(i) = \widehat{b}_{\alpha,\beta}(i) - b^m_{\alpha,\beta}(i) 
\end{equation*}
where $\alpha$ and $\beta$ are two adjacent triangles in $\Omega_i$ and $\widehat{b}_{\alpha,\beta}(i)$ (resp. $b^m_{\alpha,\beta}(i)$) is defined by~\eqref{eq:def_b_hat} (resp.~\eqref{eq:def_b_m}).

Let $N_i$ denote the number of triangles in $\Omega_i$. We infer from~\eqref{eq:helas3} (recall that $\phi_i$ are the piecewise affine basis functions) and~\eqref{eq:hyp1} that
\begin{align*}
& \widehat{d}_{1,2}(i)-\widehat{d}_{N_i,1}(i) 
\\
&= 
\left( \widehat{b}_{1,2}(i) - b^m_{1,2}(i) \right) - \left( \widehat{b}_{N_i,1}(i) - b^m_{N_i,1}(i) \right) 
\\
&= 
Q_i^{K_1} - \left(b^m_{1,2}(i)-b^m_{N_i,1}(i)\right) 
\\
&= 
\int_{K_1} (\bq_h \cdot \nab \phi_i - f\phi_i)
\\
& \qquad - \frac{1}{2} \left( \int_{\Gamma_{1,2}} \phi_i \, (\bq_{h|K_1}+\bq_{h|K_2}) \cdot \bn_{K_1}- \int_{\Gamma_{N_i,1}} \phi_i \, (\bq_{h|K_{N_i}}+\bq_{h|K_1})\cdot \bn_{K_{N_i}} \right)
\\
&= \int_{\Gamma_{1,2}}\phi_i \, \bq_{h|K_1}\cdot \bn_{K_1}+\int_{\Gamma_{N_i,1}}\phi_i \, \bq_{h|K_1}\cdot \bn_{K_1}-\int_{K_1}f\phi_i
\\
& \qquad - \frac{1}{2} \left( \int_{\Gamma_{1,2}} \phi_i \, (\bq_{h|K_1}+\bq_{h|K_2})\cdot \bn_{K_1}-\int_{\Gamma_{N_i,1}} \phi_i \, (\bq_{h|K_{N_i}}+\bq_{h|K_1})\cdot \bn_{K_{N_i}} \right)
\\
&= -\int_{K_1}f\phi_i
\\
& \qquad + \frac{1}{2}\int_{\Gamma_{1,2}}\phi_i \, (\bq_{h|K_1}-\bq_{h|K_2})\cdot \bn_{K_1}+\frac{1}{2}\int_{\Gamma_{N_i,1}}\phi_i \, (\bq_{h|K_{N_i}}-\bq_{h|K_1})\cdot \bn_{K_{N_i}}.
\end{align*}
Similarly, we compute that
\begin{multline*}
  \widehat{d}_{2,3}(i)-\widehat{d}_{1,2}(i) = -\int_{K_2} f\phi_i
  \\
  + \frac{1}{2}\int_{\Gamma_{2,3}}\phi_i \, (\bq_{h|K_2}-\bq_{h|K_3})\cdot \bn_{K_2}+\frac{1}{2}\int_{\Gamma_{1,2}}\phi_i \, (\bq_{h|K_1}-\bq_{h|K_2})\cdot \bn_{K_1},
\end{multline*}
and likewise for any $\widehat{d}_{\alpha,\alpha+1}(i)-\widehat{d}_{\alpha-1,\alpha}(i)$ (with $3 \leq \alpha \leq N_i-1$), until
\begin{multline*}
  \widehat{d}_{N_i,1}(i)-\widehat{d}_{N_i-1,N_i}(i) = -\int_{K_{N_i}}f\phi_i
  \\
  + \frac{1}{2}\int_{\Gamma_{N_i,1}}\phi_i \, (\bq_{h|K_{N_i}}-\bq_{h|K_1})\cdot \bn_{K_{N_i}}+\frac{1}{2}\int_{\Gamma_{N_i-1,N_i}}\phi_i \, (\bq_{h|K_{N_i-1}}-\bq_{h|K_{N_i}})\cdot \bn_{K_{N_i-1}}.
\end{multline*}
We introduce
\begin{equation}\label{eq:def_j}
j_{\alpha,\beta}=\frac{1}{2}(\bq_{h|K_{\alpha}}-\bq_{h|K_{\beta}})\cdot \bn_{K_{\alpha}}, \qquad J_{\alpha,\beta}(i)=\int_{\Gamma_{\alpha,\beta}}\phi_i \, j_{\alpha,\beta},
\end{equation}
and remark that $j_{\beta,\alpha}=j_{\alpha,\beta}$ and $J_{\beta,\alpha}(i) = J_{\alpha,\beta}(i)$. Using these notations, we write the local linear system as
\begin{equation}\label{eq:localsystemd}
\begin{aligned}
\widehat{d}_{1,2}(i)-\widehat{d}_{N_i,1}(i) &= -\int_{K_1}f\phi_i + J_{1,2}(i) + J_{N_i,1}(i), \\
\widehat{d}_{2,3}(i)-\widehat{d}_{1,2}(i) &=-\int_{K_2}f\phi_i + J_{2,3}(i) + J_{1,2}(i), \\
\vdots &= \vdots \\
\widehat{d}_{N_i,1}(i)-\widehat{d}_{N_i-1,N_i}(i) &=-\int_{K_{N_i}}f\phi_i + J_{N_i,1}(i) + J_{N_i-1,N_i}(i).
\end{aligned}
\end{equation}
Just like the system~\eqref{eq:helas3}, the system~\eqref{eq:localsystemd} has an infinity of solutions. As explained in Remark~\ref{rem:underdetermined}, the one minimizing $\dis \sum_{\alpha,\beta} \frac{(\widehat{d}_{\alpha,\beta}(i))^2}{|\Gamma_{\alpha,\beta}|^2}$ is chosen in practice. The constrained minimization leads to an explicit solution of the form (see Remark~\ref{rem:underdetermined} and~\cite{LAD04})
$$
\widehat{d}_{N_i,1}(i) = s, \qquad \widehat{d}_{1,2}(i) = s + \overline{Q}_i^{K_1}, \qquad \widehat{d}_{2,3}(i) = s + \overline{Q}_i^{K_1} + \overline{Q}_i^{K_2}, \qquad \dots
$$
with $\dis \overline{Q}_i^{K_\alpha} = -\int_{K_\alpha} f \phi_i + J_{\alpha,\alpha+1}(i) + J_{\alpha-1,\alpha}(i)$ and $\dis s = - \frac{\sum_{\alpha,\beta} |\Gamma_{\alpha,\beta}|^{-2} \, \overline{Q}_i^{K_\beta}}{\sum_{\alpha,\beta} |\Gamma_{\alpha,\beta}|^{-2}}$.

Using the regularity of the mesh, we deduce that there exists $C$ independent of $h$ and $i$ such that $\dis |s| \leq C \sum_{n \in \Omega_i} \left| \overline{Q}_i^{K_n} \right|$, which implies that
$$
\left| \widehat{d}_{\alpha,\beta}(i) \right| \leq C \sum_{n \in \Omega_i} \left| \overline{Q}_i^{K_n} \right|.
$$
We thus get that there exists a constant $C$ independent of $h$, $i$, $\alpha$ and $\beta$ such that
\begin{equation}\label{eq:fin1}
\left| \widehat{d}_{\alpha,\beta}(i) \right| \leq C\left[ \sum_{n \in \Omega_i} \left| \int_{K_n} f \phi_i \right| + \sum_{n,m \in \Omega_i} \left| J_{n,m}(i) \right| \right].
\end{equation}

\bigskip

\noindent \textbf{Step 2.} On the edge linking vertices $i_1$ and $i_2$ (at the interface between triangles $K_\alpha$ and $K_\beta$), we wish to find $c_{i_1}$ and $c_{i_2}$ such that the function
$$
\sigma_{\Gamma_{\alpha,\beta},K_\alpha} \, \widehat{g}_{\Gamma_{\alpha,\beta}} = c_{i_1} \, \phi_{i_1} + c_{i_2} \, \phi_{i_2}
$$
satisfies
\begin{equation}
  \label{eq:helas9}
\int_{\Gamma_{\alpha,\beta}} \phi_{i_1} \, \sigma_{\Gamma_{\alpha,\beta},K_\alpha} \, \widehat{g}_{\Gamma_{\alpha,\beta}} = \widehat{b}_{\alpha,\beta}(i_1), 
\quad
\int_{\Gamma_{\alpha,\beta}} \phi_{i_2} \, \sigma_{\Gamma_{\alpha,\beta},K_\alpha} \, \widehat{g}_{\Gamma_{\alpha,\beta}} = \widehat{b}_{\alpha,\beta}(i_2).
\end{equation}
Introduce
\begin{align}
\widehat{G}_{i_1,i_2}
&=
\sigma_{\Gamma_{\alpha,\beta},K_\alpha} \, \widehat{g}_{\Gamma_{\alpha,\beta}} - \frac{1}{2} \left(\bq_{h|K_\alpha}+\bq_{h|K_\beta}\right) \cdot \bn_{K_\alpha}
\nonumber
\\
&=
\left( c_{i_1}-q_h^{\alpha \beta} \right)\phi_{i_1} + \left( c_{i_2}-q_h^{\alpha \beta} \right)\phi_{i_2}
\label{eq:def_Ghat}
\end{align}
where $\dis q_h^{\alpha \beta} = \frac{1}{2}\left(\bq_{h|K_\alpha}+\bq_{h|K_\beta}\right)\cdot \bn_{K_\alpha}$ is a constant (Assumption~\eqref{eq:hyp2}) and we recall that $\phi_{i_1} + \phi_{i_2} = 1$ on the edge. It is straightforward to observe that~\eqref{eq:helas9} is equivalent to
\begin{equation}\label{eq:sys_G}
\begin{cases}
\dis \int_{\Gamma_{\alpha,\beta}} \phi_{i_1} \, \widehat{G}_{i_1,i_2} = \widehat{b}_{\alpha,\beta}(i_1) - b^m_{\alpha,\beta}(i_1) = \widehat{d}_{\alpha,\beta}(i_1),
\\ \noalign{\vskip 3pt}
\dis \int_{\Gamma_{\alpha,\beta}} \phi_{i_2} \, \widehat{G}_{i_1,i_2} = \widehat{b}_{\alpha,\beta}(i_2) - b^m_{\alpha,\beta}(i_2) = \widehat{d}_{\alpha,\beta}(i_2).
\end{cases}
\end{equation}
Collecting~\eqref{eq:def_Ghat} and~\eqref{eq:sys_G}, we obtain
\begin{equation}
  \label{eq:helas10}
\left( \begin{array}{cc}
\dis \int_{\Gamma_{\alpha,\beta}} \phi_{i_1}^2 & \dis \int_{\Gamma_{\alpha,\beta}} \phi_{i_1} \phi_{i_2}
\\
\dis \int_{\Gamma_{\alpha,\beta}} \phi_{i_1} \phi_{i_2} & \dis \int_{\Gamma_{\alpha,\beta}} \phi_{i_2}^2
\end{array} \right)
\left( \begin{array}{c}
c_{i_1} - q_h^{\alpha\beta} 
\\
c_{i_2} - q_h^{\alpha\beta} 
\end{array} \right)
=
\left( \begin{array}{c}
\widehat{d}_{\alpha,\beta}(i_1)
\\
\widehat{d}_{\alpha,\beta}(i_2)
\end{array} \right),
\end{equation}
a system the explicit solution of which is
\begin{equation}\label{eq:yes4}
\hspace{-1mm} \left( \begin{array}{c}
c_{i_1} - q_h^{\alpha\beta} 
\\
c_{i_2} - q_h^{\alpha\beta} 
\end{array} \right) = \frac{1}{D}
\left( \begin{array}{cc}
\dis \int_{\Gamma_{\alpha,\beta}} \phi_{i_2}^2 & \dis -\int_{\Gamma_{\alpha,\beta}} \phi_{i_1} \phi_{i_2}
\\
\dis -\int_{\Gamma_{\alpha,\beta}} \phi_{i_1} \phi_{i_1} & \dis \int_{\Gamma_{\alpha,\beta}} \phi_{i_2}^2
\end{array} \right)
\left( \begin{array}{c}
\widehat{d}_{\alpha,\beta}(i_1)
\\
\widehat{d}_{\alpha,\beta}(i_2)
\end{array} \right),
\end{equation}
with $D$ the determinant of the $2 \times 2$ matrix in~\eqref{eq:helas10}. In dimension $d=2$, the terms of the matrix in~\eqref{eq:helas10} are of the order of $O(h)$, hence $D$ is of the order of $O(h^2)$ (in dimension $d$, the matrix appearing in~\eqref{eq:helas10} is of size $d \times d$, and each of its term is of the order of $O(h^{d-1})$; its determinant is thus of the order of $O(h^{d(d-1)})$, and its cofactors, appearing in~\eqref{eq:yes4}, are of the order of $O(h^{(d-1)^2})$). Introducing these scalings, we get from~\eqref{eq:yes4} that
\begin{equation*}
\left|c_{i_1} - q_h^{\alpha\beta} \right| + \left| c_{i_2} - q_h^{\alpha\beta}\right| \leq \frac{C}{h^{d-1}} \Big( \left| \widehat{d}_{\alpha,\beta}(i_1) \right| + \left| \widehat{d}_{\alpha,\beta}(i_2) \right| \Big),
\end{equation*}
and therefore
\begin{equation}\label{eq:fin2}
\int_{\Gamma_{\alpha,\beta}} \left| \widehat{G}_{i_1,i_2} \right|^2 
\leq 
\frac{C}{h^{d-1}} \Big( \left| \widehat{d}_{\alpha,\beta}(i_1) \right| + \left| \widehat{d}_{\alpha,\beta}(i_2) \right| \Big)^2. 
\end{equation}

\bigskip

\noindent \textbf{Step 3.} We introduce $\boldsymbol{\Sigma}_h = \widehat{\bq}_h-\bq_h$, where $\widehat{\bq}_h$ is the equilibrated field constructed in Section~\ref{section:hybridfluxapp} and $\bq_h = \Aa \nab u_h$. We recall that, for any triangle $K_\alpha$,
\begin{equation*}
-\dive \widehat{\bq}_h = f \ \ \text{in $K_\alpha$}, \qquad
\widehat{\bq}_h \cdot \bn_{K_\alpha} = \sigma_{\Gamma_{\alpha,\beta},K_\alpha} \, \widehat{g}_{\Gamma_{\alpha,\beta}} \ \ \text{on the edges $\Gamma_{\alpha,\beta} \subset \partial K_\alpha$}
\end{equation*}
where $\bn_{K_\alpha}$ is the outgoing normal vector to $K_\alpha$. Among all equilibrated flux fields which verify the previous condition, we select the one that minimizes the complementary energy (or CRE functional) over $K_\alpha$ so that, by duality, it derives from a primal field $\widehat{u}_h$. We thus choose $\widehat{\bq}_h = \Aa \nab \widehat{u}_h$ where $\widehat{u}_h$ is a solution to
\begin{equation*}
-\dive \Aa \nab \widehat{u}_h = f \ \text{in $K_\alpha$}, \quad
\Aa \nab \widehat{u}_h \cdot \bn_{K_\alpha} = \sigma_{\Gamma_{\alpha,\beta},K_\alpha} \, \widehat{g}_{\Gamma_{\alpha,\beta}} \ \text{on the edges $\Gamma_{\alpha,\beta} \subset \partial K_\alpha$},
\end{equation*}
and is of course uniquely defined up to an additive constant. We choose it such that
\begin{equation}\label{eq:choix}
\int_{K_\alpha} \widehat{u}_h = \int_{K_\alpha} u_h.
\end{equation} 
Using Assumption~\eqref{eq:hyp1}, we see that $\boldsymbol{\Sigma}_h$ verifies
\begin{equation}\label{eq:yes1}
-\dive \boldsymbol{\Sigma}_h = f \quad \text{in $K_\alpha$}.
\end{equation}
Furthermore, on the edge $\Gamma_{\alpha,\beta} \subset \partial K_\alpha$ linking vertices $i_1$ and $i_2$, we have
\begin{equation}\label{eq:yes2}
\boldsymbol{\Sigma}_h \cdot \bn_{K_\alpha} = \sigma_{\Gamma_{\alpha,\beta},K_\alpha} \, \widehat{g}_{\Gamma_{\alpha,\beta}} - \bq_{h|K_\alpha}\cdot \bn_{K_\alpha} = \widehat{G}_{i_1,i_2} - j_{\alpha,\beta},
\end{equation}
where $\widehat{G}_{i_1,i_2}$ and $j_{\alpha,\beta}$ are defined in~\eqref{eq:def_Ghat} and~\eqref{eq:def_j}, respectively.

Using~\eqref{eq:yes1}, we get
\begin{align*}
  \int_{K_\alpha} \boldsymbol{\Sigma}_h^T \Aa^{-1}\boldsymbol{\Sigma}_h &= \int_{K_\alpha} \boldsymbol{\Sigma}_h \cdot \nab(\widehat{u}_h-u_h) \\
  &= \int_{\partial K_\alpha}\boldsymbol{\Sigma}_h \cdot \bn_{K_\alpha}(\widehat{u}_h-u_h)-\int_{K_\alpha}(\widehat{u}_h-u_h) \dive \boldsymbol{\Sigma}_h \\
&\le \|\boldsymbol{\Sigma}_h \cdot \bn_{K_\alpha}\|_{0,\partial K_\alpha} \|\widehat{u}_h-u_h\|_{0,\partial K_\alpha}+\|f\|_{0,K_\alpha} \|\widehat{u}_h-u_h\|_{0,K_\alpha}.
\end{align*}
Using~\eqref{eq:choix} and scaling arguments similar to those employed in~\cite[Sec.~4.2, eqs.~(4.5) and~(4.3)]{LEB13}, we deduce from the previous inequality that
\begin{align*}
  & \int_{K_\alpha} \boldsymbol{\Sigma}_h\Aa^{-1}\boldsymbol{\Sigma}_h
  \\
  & \le C\sqrt{h}\|\nab (\widehat{u}_h - u_h) \|_{0,K_\alpha} \|\boldsymbol{\Sigma}_h \cdot \bn_{K_\alpha}\|_{0,\partial K_\alpha} + Ch \|\nab (\widehat{u}_h - u_h) \|_{0,K_\alpha} \|f\|_{0,K_\alpha} \\
&\le C\sqrt{h}\|\Aa^{-1/2}\boldsymbol{\Sigma}_h\|_{0,K_\alpha} \|\boldsymbol{\Sigma}_h \cdot \bn_{K_\alpha}\|_{0,\partial K_\alpha} + Ch\|\Aa^{-1/2}\boldsymbol{\Sigma}_h\|_{0,K_\alpha} \|f\|_{0,K_\alpha}.
\end{align*}
It should be noticed that the previous scalings are independent of the dimension $d$. We infer from the above bound that
\begin{equation*}
\int_{K_\alpha} \boldsymbol{\Sigma}_h\Aa^{-1}\boldsymbol{\Sigma}_h \le Ch\|\boldsymbol{\Sigma}_h \cdot \bn_{K_\alpha}\|^2_{0,\partial K_\alpha} + Ch^2\|f\|^2_{0,K_\alpha}.
\end{equation*}
We now use~\eqref{eq:yes2} and obtain
\begin{equation*}
\int_{K_\alpha} \boldsymbol{\Sigma}_h\Aa^{-1}\boldsymbol{\Sigma}_h \le Ch \sum_{\Gamma_{\alpha,\beta} \subset \partial K_\alpha}\int_{\Gamma_{\alpha,\beta}} \left( \left| \widehat{G}_{i_1,i_2} \right|^2 + \left| j_{\alpha,\beta} \right|^2 \right) + Ch^2\|f\|^2_{0,K_\alpha}.
\end{equation*}
Using~\eqref{eq:fin2}, we deduce
\begin{multline}\label{eq:yes5}
\int_{K_\alpha} \boldsymbol{\Sigma}_h\Aa^{-1}\boldsymbol{\Sigma}_h \le 
Ch^{2-d}\sum_{\Gamma_{\alpha,\beta} \subset \partial K_\alpha}
\Big( \left| \widehat{d}_{\alpha,\beta}(i_1) \right| + \left| \widehat{d}_{\alpha,\beta}(i_2) \right| \Big)^2
\\ + Ch\sum_{\Gamma_{\alpha,\beta} \subset \partial K_\alpha} \int_{\Gamma_{\alpha,\beta}}\left| j_{\alpha,\beta} \right|^2
+Ch^2\|f\|^2_{0,K_\alpha}.
\end{multline}
In order to bound the first term above, we are going to use~\eqref{eq:fin1}. We recast that estimate as follows: for any vertex $i$ and any triangles $K_\alpha$ and $K_\beta$ in $\Omega_i$, we have
\begin{align*}
  & \left| \widehat{d}_{\alpha,\beta}(i) \right|
  \\
  &\leq C\left[ \sum_{n \in \Omega_i} \left| \int_{K_n} f \phi_i \right| + \sum_{n,m \in \Omega_i} \left| J_{n,m}(i) \right| \right] \\
&\le C \sum_{n \in \Omega_i} \| f \|_{0,K_n} \| \phi_i \|_{0,K_n} + C \sum_{n,m \in \Omega_i} \| j_{n,m} \|_{0,\Gamma_{n,m}} \| \phi_i \|_{0,\Gamma_{n,m}} \\
&\le C \sqrt{\sum_{n \in \Omega_i} \| f \|^2_{0,K_n}} \ \sqrt{\sum_{n \in \Omega_i} \| \phi_i \|^2_{0,K_n}} 
+C \sqrt{\sum_{n,m \in \Omega_i} \| j_{n,m} \|^2_{0,\Gamma_{n,m}}} \ \sqrt{\sum_{n,m \in \Omega_i} \| \phi_i \|^2_{0,\Gamma_{n,m}} } \\
&\le C h^{d/2} \| f \|_{0,\Omega_i} + C h^{(d-1)/2} \sqrt{\sum_{n,m \in \Omega_i} \| j_{n,m} \|^2_{0,\Gamma_{n,m}}}.
\end{align*}
Inserting this estimate in~\eqref{eq:yes5} yields
\begin{equation}\label{eq:fin3}
\int_{K_\alpha} \boldsymbol{\Sigma}_h\Aa^{-1}\boldsymbol{\Sigma}_h \le Ch^2\sum_{i \in K_\alpha}\| f \|^2_{0,\Omega_i} + Ch\sum_{i \in K_\alpha}\sum_{n,m \in \Omega_i}\| j_{n,m} \|^2_{0,\Gamma_{n,m}}.
\end{equation}

\bigskip

\noindent \textbf{Step 4.} In this step, we bound $\| f \|_{0,K}$ from above by the error $u-u_h$. Let $K$ be a triangle element with vertices $i$, $j$ and $k$. We consider the test function $\widetilde{v} \in V$ defined as $\widetilde{v}(\bx) = \phi_i(\bx) \, \phi_j(\bx) \, \phi_k(\bx)$ over $K$ and $\widetilde{v}(\bx) = 0$ elsewhere (bubble function). Inserting this test function in~\eqref{eq:refpb}, we get
\begin{equation*}
\int_K \Aa \nab u \cdot \nab \widetilde{v} = f_K \int_K \phi_i \, \phi_j \, \phi_k
\end{equation*}
with $f_K$ the constant value of $f$ in $K$ (see Assumption~\eqref{eq:hyp4}). Using an integration by parts, we get
\begin{equation*}
\int_K \Aa \nab u_h \cdot \nab \widetilde{v} =\int_{\partial K} \widetilde{v} \, \Aa \nab u_h \cdot \bn_K - \int_K \widetilde{v} \, \dive (\Aa \nab u_h) = 0
\end{equation*}
where we have used that $\widetilde{v}_{|\partial K}=0$ (to cancel the first term) and Assumption~\eqref{eq:hyp1} (to cancel the second term). We deduce that
\begin{equation*}
f_K\int_K \phi_i \, \phi_j \, \phi_k = \int_K \Aa \nab (u-u_h) \cdot \nab \widetilde{v},
\end{equation*}
which implies that there is a constant $C>0$ independent of $h$ such that
\begin{equation}
\label{eq:yes3}
C \, h^d \, |f_K| \leq \| \Aa^{1/2} \nab \widetilde{v} \|_{0,K} \| \Aa^{1/2} \nab (u-u_h) \|_{0,K}.
\end{equation}
Since the three terms in $\nab \widetilde{v}$ are of the same order of magnitude and $\nab \phi_i$ is constant in $K$, we get
\begin{equation*}
  \| \Aa^{1/2} \nab \widetilde{v} \|^2_{0,K}
  \leq
  C \int_K |\nab \phi_i|^2 \, \phi_j^2 \, \phi_k^2
  \leq
  C h^{-2} \int_K \phi_j^2 \, \phi_k^2
  \leq
  C h^{d-2}.
\end{equation*}
We deduce from~\eqref{eq:yes3} that
\begin{equation*}
| f_K| \le C h^{-d/2-1}\| \Aa^{1/2} \nab (u-u_h) \|_{0,K}
\end{equation*}
and thus
\begin{equation}\label{eq:fin4}
\| f \|^2_{0,K} \leq C h^d f_K^2 \leq C h^{-2}\| \Aa^{1/2} \nab (u-u_h) \|^2_{0,K}.
\end{equation}

\bigskip

\noindent \textbf{Step 5.} In this step, we bound $\| j_{\alpha,\beta} \|_{0,\Gamma_{\alpha,\beta}}$ from above by the error $u-u_h$. We consider the edge $\Gamma_{\alpha,\beta}$, separating the triangles $K_\alpha$ and $K_\beta$, and denote $i_1$ and $i_2$ its vertices. We choose the test function $\widetilde{v}$ defined as $\widetilde{v}(\bx) = \phi_{i_1}(\bx) \, \phi_{i_2}(\bx)$ over $K_\alpha \cup K_\beta$ and $\widetilde{v}(\bx) = 0$ elsewhere. Inserting this test function in~\eqref{eq:refpb}, we get
\begin{equation*}
\int_{K_\alpha \cup K_\beta} \Aa \nab u \cdot \nab \widetilde{v} = \int_{K_\alpha \cup K_\beta} f \, \phi_{i_1} \, \phi_{i_2}. 
\end{equation*}
Considering now the approximate solution $u_h$, we compute, using Assumption~\eqref{eq:hyp1}, that
\begin{align*}
  \int_{K_\alpha \cup K_\beta} \Aa \nab u_h \cdot \nab \widetilde{v}
  &=
  \int_{K_\alpha} \Aa \nab u_h \cdot \nab \widetilde{v} + \int_{K_\beta} \Aa \nab u_h \cdot \nab \widetilde{v}
  \\
  &=
  \int_{\partial K_\alpha} \Aa \nab u_{h|K_\alpha} \cdot\bn_{K_\alpha} \, \widetilde{v} - 
  \int_{K_\alpha} \widetilde{v} \, \dive (\Aa \nab u_h)
  \\
  & \qquad \qquad + \int_{\partial K_\beta} \Aa \nab u_{h|K_\beta}\cdot\bn_{K_\beta} \, \widetilde{v} - 
  \int_{K_\beta}  \widetilde{v} \, \dive (\Aa \nab u_h)
  \\
  &=
  \int_{\Gamma_{\alpha,\beta}} \bq_{h|K_\alpha}\cdot\bn_{K_\alpha} \, \widetilde{v} + 
  \int_{\Gamma_{\alpha,\beta}} \bq_{h|K_\beta}\cdot\bn_{K_\beta} \, \widetilde{v}
  \\
  &=
  2 \int_{\Gamma_{\alpha,\beta}} j_{\alpha,\beta} \, \widetilde{v}.
\end{align*}
Consequently, 
\begin{equation*}
\int_{K_\alpha \cup K_\beta} f \, \phi_{i_1} \, \phi_{i_2} - 2 \int_{\Gamma_{\alpha,\beta}} j_{\alpha,\beta} \, \widetilde{v} = 
\int_{K_\alpha \cup K_\beta} \Aa \nab (u-u_h) \cdot \nab \widetilde{v}.
\end{equation*}
The term $j_{\alpha,\beta}$ being constant on the edge $\Gamma_{\alpha,\beta}$ (see Assumption~\eqref{eq:hyp3}), there is a constant $C>0$ independent of $h$ such that
\begin{multline}
\label{eq:yes3_bis}
C h^{d-1} | j_{\alpha,\beta} | \leq \| \Aa^{1/2} \nab \widetilde{v} \|_{0,K_\alpha \cup K_\beta} \| \Aa^{1/2} \nab (u-u_h) \|_{0,K_\alpha \cup K_\beta} \\ + \| f \|_{0,K_\alpha \cup K_\beta} \| \phi_{i_1} \, \phi_{i_2} \|_{0,K_\alpha \cup K_\beta}.
\end{multline}
Since the two terms in $\nab \widetilde{v}$ are of the same order of magnitude and $\nab \phi_{i_1}$ is constant over each triangle, we get
\begin{equation*}
\| \Aa^{1/2} \nab \widetilde{v} \|^2_{0,K_\alpha \cup K_\beta} \leq
C \int_{K_\alpha\cup K_\beta} |\nab \phi_{i_1}|^2 \, \phi_{i_2}^2 \leq
C h^{-2} \int_{K_\alpha\cup K_\beta} \phi_{i_2}^2 \leq C h^{d-2}.
\end{equation*}
Besides, we have $\| \phi_{i_1} \, \phi_{i_2} \|^2_{0,K_\alpha \cup K_\beta}\leq C h^d$. We hence deduce from~\eqref{eq:yes3_bis} that
\begin{equation} \label{eq:fin5_pre}
h^{d-1} | j_{\alpha,\beta} | \leq C h^{d/2-1} \| \Aa^{1/2} \nab (u-u_h) \|_{0,K_\alpha \cup K_\beta} + C h^{d/2} \| f \|_{0,K_\alpha \cup K_\beta},
\end{equation}
and hence, using~\eqref{eq:fin4},
\begin{align*}
  | j_{\alpha,\beta} |
  &\leq
  C h^{-d/2} \| \Aa^{1/2} \nab (u-u_h) \|_{0,K_\alpha \cup K_\beta} + C h^{1-d/2} \| f \|_{0,K_\alpha \cup K_\beta}
  \\
  &\leq
  C h^{-d/2} \| \Aa^{1/2} \nab (u-u_h) \|_{0,K_\alpha \cup K_\beta}.
\end{align*}
We therefore obtain
\begin{equation}
\label{eq:fin5}
\| j_{\alpha,\beta} \|^2_{0,\Gamma_{\alpha,\beta}} \leq C h^{d-1} | j_{\alpha,\beta} |^2 \leq C h^{-1} \| \Aa^{1/2} \nab (u-u_h) \|^2_{0,K_\alpha \cup K_\beta}.
\end{equation}

\bigskip

\noindent \textbf{Step 6.} Collecting~\eqref{eq:fin3}, \eqref{eq:fin4} and~\eqref{eq:fin5}, we obtain
\begin{align*}
  & \int_{K_\alpha} \boldsymbol{\Sigma}_h\Aa^{-1}\boldsymbol{\Sigma}_h
  \\
  & \le
C\sum_{i \in K_\alpha}\| \Aa^{1/2} \nab (u-u_h) \|^2_{0,\Omega_i}
+ C\sum_{i \in K_\alpha}\sum_{n,m \in \Omega_i}
\| \Aa^{1/2} \nab (u-u_h) \|^2_{0,K_n \cup K_m} \\
&\le C\sum_{i \in K_\alpha}\| \Aa^{1/2} \nab (u-u_h) \|^2_{0,\Omega_i}.
\end{align*}
Using the regularity of the mesh, we deduce that
\begin{equation*}
E^2_{\rm CRE}(u_h,\widehat{\bq}_h) = \int_{\Omega} \boldsymbol{\Sigma}_h\Aa^{-1}\boldsymbol{\Sigma}_h \le C\| \Aa^{1/2} \nab (u-u_h) \|^2_{0,\Omega}.
\end{equation*}
This concludes the proof of Theorem~\ref{theo:borneinf}.
\end{proof}

\bigskip

\begin{corollary}\label{coro:erc_autre}
We work under the assumptions of Theorem~\ref{theo:borneinf}, except Assumption~\eqref{eq:hyp4} (i.e. $f$ constant in each triangle) which is replaced by $f \in H^1(\Omega)$. We then have
\begin{equation}\label{eq:erc_autre_bis}
E_{\rm CRE}^2(u_h,\widehat{\bq}_h) \leq C \| \Aa^{1/2} \nab (u-u_h) \|^2_{0,\Omega} + C h^4 \| \nab f \|^2_{0,\Omega}.
\end{equation}
\end{corollary}

Since we use P1 elements, we note that the first term in the right-hand side of~\eqref{eq:erc_autre_bis} is of the order of $O(h^2)$. The second term, of the order of $O(h^4)$, is hence expected to be much smaller than the first term. 

\begin{proof}
The first three steps of the proof of Theorem~\ref{theo:borneinf} again hold, and we thus have (see~\eqref{eq:fin3})
\begin{equation}\label{eq:fin3_bis}
\int_{K_\alpha} \boldsymbol{\Sigma}_h\Aa^{-1}\boldsymbol{\Sigma}_h \le Ch^2\sum_{i \in K_\alpha}\| f \|^2_{0,\Omega_i} + Ch\sum_{i \in K_\alpha}\sum_{n,m \in \Omega_i}\| j_{n,m} \|^2_{0,\Gamma_{n,m}}.
\end{equation}

We now follow Step~4 of the proof of Theorem~\ref{theo:borneinf}. Let $K$ be a triangle element with vertices $i$, $j$ and $k$. We consider the test function $\widetilde{v} \in V$ defined as $\widetilde{v}(\bx) = f(\bx) \, \phi_i(\bx) \, \phi_j(\bx) \, \phi_k(\bx)$ over $K$ and $\widetilde{v}(\bx) = 0$ elsewhere. We have
\begin{equation*}
\int_K f^2 \, \phi_i \, \phi_j \, \phi_k = \int_K \Aa \nab (u-u_h) \cdot \nab \widetilde{v},
\end{equation*}
hence
$$
\int_K f^2 \, \phi_i \, \phi_j \, \phi_k
\leq
\| \Aa^{1/2} \nab \widetilde{v} \|_{0,K} \| \Aa^{1/2} \nab (u-u_h) \|_{0,K}.
$$
Since we have
$$
\| \Aa^{1/2} \nab \widetilde{v} \|^2_{0,K}
\leq
C \int_K | \nab \widetilde{v} |^2
\leq
C \left( h^{-2} \int_K f^2 + \int_K | \nab f |^2 \right),
$$
we obtain
\begin{equation}
\label{eq:helas11}
\int_K f^2 \, \phi_i \, \phi_j \, \phi_k
\leq
C \left( h^{-1} \| f \|_{0,K} + \| \nab f \|_{0,K} \right) \| \Aa^{1/2} \nab (u-u_h) \|_{0,K}.
\end{equation}
We now claim that there exists a constant $C$, independent of $h$ and $f$, such that
\begin{equation}
\label{eq:helas12}
\int_K f^2 \leq C \left( \int_K f^2 \, \phi_i \, \phi_j \, \phi_k + h^2 \int_K | \nab f |^2 \right).
\end{equation}
To prove~\eqref{eq:helas12}, we first argue on the unit triangle $\overline{K}$. Consider a function $\overline{\theta} \in L^1(\overline{K})$ such that $\overline{\theta} \geq 0$ on $\overline{K}$ and $\dis \int_{\overline{K}} \overline{\theta} > 0$. Then there exists $C$ such that, for any $g \in H^1(\overline{K})$, we have
\begin{equation}
\label{eq:helas13}
\int_{\overline{K}} g^2 \leq C \left( \int_{\overline{K}} g^2 \, \overline{\theta} + \int_{\overline{K}} | \nab g |^2 \right).
\end{equation}
This can be shown using a contradiction argument. Then, by scaling, we deduce~\eqref{eq:helas12} from~\eqref{eq:helas13}.

Collecting~\eqref{eq:helas12} and~\eqref{eq:helas11}, we obtain
$$
\| f \|^2_{0,K} \leq C h^2 \| \nab f \|^2_{0,K} + C \left( h^{-1} \| f \|_{0,K} + \| \nab f \|_{0,K} \right) \| \Aa^{1/2} \nab (u-u_h) \|_{0,K},
$$
from which we infer that
\begin{multline*}
\| f \|^2_{0,K} \leq C h^{-2} \| \Aa^{1/2} \nab (u-u_h) \|^2_{0,K} + C \| \nab f \|_{0,K} \| \Aa^{1/2} \nab (u-u_h) \|_{0,K} \\ + C h^2 \| \nab f \|^2_{0,K},
\end{multline*}
and hence (compare with~\eqref{eq:fin4})
\begin{equation}
\label{eq:helas14}
\| f \|^2_{0,K} \leq C h^{-2} \| \Aa^{1/2} \nab (u-u_h) \|^2_{0,K} + C h^2 \| \nab f \|^2_{0,K}.
\end{equation}

\medskip

Following Step~5 of the proof of Theorem~\ref{theo:borneinf}, we obtain~\eqref{eq:fin5_pre}, which we write as
$$
| j_{\alpha,\beta} | \leq C h^{-d/2} \| \Aa^{1/2} \nab (u-u_h) \|_{0,K_\alpha \cup K_\beta} + C h^{1-d/2} \| f \|_{0,K_\alpha \cup K_\beta}.
$$
Using~\eqref{eq:helas14}, we get
$$
| j_{\alpha,\beta} | \leq C h^{-d/2} \| \Aa^{1/2} \nab (u-u_h) \|_{0,K_\alpha \cup K_\beta} + C h^{2-d/2} \| \nab f \|_{0,K_\alpha \cup K_\beta},
$$
and thus (compare with~\eqref{eq:fin5})
\begin{align}
\| j_{\alpha,\beta} \|^2_{0,\Gamma_{\alpha,\beta}}
&\leq C h^{d-1} | j_{\alpha,\beta} |^2
\nonumber
\\
&\leq C h^{-1} \| \Aa^{1/2} \nab (u-u_h) \|^2_{0,K_\alpha \cup K_\beta} + C h^3 \| \nab f \|^2_{0,K_\alpha \cup K_\beta}.
\label{eq:fin5_bis}
\end{align}

\medskip

Collecting~\eqref{eq:fin3_bis}, \eqref{eq:helas14} and~\eqref{eq:fin5_bis}, we obtain, as in Step~6 of the proof of Theorem~\ref{theo:borneinf}, that
$$
\int_{K_\alpha} \boldsymbol{\Sigma}_h\Aa^{-1}\boldsymbol{\Sigma}_h
\leq
C \sum_{i \in K_\alpha}\| \Aa^{1/2} \nab (u-u_h) \|^2_{0,\Omega_i}
+ C h^4 \sum_{i \in K_\alpha} \| \nab f \|^2_{0,\Omega_i}.
$$
Using the regularity of the mesh, we deduce that
\begin{equation*}
E^2_{\rm CRE}(u_h,\widehat{\bq}_h) = \int_{\Omega} \boldsymbol{\Sigma}_h\Aa^{-1}\boldsymbol{\Sigma}_h \le C\| \Aa^{1/2} \nab (u-u_h) \|^2_{0,\Omega} + C h^4 \| \nab f \|^2_{0,\Omega}.
\end{equation*}
This concludes the proof of Corollary~\ref{coro:erc_autre}.
\end{proof}

\medskip

\noindent {\bf Acknowledgments.} The first author thanks Inria for enabling his two-year leave (2014--2016) in the MATHERIALS project team. The work of the second author is partially supported by ONR under grant N00014-20-1-2691 and by EOARD under grant FA8655-20-1-7043. The second author acknowledges the continuous support from these two agencies. The authors would also like to thank Claude Le Bris for the fruitful discussions we had on the topics covered by this article.

%\fl{biblio relue, ordre alpha verifi\'e}

%%%%%% Insert references here %%%%%%%%%%%%%%%%%%%%%

\end{document}